\documentclass[12pt]{amsart}
\usepackage{amssymb,latexsym,amsmath,comment}
\usepackage[mathscr]{eucal}
\usepackage{bbm,mathdots}
\usepackage{xcolor,tikz}
\usepackage{afterpage}

\numberwithin{equation}{section}
\numberwithin{figure}{section}

\def\naive{{na\"ive}}

\def\cf{cf.~}

\newcommand{\hsp}[1]{{\hbox{\hspace{#1}}}}
\newcommand{\bmath}[1]{{\bf \boldmath {#1} \unboldmath}}

\newcommand{\mystack}[2]{\ensuremath{ \substack{ \hbox{\tiny{${#1}$}} \\ \hbox{\tiny{${#2}$}} }} }

\def\a{\alpha}  
\def\b{\beta}  
\def\d{\delta}  
\def\e{\varepsilon}

\def\m{\mu}

\def\s{\sigma}

\def\w{\omega} 

\def\x{\xi}

\def\fa{\mathfrak{a}} 
  
\def\tAd{\mathrm{Ad}} 
 
\def\tAut{\mathrm{Aut}}

\def\fb{\mathfrak{b}} 
 \def\tbd{\mathrm{bd}}
\def\bC{\mathbb C}

\def\bc{\mathbf{c}}
 
\def\tcodim{\mathrm{codim}} \def\tcpt{\mathrm{cpt}}
 \def\sD{\mathscr{D}}

 \def\tdim{\mathrm{dim}}

\def\tEnd{\mathrm{End}} 

 \def\fe{\mathfrak{e}} 
 
\def\texp{\mathrm{exp}}

\def\ff{\mathfrak{f}}

\def\tFlag{\mathrm{Flag}}

\def\tGr{\mathrm{Gr}}
\def\fg{{\mathfrak{g}}}

\def\bH{\mathbb{H}}

\def\tHom{\mathrm{Hom}}
\def\fh{\mathfrak{h}} 
\def\bh{\mathbf{h}}
\def\bi{\mathbf{i}}

 \def\tIm{\mathrm{Im}}

\def\fk{\mathfrak{k}}

 \def\tker{\mathrm{ker}}

\def\fl{\mathfrak{l}} 
\def\tlim{\mathrm{lim}}

\def\bN{\mathbb N} \def\cN{\mathcal N}

 \def\tnilp{\mathrm{nilp}}

 \def\cO{\mathcal O}

\def\bP{\mathbb P}

\def\fp{\mathfrak{p}} 
 
\def\tprim{\mathrm{prim}}

\def\bQ{\mathbb Q}

\def\bR{\mathbb R}

\def\fr{\mathfrak{r}}

\def\trank{\mathrm{rank}}

\def\fs{\mathfrak{s}}
 
\def\tss{\mathrm{ss}}
\def\tSL{\mathrm{SL}} \def\tSO{\mathrm{SO}}
\def\tSp{\mathrm{Sp}} 
 
\def\tStab{\mathrm{Stab}}

 \def\tspan{\mathrm{span}}
\def\fsl{\mathfrak{sl}} \def\fso{\mathfrak{so}} 
\def\fsp{\mathfrak{sp}} \def\fsu{\mathfrak{su}}
 \def\sT{\mathscr{T}} 
 
\def\ft{\mathfrak{t}}

\def\sX{\mathscr{X}}

   \def\bZ{\mathbb Z}
 
\def\fz{\mathfrak{z}} 
 
\def\half{\tfrac{1}{2}}

\def\one{\mathbbm{1}}
\def\tand{\quad\hbox{and}\quad}
\def\del{\partial}

\def\sb{{\hbox{\tiny{$\bullet$}}}}

\def\inj{\hookrightarrow}

\def\op{\oplus}
\def\ot{\otimes}

\def\wtL{{\Lambda_\mathrm{wt}}}

\def\rtL{{\Lambda_\mathrm{rt}}}

\newcounter{numcnt}

\newcounter{cnt}

\newcounter{acnt}
\newenvironment{a_list}{ 
  \begin{list}{{(\alph{acnt})}}
   {\usecounter{acnt} \setlength{\itemsep}{3pt}
    \setlength{\leftmargin}{25pt} \setlength{\labelwidth}{20pt} }
   }
   {\end{list}}
\newenvironment{a_list_emph}{ 
  \begin{list}{{\emph{(\alph{acnt})}}}
   {\usecounter{acnt} \setlength{\itemsep}{3pt}
    \setlength{\leftmargin}{25pt} \setlength{\labelwidth}{20pt} }
   }
   {\end{list}}

\newcounter{Acnt}

\newcounter{icnt}
\newenvironment{i_list_emph}{ 
  \begin{list}{{\emph{(\roman{icnt})}}}
   {\usecounter{icnt} \setlength{\itemsep}{3pt}
    \setlength{\leftmargin}{25pt} \setlength{\labelwidth}{20pt} }
   }
   {\end{list}}
\newenvironment{i_list}{ 
  \begin{list}{{(\roman{icnt})}}
   {\usecounter{icnt} \setlength{\itemsep}{3pt}
    \setlength{\leftmargin}{25pt} \setlength{\labelwidth}{20pt} }
   }
   {\end{list}}
\newcounter{Icnt}

\newcounter{exam_cnt}

\newcounter{mccnt}

\newenvironment{bcirclist}{ 
  \begin{list}{\boldmath$\circ$\unboldmath}
   {\usecounter{cnt} \setlength{\itemsep}{2pt}
    \setlength{\leftmargin}{15pt} \setlength{\labelwidth}{20pt} }
   }
   {\end{list}}

\newtheorem{corollary}[equation]{Corollary}
\newtheorem*{corollary*}{Corollary}
\newtheorem{lemma}[equation]{Lemma}
\newtheorem*{lemma*}{Lemma}
\newtheorem{proposition}[equation]{Proposition}
\newtheorem*{proposition*}{Proposition}
\newtheorem{theorem}[equation]{Theorem}
\newtheorem*{theorem*}{Theorem}

\theoremstyle{definition}

\newtheorem*{boldQ*}{Question}
\newtheorem*{boldP*}{Problem}

\theoremstyle{remark}
\newtheorem*{assume*}{Assume}
\newtheorem*{answer*}{Answer}

\newtheorem*{claim*}{Claim}

\newtheorem{definition}[equation]{Definition}
\newtheorem*{definition*}{Definition}
\newtheorem{example}[equation]{Example}
\newtheorem*{example*}{Example}
\newtheorem*{hint*}{Hint}
\newtheorem*{notation*}{Notation}
\newtheorem{remark}[equation]{Remark}
\newtheorem*{remark*}{Remark}
\newtheorem*{remarks*}{Remarks}
\newtheorem*{fact*}{Fact}
\newtheorem*{emphL*}{Lemma}

\newtheorem*{emphQ*}{Question}
\newtheorem*{emphA*}{Answer}
\newtheorem*{convention}{Convention}

\numberwithin{HWeq}{section}
\theoremstyle{definition}

\def\tcpt{\mathit{c}}
\def\ncpt{\mathit{nc}}

\theoremstyle{theorem}
\newtheorem*{propI5}{Proposition \ref{P:I5}}

\newtheorem*{thmI8}{Theorem \ref{T:I8}}

\begin{document}
\title[Extremal degenerations of PHS's]{Extremal degenerations of polarized Hodge structures}
\author[Green]{Mark Green}
\email{mlg@ipam.ucla.edu}
\address{Department of Mathematics, University of California at Los Angles, Los Angeles, CA 90095}
\author[Griffiths]{Phillip Griffiths}
\email{pg@ias.edu}
\address{Institute for Advanced Study, Einstein Drive, Princeton, NJ 08540}
\author[Robles]{Colleen Robles}
\email{robles@math.tamu.edu}
\address{Mathematics Department, Mail-stop 3368, Texas A\&M University, College Station, TX  77843-3368} 
\thanks{Robles is partially supported by NSF grant DMS-1309238.  This work was undertaken while Robles was a member of the Institute for Advanced Study; she thanks the institute for a wonderful working environment and the Robert and Luisa Fernholz Foundation for financial support.}
\date{\today}


\maketitle

\setcounter{tocdepth}{1}

\tableofcontents

\section{Introduction}

An interesting question in algebraic geometry is: \emph{In what ways can a smooth projective variety $X$ degenerate?}  Here one imagines a situation
\begin{equation} \label{E:I1}
  \sX \ \stackrel{\pi}{\longrightarrow} \ S
\end{equation}
where $\sX$ and $S$ are complex manifolds with $\sX \subset \bP^N$ and where $\pi$ is a proper holomorphic mapping with $X$ a smooth fibre.  Then $\pi$ is a holomorphic submersion over a Zariski open set $S^*\subset S$, and one is interested in which varieties $X_s = \pi^{-1}(s)$ can arise when $s \in S \backslash S^*$.  The question of course needs refinement; e.g., by assuming some sort of semi--stable reduction for \eqref{E:I1} (\cf\cite{MR1738451}).

Hodge theory provides an invariant associated to \eqref{E:I1}.  Namely there is a period mapping
\begin{equation} \label{E:I2}
  \Phi : S^* \to \Gamma \backslash D
\end{equation}
where $D = G_\bR/R$ is a Mumford--Tate domain and for $s \in S^*$,
\begin{center}
  $\Phi(s)$ is the polarized Hodge structure on $H^n(X_s,\bQ)_\tprim$.
\end{center}
The ambiguity in the identification of $H^n(X_s,\bQ)_\tprim$ with a fixed vector space $V$ is given by the image of the monodromy representation
\[
  \rho : \pi_1(S^*) \to \Gamma \subset G \,.
\]
We note that the invariant \eqref{E:I2} of \eqref{E:I1} only depends on the family over $S^*$; it does not depend on the generally non-unique semi-stable reduction, although as we shall see it strongly limits what the singular fibres can be.

There are two ways of attaching Hodge--theoretic data to the limits
\begin{equation}\label{E:I3}
  \lim_{s\to s_o} \Phi(s) \,;
\end{equation}
these data will then reflect the specialization $X_s \to X_{s_o}$.  The first, and traditional, way is to think of \eqref{E:I3} as giving a \emph{limiting mixed Hodge structure}.  Specifically, following \cite{MR840721} and \cite{MR2465224} and taking $S \backslash S^*$ to a be a local normal crossing divisor, one attaches to $D$ a set $B(\Gamma)$ of equivalence classes of limiting mixed Hodge structures and extends \eqref{E:I2} to an \emph{extended period mapping}
\[
  \Phi_e : S \to \Gamma \backslash \left(D \cup B(\Gamma) \right) \,.
\]
One may roughly think of $\Phi_e(s_o)$ as containing a \emph{maximal} amount of Hodge--theoretic information in the limit.

To explain the second, more recent method we assume that $S = \Delta$ is the unit disc and $S^* = \Delta^*$ the punctured disc so that \eqref{E:I1} becomes 
\begin{equation}\label{E:I4}
  \Phi : \Delta^* \to \Gamma_T\backslash D
\end{equation}
where $T$ is the unipotent monodromy transformation with logarithm $N$ and $\Gamma_T = \{ T^k\}_{k \in \bN}$.  Then \eqref{E:I4} may be lifted to a mapping of the upper--half plane 
\[
  \tilde \Phi : \bH \to D \,,
\]
and following \cite[Appendix to Lecture 10]{MR3115136} and \cite{KP2013} we define\footnote{Precise definitions of all notions discussed in this introduction are given in later sections.} the \emph{reduced limit period mapping} associated to \eqref{E:I4} by 
\begin{equation}\label{E:I5}
  \lim_{z \to \infty} \tilde\Phi(z) \ \in \ \del D
\end{equation}
where $D$ is embedded in its compact dual $\check D$ and $\del D \subset \check D$, \cf Section \ref{S:rlpm}.

To explain this a bit more, in the situation \eqref{E:I4} the boundary component $B(\Gamma)$ referred to above becomes 
\[
  B(N) \ = \ \left\{
  \begin{array}{l}
  \hbox{equivalence classes of limiting mixed Hodge structures}\\
  \hbox{$(V,W_\sb(N),F^\sb)$ with monodromy weight filtration $W_\sb(N)$}
  \end{array} \right\}\,.
\]
There is then a mapping 
\[
  \Phi_\infty : B(N) \to \del D \,,
\]
and the reduced limit period mapping \eqref{E:I5} is the composition of this mapping with $\Phi_e(s_o)$, where $s_o = 0 \in \Delta$.  We will abbreviate it by $\Phi_\infty(s_o)$. It is well--defined since \eqref{E:I5} is a fixed point of $T$.  We may roughly think of $\Phi_\infty(s_o)$ as containing the \emph{minimal} amount of Hodge--theoretic information in the limit \eqref{E:I3}.  For the classical case of weight $n=1$ polarized Hodge structures, $\Phi_e$ corresponds to a toroidal compactification \cite{MR0457437} and $\Phi_\infty$ to the Satake--Bailey--Borel compactification \cite{MR0216035,MR0170356}.

One advantage of the reduced limit period mapping is that it maps to a space on which the group $G_\bR$ acts.  One may then use the rich and well understood structure of the partially ordered lattice of $G_\bR$--orbits in $\del D$ to define what is meant by extremal degenerations of a polarized Hodge structure, \cf Section \ref{S:dfn}.  Specifically, a $G_\bR$--orbit $\cO \subset \del D$ is said to be \emph{polarized relative to the Mumford--Tate domain $D$} in case there is a period mapping \eqref{E:I4} whose reduced limit period lies in $\cO$, \cf Definition \ref{d:pol}.   When the infinitesimal period relation is bracket--generating, all orbits in $\del D$ of real codimension one in $\check D$ are polarizable relative to some Mumford--Tate domain structure on the open $G_\bR$--orbit $D$, \cf \cite{MR3115136, KP2013} or Section \ref{S:codim=1}.  The unique closed orbit in $\del D$ is sometimes, but not always, polarizable.  The general question of polarizability is discussed in \cite{KP2013, KR1}.

A degeneration \eqref{E:I4} of a polarized Hodge structure $\Phi(s)$ is said to be \emph{minimal} if the reduced limit period lies in a codimension--one $G_\bR$--orbit; it is said to be \emph{maximal} if its reduced limit lies in an orbit whose closure does not lie in a proper sub-orbit that is polarizable relative to $D$, \cf Definition \ref{d:extremal}.  One way to think of this is the following: Points of $D$ are given by polarized Hodge structures $(V,Q,F^\sb)$.  Then arbitrary period mappings may be well--approximated by nilpotent orbits with the same limit period mapping.  Thus a minimal degeneration of $F^\sb \in D$ is given by a nilpotent orbit such that 
\[
  \lim_{z\to\infty} e^{zN} \cdot F^\sb \ \hbox{ lies in a codimension--one $G_\bR$--orbit in $\del D$.}
\]
Intuitively these are the \emph{least} degenerate limiting mixed Hodge structures that $F^\sb \in D$ can specialize to.  Similarly, maximal degenerations are the \emph{most} degenerate that $F^\sb$ can specialize to.

The main results of this paper will describe the \emph{extremal} --- the minimal and maximal --- degenerations of polarized Hodge structures in a number of cases.  These will be described in terms of the types of limiting mixed Hodge structure that maps to the reduced limit period point in $\del D$.  We shall deal with two types of polarized Hodge structures.

\medskip

\noindent\emph{Type I.} \ These are polarized Hodge structures $(V,Q,F^\sb)$ of weight $n > 0$ that we think of as $H^n(X,\bC)_\tprim$ for a smooth algebraic variety $X$ of dimension $n$. 

\medskip

\noindent\emph{Type II.} \ These are polarized Hodge structures $(\fg,Q_\fg,F^\sb_\fg)$ of weight $n=0$ and where, unless otherwise mentioned, $-Q_\fg$ is the Cartan--Killing form.

\medskip

\noindent We think of polarized Hodge structures of Type I as directly related to algebraic geometry.  Limiting mixed Hodge structures $(V,W_\sb(N),F^\sb)$ arising from polarized Hodge structures of Type I may be pictured in the first quadrant of the $(p,q)$--plane in terms of the Deligne splitting
\[
  V_\bC \ = \ \bigoplus_{\mystack{0\le p,q}{0\le p+q\le 2n}} I^{p,q} 
\]
where dots indicate a possibly nonzero $I^{p,q}$.  For example, a pure Hodge structure of weight $n=5$ is depicted as
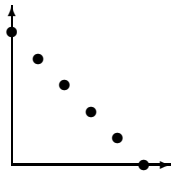
\begin{center}
\setlength{\unitlength}{10pt}
\begin{picture}(6,6)
\put(0,0){\vector(1,0){6}}
\put(0,0){\vector(0,1){6}}
\multiput(0,5)(1,-1){6}{\circle*{0.35}}
\end{picture}
\end{center} 

A polarized Hodge structure of Type I gives one of Type II with $\fg \subset \tEnd(V,Q)$.  In this case the corresponding Mumford--Tate domains are the same.  The polarized Hodge structures of Type II are especially convenient when studying the geometry of the $G_\bR$--orbits $\cO$ in $\check D$.  For example, suppose that the limiting mixed Hodge structures $(V,W_\sb(N),F^\sb)$ is $\bR$--split.  The induced adjoint limiting mixed Hodge structures $(\fg,W_\sb(N)_\fg,F^\sb_\fg)$ is also $\bR$--split.  Let $\fg_\bC = \op I^{p,q}_\fg$ be the Deligne splitting.  Let 
\[
  F^\sb_\infty \ = \ \lim_{z\to \infty} e^{zN}F^\sb \ \in \ \cO \,.
\]  
Then the tangent and normal spaces are naturally identified with
\begin{subequations} \label{SE:TN}
\begin{eqnarray}
  T_{F^\sb_\infty} \cO & = & \bigoplus_{p>0 \ \mathrm{or} \ q>0} 
  \left( I_\fg^{p,q} \op I_\fg^{q,p}\right)_\bR \,,\\
  \label{E:NcO}
  N_{F^\sb_\infty} \cO & = & 
  \bi \,\bigoplus_{p,q>0}
   \left( I_\fg^{p,q} \op I_\fg^{q,p}\right)_\bR \,,
\end{eqnarray}
\end{subequations}
\cf Section \ref{S:Lie_str}.  In particular, much of the geometry (such as dimension and codimension, CR--tangent space, and intrinsic Levi form) associated with the $G_\bR$--orbit $\cO \subset \check D$ can be ``read off'' from the Deligne splitting.  Moreover, each $I^{p,q}_\fg$ may be realized as a direct sum of root spaces (and a Cartan subalgebra if $p=q=0$), and this Lie theoretic structure plays an essential r\^ole in the analysis. 

\subsection{Minimal degenerations}

We begin with a result for period domains.

\begin{theorem} \label{T:intro_min}
Given a period domain $D$ parameterizing polarized Hodge structures of weight $n$, the minimal degenerations have either 
\[
\begin{array}{l}
  N^2 \ = \ 0 \tand \trank\, N \ \in \ \{ 1,2\} \,, \quad \hbox{or}\\
  N^2\ \not=0 \ \,,\ N^3 \ = \ 0 \tand \trank\,N = 2 \,.
\end{array}
\]
\end{theorem}

\noindent  We shall describe the nonzero $I^{p,q}$.  Figure \ref{f:min-pd} illustrates the possibilities for weights one through four; from these the reader will easily guess (correctly) what the general case will be.

\afterpage{\clearpage}

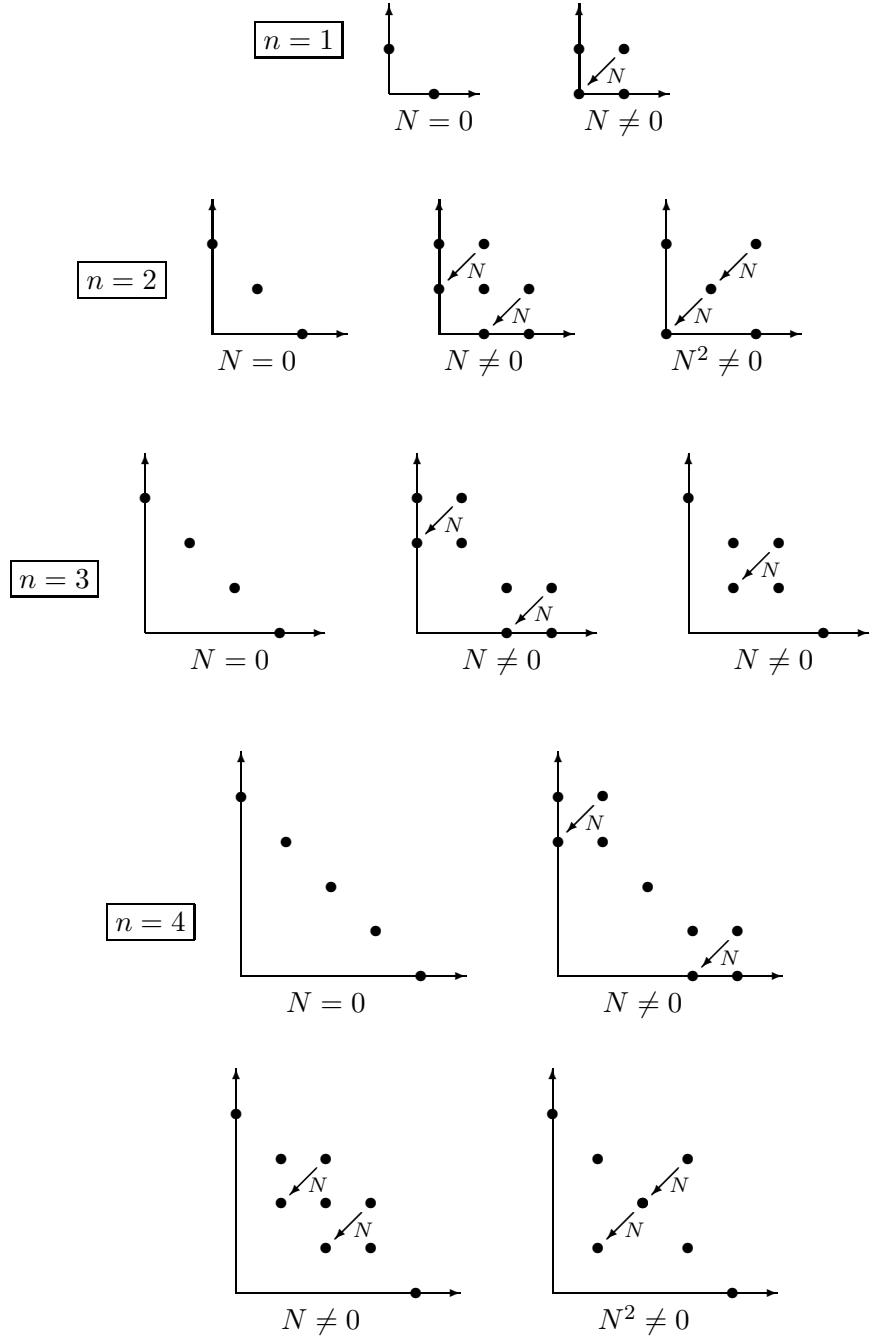
\begin{figure}[!th]
\caption{Minimal degenerations for weight $n$ period domains}
\setlength{\unitlength}{17pt}
\begin{picture}(5,4)(-3,-1)  
\put(0,0){\vector(0,1){2}}
\put(0,0){\vector(1,0){2}}
\put(0,1){\circle*{0.22}}
\put(1,0){\circle*{0.22}}
\put(-3,1){\fbox{\small{$n=1$}}}
\put(0.1,-0.8){\small{$N=0$}}
\end{picture}
\hspace{30pt}
\begin{picture}(2,3)(0,-1)
\put(0,0){\vector(0,1){2}}
\put(0,0){\vector(1,0){2}}
\put(0,1){\circle*{0.22}}
\put(1,0){\circle*{0.22}}
\put(0,0){\circle*{0.22}}
\put(1,1){\circle*{0.21}}
\put(0.8,0.8){\vector(-1,-1){0.6}}
\put(0.1,-0.8){\small{$N\not=0$}}
\put(0.6,0.25){\tiny{$N$}}
\end{picture}
%
\vspace{5pt}\\
%
\begin{picture}(7,5)(-3,-1)  
\put(0,0){\vector(0,1){3}}
\put(0,0){\vector(1,0){3}}
\put(0,2){\circle*{0.22}}
\put(1,1){\circle*{0.22}}
\put(2,0){\circle*{0.22}}
\put(-3,1){\fbox{\small{$n=2$}}}
\put(0.1,-0.8){\small{$N=0$}}
\end{picture}
\hspace{10pt}
\begin{picture}(4,5)(0,-1)
\put(0,0){\vector(0,1){3}}
\put(0,0){\vector(1,0){3}}
\put(0,2){\circle*{0.22}} \put(0,1){\circle*{0.22}}
\put(1,1){\circle*{0.22}} \put(1,0){\circle*{0.22}} \put(1,2){\circle*{0.22}}
\put(2,0){\circle*{0.22}} \put(2,1){\circle*{0.22}}
\put(0.8,1.8){\vector(-1,-1){0.6}}
\put(1.8,0.8){\vector(-1,-1){0.6}}
\put(0.1,-0.8){\small{$N\not=0$}}
\multiput(0.6,1.25)(1,-1){2}{\tiny{$N$}}
\end{picture}
\hspace{10pt}
\begin{picture}(4,5)(0,-1)
\put(0,0){\vector(0,1){3}}
\put(0,0){\vector(1,0){3}}
\put(0,2){\circle*{0.22}} \put(0,0){\circle*{0.22}}
\put(1,1){\circle*{0.22}} 
\put(2,0){\circle*{0.22}} \put(2,2){\circle*{0.22}}
\put(0.8,0.8){\vector(-1,-1){0.6}}
\put(1.8,1.8){\vector(-1,-1){0.6}}
\put(0.1,-0.8){\small{$N^2\not=0$}}
\multiput(0.6,0.25)(1,1){2}{\tiny{$N$}}
\end{picture}
%
\vspace{10pt}\\
%
\begin{picture}(8,6)(-3,-1)  
\put(0,0){\vector(0,1){4}}
\put(0,0){\vector(1,0){4}}
\put(0,3){\circle*{0.22}}
\put(1,2){\circle*{0.22}}
\put(2,1){\circle*{0.22}}
\put(3,0){\circle*{0.22}}
\put(-3,1){\fbox{\small{$n=3$}}}
\put(1,-0.8){\small{$N=0$}}
\end{picture}
\hspace{10pt}
\begin{picture}(5,6)(0,-1)
\put(0,0){\vector(0,1){4}}
\put(0,0){\vector(1,0){4}}
\put(0,3){\circle*{0.22}} \put(0,2){\circle*{0.22}}
\put(1,2){\circle*{0.22}} \put(1,3){\circle*{0.22}}
\put(2,1){\circle*{0.22}} \put(2,0){\circle*{0.22}}
\put(3,0){\circle*{0.22}} \put(3,1){\circle*{0.22}}
\put(0.8,2.8){\vector(-1,-1){0.6}}
\put(2.8,0.8){\vector(-1,-1){0.6}}
\put(1,-0.8){\small{$N\not=0$}}
\multiput(0.6,2.25)(2,-2){2}{\tiny{$N$}}
\end{picture}
\hspace{10pt}
\begin{picture}(5,6)(0,-1)
\put(0,0){\vector(0,1){4}}
\put(0,0){\vector(1,0){4}}
\put(0,3){\circle*{0.22}} 
\put(1,2){\circle*{0.22}} \put(1,1){\circle*{0.22}}
\put(2,1){\circle*{0.22}} \put(2,2){\circle*{0.22}}
\put(3,0){\circle*{0.22}} 
\put(1.8,1.8){\vector(-1,-1){0.6}}
\put(1,-0.8){\small{$N\not=0$}}
\put(1.6,1.25){\tiny{$N$}}
\end{picture}
%
\vspace{10pt}\\
%
\begin{picture}(9,7)(-3,-1)  
\put(0,0){\vector(0,1){5}}
\put(0,0){\vector(1,0){5}}
\multiput(0,4)(1,-1){5}{\circle*{0.22}}
\put(-3,1){\fbox{\small{$n=4$}}}
\put(1,-0.8){\small{$N=0$}}
\end{picture}
\hspace{10pt}
\begin{picture}(6,7)(0,-1)
\put(0,0){\vector(0,1){5}}
\put(0,0){\vector(1,0){5}}
\multiput(0,4)(1,-1){5}{\circle*{0.22}}
\multiput(0,3)(1,1){2}{\circle*{0.22}}
\multiput(3,0)(1,1){2}{\circle*{0.22}}
\put(0.8,3.8){\vector(-1,-1){0.6}}
\put(3.8,0.8){\vector(-1,-1){0.6}}
\put(1,-0.8){\small{$N\not=0$}}
\multiput(0.6,3.25)(3,-3){2}{\tiny{$N$}}
\end{picture}
\\
\begin{picture}(9,7)(-3,-1)
\put(0,0){\vector(0,1){5}}
\put(0,0){\vector(1,0){5}}
\multiput(0,4)(1,-1){5}{\circle*{0.22}}
\multiput(1,2)(1,1){2}{\circle*{0.22}}
\multiput(2,1)(1,1){2}{\circle*{0.22}}
\put(1.8,2.8){\vector(-1,-1){0.6}}
\put(2.8,1.8){\vector(-1,-1){0.6}}
\put(1,-0.8){\small{$N\not=0$}}
\multiput(1.6,2.25)(1,-1){2}{\tiny{$N$}}
\end{picture}
\hspace{10pt}
\begin{picture}(6,7)(0,-1)
\put(0,0){\vector(0,1){5}}
\put(0,0){\vector(1,0){5}}
\multiput(0,4)(1,-1){5}{\circle*{0.22}}
\multiput(1,1)(1,1){3}{\circle*{0.22}}
\put(1.8,1.8){\vector(-1,-1){0.6}}
\put(2.8,2.8){\vector(-1,-1){0.6}}
\put(1,-0.8){\small{$N^2\not=0$}}
\multiput(1.6,1.25)(1,1){2}{\tiny{$N$}}
\end{picture}
\label{f:min-pd}
\end{figure}

\newpage

The general rule is this: Let $h^{p,q}$ be the Hodge numbers of the polarized Hodge structure $(V,Q,F^\sb)$ and 
\[
  i^{p,q} \ = \ \tdim\,I^{p,q} \,.
\]  
If $N\not=0$ and $N^2=0$, then for one pair $p_o<q_o$ with $p_o+q_o = n$ we have:
\begin{bcirclist}
\item
$i^{p_o+1,q_o}\,,\ i^{p_o,q_o-1} = 1$;
\item
$i^{p_o,q_o} = h^{p_o,q_o} - 1$ and $i^{p_o+1,q_o-1} = h^{p_o+1,q_o-1} - 1$;
\item
for all other $p<q$, $i^{p,q} = h^{p,q}$.
\end{bcirclist}
Put another way, for some $p_o<q_o$ one class in $V^{p_o,q_o}$ and one class in $V^{p_o+1,q_o-1}$ disappear in $\tGr^{W(N)}_n$ and reappear as classes
\begin{eqnarray*}
  \a & \in & I^{p_o+1,q_o} \,,\\
  N\a & \in & I^{p_o,q_o-1} \,.
\end{eqnarray*}
If $N^2\not=0$ and $N^3=0$, then $n=2m$ is even and we have:
\begin{bcirclist}
\item
$i^{m-1,m-1}\,,\ i^{m+1,m+1} =1$;
\item
$i^{m-1,m+1} = h^{m-1,m+1} - 1$ and $i^{m+1,m-1} = h^{m+1,m-1} - 1$;
\item
for all other $p<q$, $i^{p,q} = h^{p,q}$.
\end{bcirclist}
In this case, one class in $V^{m-1,m+1}$ and one class in $V^{m+1,m-1}$  disappear in $\tGr^{W(N)}_n$ and reappear as classes
\begin{eqnarray*}
  \a & \in & I^{m+1,m+1} \,,\\
  N^2\a & \in & I^{m-1,m-1} \,.
\end{eqnarray*}

The pictures above are particularly revealing when 
\begin{equation} \label{E:h1}
  h^{p,q} \ = \ \left\{ \begin{array}{ll}
    1 & \hbox{for all } \ p\not=q \\
    2 & \hbox{if } \ p=q \,.
  \end{array} \right.
\end{equation}
In this case they are as pictured in Figure \ref{f:min-pd-h1}; in these figures a uncircled node indicates $i^{p,q}=1$, a circled node indicates $i^{p,q}=2$.
\afterpage{\clearpage}
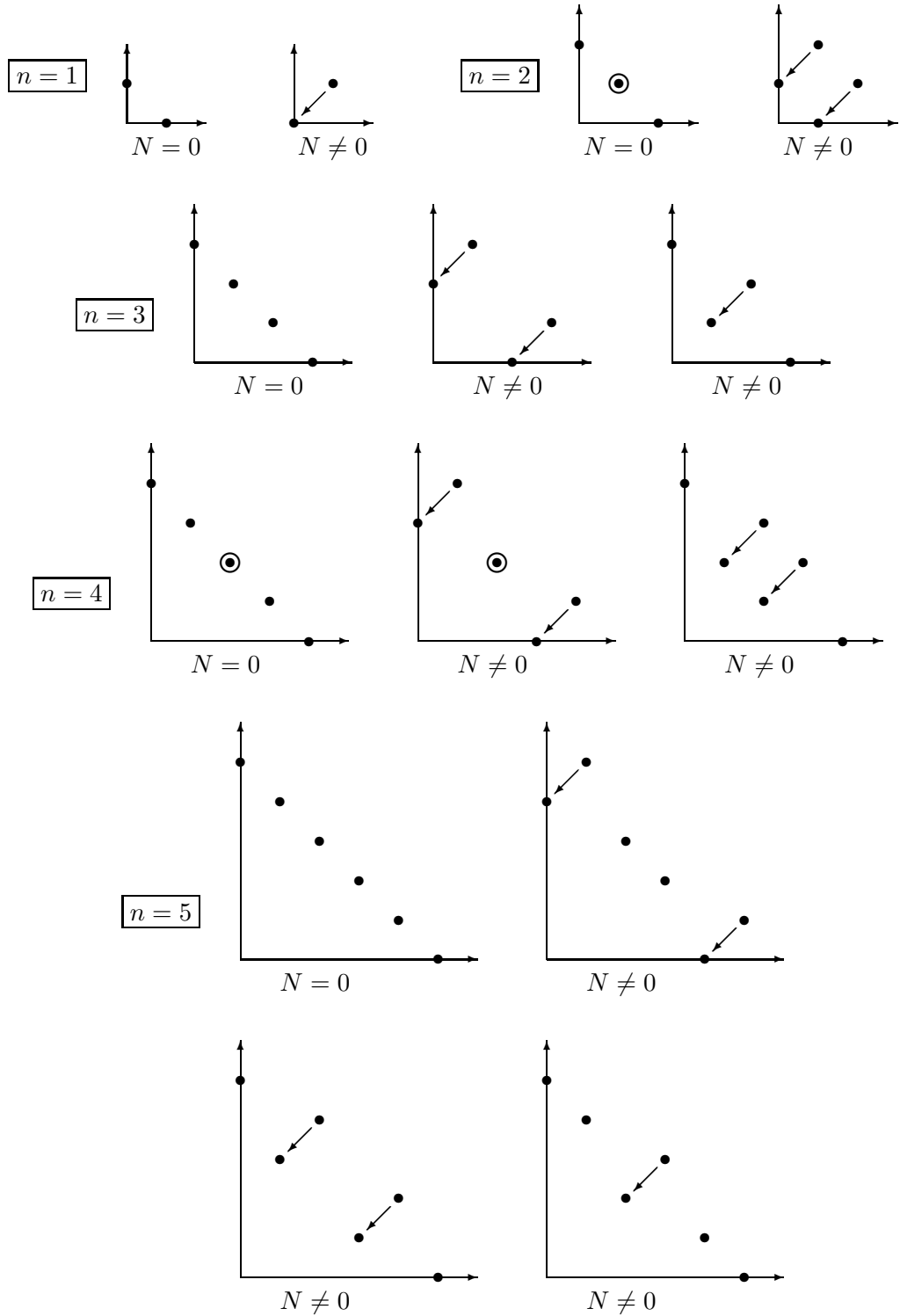
\begin{figure}[!th]
\caption{Minimal degenerations for period domains with \eqref{E:h1}}
\setlength{\unitlength}{17pt}
\begin{picture}(5,4)(-3,-1)
\put(0,0){\vector(0,1){2}}
\put(0,0){\vector(1,0){2}}
\put(0,1){\circle*{0.22}}
\put(1,0){\circle*{0.22}}
\put(-3,1){\fbox{\small{$n=1$}}}
\put(0.1,-0.8){\small{$N=0$}}
\end{picture}
\hspace{30pt}
\begin{picture}(2,3)(0,-1)
\put(0,0){\vector(0,1){2}}
\put(0,0){\vector(1,0){2}}
\put(0,0){\circle*{0.22}}
\put(1,1){\circle*{0.22}}
\put(0.8,0.8){\vector(-1,-1){0.6}}
\put(0.1,-0.8){\small{$N\not=0$}}
\end{picture}
%
\hspace{30pt}  
%
\begin{picture}(7,5)(-3,-1)
\put(0,0){\vector(0,1){3}}
\put(0,0){\vector(1,0){3}}
\put(0,2){\circle*{0.22}}
\put(1,1){\circle*{0.22}} \thicklines \put(1,1){\circle{0.45}} \thinlines
\put(2,0){\circle*{0.22}}
\put(-3,1){\fbox{\small{$n=2$}}}
\put(0.1,-0.8){\small{$N=0$}}
\end{picture}
\hspace{10pt}
\begin{picture}(4,5)(0,-1)
\put(0,0){\vector(0,1){3}}
\put(0,0){\vector(1,0){3}}
\put(0,1){\circle*{0.22}}
\put(1,0){\circle*{0.22}} \put(1,2){\circle*{0.22}}
\put(2,1){\circle*{0.22}}
\put(0.8,1.8){\vector(-1,-1){0.6}}
\put(1.8,0.8){\vector(-1,-1){0.6}}
\put(0.1,-0.8){\small{$N\not=0$}}
\end{picture}
\vspace{0pt}\\
\begin{picture}(8,6)(-3,-1)
\put(0,0){\vector(0,1){4}}
\put(0,0){\vector(1,0){4}}
\put(0,3){\circle*{0.22}}
\put(1,2){\circle*{0.22}}
\put(2,1){\circle*{0.22}}
\put(3,0){\circle*{0.22}}
\put(-3,1){\fbox{\small{$n=3$}}}
\put(1,-0.8){\small{$N=0$}}
\end{picture}
\hspace{10pt}
\begin{picture}(5,6)(0,-1)
\put(0,0){\vector(0,1){4}}
\put(0,0){\vector(1,0){4}}
\put(0,2){\circle*{0.22}}
\put(1,3){\circle*{0.22}}
\put(2,0){\circle*{0.22}}
\put(3,1){\circle*{0.22}}
\put(0.8,2.8){\vector(-1,-1){0.6}}
\put(2.8,0.8){\vector(-1,-1){0.6}}
\put(1,-0.8){\small{$N\not=0$}}
\end{picture}
\hspace{10pt}
\begin{picture}(5,6)(0,-1)
\put(0,0){\vector(0,1){4}}
\put(0,0){\vector(1,0){4}}
\put(0,3){\circle*{0.22}} 
\put(1,1){\circle*{0.22}} \put(2,2){\circle*{0.22}}
\put(3,0){\circle*{0.22}} 
\put(1.8,1.8){\vector(-1,-1){0.6}}
\put(1,-0.8){\small{$N\not=0$}}
\end{picture}
%
\vspace{0pt}\\
%
\begin{picture}(9,7)(-3,-1)
\put(0,0){\vector(0,1){5}}
\put(0,0){\vector(1,0){5}}
\multiput(0,4)(1,-1){5}{\circle*{0.22}}
\thicklines \put(2,2){\circle{0.45}} \thinlines
\put(-3,1){\fbox{\small{$n=4$}}}
\put(1,-0.8){\small{$N=0$}}
\end{picture}
\hspace{5pt}
\begin{picture}(6,7)(0,-1)
\put(0,0){\vector(0,1){5}}
\put(0,0){\vector(1,0){5}}
\multiput(0,3)(1,1){2}{\circle*{0.22}}
\multiput(3,0)(1,1){2}{\circle*{0.22}}
\put(2,2){\circle*{0.22}}
\thicklines \put(2,2){\circle{0.45}} \thinlines
\put(0.8,3.8){\vector(-1,-1){0.6}}
\put(3.8,0.8){\vector(-1,-1){0.6}}
\put(1,-0.8){\small{$N\not=0$}}
\end{picture}
\hspace{5pt}
\begin{picture}(6,7)(0,-1)
\put(0,0){\vector(0,1){5}}
\put(0,0){\vector(1,0){5}}
\put(0,4){\circle*{0.22}}
\put(4,0){\circle*{0.22}}
\multiput(1,2)(1,1){2}{\circle*{0.22}}
\multiput(2,1)(1,1){2}{\circle*{0.22}}
\put(1.8,2.8){\vector(-1,-1){0.6}}
\put(2.8,1.8){\vector(-1,-1){0.6}}
\put(1,-0.8){\small{$N\not=0$}}
\end{picture}
%
%
\begin{picture}(10,8)(-3,-1)
\put(0,0){\vector(0,1){6}}
\put(0,0){\vector(1,0){6}}
\multiput(0,5)(1,-1){6}{\circle*{0.22}}
\put(-3,1){\fbox{\small{$n=5$}}}
\put(1,-0.8){\small{$N=0$}}
\end{picture}
\hspace{5pt}
\begin{picture}(7,8)(0,-1)
\put(0,0){\vector(0,1){6}}
\put(0,0){\vector(1,0){6}}
\multiput(0,4)(1,1){2}{\circle*{0.22}}
\multiput(4,0)(1,1){2}{\circle*{0.22}}
\multiput(2,3)(1,-1){2}{\circle*{0.22}}
\put(0.8,4.8){\vector(-1,-1){0.6}}
\put(4.8,0.8){\vector(-1,-1){0.6}}
\put(1,-0.8){\small{$N\not=0$}}
\end{picture}
\vspace{0pt}\\
\begin{picture}(10,8)(-3,-1)
\put(0,0){\vector(0,1){6}}
\put(0,0){\vector(1,0){6}}
\put(0,5){\circle*{0.22}}
\put(5,0){\circle*{0.22}}
\multiput(1,3)(1,1){2}{\circle*{0.22}}
\multiput(3,1)(1,1){2}{\circle*{0.22}}
\put(1.8,3.8){\vector(-1,-1){0.6}}
\put(3.8,1.8){\vector(-1,-1){0.6}}
\put(1,-0.8){\small{$N\not=0$}}
\end{picture}
\hspace{5pt}
\begin{picture}(7,8)(0,-1)
\put(0,0){\vector(0,1){6}}
\put(0,0){\vector(1,0){6}}
\put(0,5){\circle*{0.22}}
\put(5,0){\circle*{0.22}}
\put(1,4){\circle*{0.22}}
\put(4,1){\circle*{0.22}}
\multiput(2,2)(1,1){2}{\circle*{0.22}}
\put(2.8,2.8){\vector(-1,-1){0.6}}
\put(1,-0.8){\small{$N\not=0$}}
\end{picture}
\label{f:min-pd-h1}
\end{figure}

From the algebro--geometric perspective, Theorem \ref{T:intro_min} (we will prove the more precise Theorem \ref{T:basicBC}) may at first glance seem surprising.  For example, for a smooth threefold $X$ a ``generic'' specialization might be thought to be $X \to X_o$ where $X_o$ has a node.  In the case that \eqref{E:h1} holds, this is the right--most picture of Figure \ref{f:min-pd-h1} for $n=3$.  The middle picture would arise from $X_o$ having a smooth double surface $D$, where there is an $\w \in H^0(\Omega^3_X)$ that specializes to $\w_o \in H^0(\Omega^3_{X_o}(\log D))$ whose residue goes to the class in $I^{2,0}$.

In the case that $n=5$, the second and fourth picture of Figure \ref{f:min-pd-h1} may be interpreted as in the $n=3$ case.  For the third we may think of a five-fold $X$ specializing to $X_o$ with the local equation $\{ x_1 x_2 + x_3 x_4 = 0\}$ in $\bC^6$; that is, $X_o$ has a double point along a surface.  Then the $h^{4,1}$ drops by one.  In summary, the three degenerations (in the $n=5$ case) correspond to the local equations
\[
\begin{array}{ll}
  x_1 x_2 = 0 \,, &  \hbox{(double four-fold)} \\
  x_1 x_2 + x_3 x_4 = 0 \,, &  \hbox{(double surface)} \\
  x_1 x_2 + x_3 x_4 + x_5 x_6 = 0  \,, & \hbox{(double point).}
\end{array} 
\]

\bigskip

In the more general setting of Mumford--Tate domains, the $I^{p,q}$ may be more complicated than those of Figure \ref{f:min-pd}: there may be more $N$--strings, and they may be have length greater than two.  What we can say, in terms of generalizing Theorem \ref{T:intro_min}, is Proposition \ref{P:min_mtd}.   Nonetheless, from the Lie theoretic perspective, the codimension one orbits all possess a uniform structure in the following sense: for an appropriate choice of Cartan subalgebra $\fh_\bR$ (essentially one may think of this as reflecting a ``good choice'' of basis of $V$), the nilpotent $N$ will be a root vector and the normal space may be identified with a real root space \cite{MR3115136, KP2013}.  That is, 
\[
  N \ \in \ \fg^\a_\bR \,,
\]
where $\a \in \fh^*_\bC$ is a root and the root space $\fg^\a \subset \fg_\bC$ is defined over $\bR$, and \eqref{E:NcO} becomes
\[
  N_{F^\sb_\infty}\cO \ = \ \bi\, I^{1,1}_\fg(\bR) \ = \ 
  \bi \,\fg^\a_\bR \,.
\]

\subsection{Maximal degenerations}

We recall that mixed Hodge structure $(V,W_\sb,F^\sb)$ is \emph{of Hodge--Tate type} if the $I^{p,q}=0$ for all $p\not= q$.  (The Hodge structures in row 4 of Figure \ref{f:G2/B} are of Hodge--Tate type.)  For limiting mixed Hodge structures there is the general

\begin{proposition}  \label{P:I5}
The limiting mixed Hodge structure $(V,W_\sb(N),F^\sb)$ is of Hodge--Tate type if and only if the associated adjoint limiting mixed Hodge structure $(\fg,W_\sb(N)_\fg,F^\sb_\fg)$ is of Hodge--Tate type.
\end{proposition}

\noindent The proposition is proved in Section \ref{S:HT}.

For the description here of the maximal degenerations we shall make the assumption 
\begin{center}
\emph{the closed $G_\bR$--orbit $\cO_\mathrm{cl}$ is polarizable relative to $D$.}
\end{center}
Thus the image \eqref{E:I5} of \eqref{E:I4} under the reduced limit period map is a point $F^\sb_\infty \in \cO_\mathrm{cl}$.

\begin{theorem} \label{T:I6}
The following are equivalent:
\begin{a_list_emph}
\item 
The orbit $\cO_\mathrm{cl}$ is totally real; i.e., the Cauchy--Riemann tangent space $T^\mathrm{CR}\cO_\mathrm{cl} = 0$.
\item 
The real dimension of the closed orbit is the complex dimension of the compact dual: $\tdim_\bR \cO_\mathrm{cl} = \tdim_\bC \check D$.
\item 
The stabilizer $P = \tStab_{G_\bC}(F^\sb_\infty)$ is $\bR$--split and $\cO_\mathrm{cl} = G_\bR/P_\bR$.
\end{a_list_emph}
Any of these imply that the limiting mixed Hodge structures $(V,W_\sb(N),F^\sb)$ and $(\fg,W_\sb(N)_\fg,F^\sb_\fg)$ are of Hodge--Tate type.
\end{theorem}

\noindent The theorem is proved as Theorem \ref{T:TR} and Corollary \ref{C:TR}.

From an algebro--geometric perspective it is not surprising that the most degenerate limiting mixed Hodge structure is one of Hodge--Tate type.  More interesting is that there are both Hodge theoretic (Lemma \ref{L:HT-HN}) and Lie theoretic (Lemma \ref{L:JMP} and Remark \ref{R:HTinG/B}) obstructions to a given type of polarized Hodge structure being able to degenerate to one of Hodge--Tate type.

In general there is the following

\begin{theorem} \label{T:I7}
Suppose that the limiting mixed Hodge structure $(V,W_\sb(N),F^\sb)$ is sent to the closed orbit under the reduced limit period mapping \eqref{E:I5} in the closed orbit.  Then Deligne splitting $\fg_\bC = \op I^{p,q}_\fg$ associated with the induced limiting mixed Hodge structure $(\fg,W_\sb(N)_\fg,F^\sb_\fg)$ satisfies:  
\begin{eqnarray*}
  I_\fg^{p,-q} = 0 & \hbox{for all} & p\not=q > 0\,, \\
  I_\fg^{p,-p} = 0 & \hbox{for all} & \hbox{odd } \ p \ge 3 \,,\\
  I_\fg^{p,q} = 0 & \hbox{for all} & p+q\not=0 \quad \hbox{with} \quad |p-q| > 2 \,.
\end{eqnarray*}
\end{theorem}

\noindent The constraints of Theorem \ref{T:I7} are illustrated in Figure \ref{f:c_orb}.b. We will prove the slightly stronger Theorem \ref{T:cp_orb}.

For period domains one may reconstruct the limiting mixed Hodge structure $(V,W_\sb(N),F^\sb)$ from $(\fg,W_\sb(N)_\fg,F^\sb_\fg)$; doing so, Theorem \ref{T:I7} yields

\begin{theorem} \label{T:I8}
Let $D$ be a period domain parameterizing weight $n$ Hodge structures.  If there exists a limiting mixed Hodge structure $(V,W_\sb(N),F^\sb)$ that maps to the closed $G_\bR$--orbit in $\check D$, but is \emph{not} of Hodge--Tate type, then $n = 2m$ is even and:
\begin{a_list_emph}
\item 
For $k\not=0$, $\tGr^{W_\sb(N)}_{n+k,\tprim}$ is of Hodge--Tate type.  
\emph{(Thus $k$ is even.)}
\item
For $k\not=0$, $\tGr^{W_\sb(N)}_{n+k,\tprim} \not=0$ implies $k \equiv 2$ (mod) $4$.
\item
$\tGr^{W_\sb(N)}_{n,\tprim} \not=0$, and the only nonzero $I^{p,q}_\tprim$, with $p+q=n$, are 
\[
  I^{m+1,m-1}_\tprim \tand
  I^{m-1,m+1}_\tprim \,.
\] 
\end{a_list_emph}
\end{theorem}

\noindent 
The theorem is proved in Section \ref{S:nHT-PD}.  A convenient schematic to picture a limiting mixed Hodge structures is its decomposition into $N$--strings
\begin{equation} \label{E:lmhs-pic}
 \begin{array}{c}
  H^0(n) \ \stackrel{N}{\longrightarrow} \ H^0(n-1) \ \stackrel{N}{\longrightarrow} 
  \cdots \stackrel{N}{\longrightarrow} \ H^0(1) \ \stackrel{N}{\longrightarrow} \ H^0 \\
  H^1(n-1)\ \stackrel{N}{\longrightarrow} 
  \cdots \stackrel{N}{\longrightarrow} \ H^1 \\
  \vdots \\
  H^{n-1}(1) \ \stackrel{N}{\longrightarrow} \ H^{n-1} \\
  H^n
\end{array}
\end{equation}
where 
\[
  H^k \ = \ N^{n-k} \,\tGr^{W_\sb(N)}_{2n-k,\tprim} 
\]
is a polarized Hodge structure of weight $k$.  (It may happen that $H^k$ is a Tate twist of a lower weight Hodge structure.)  Under the schematic \eqref{E:lmhs-pic}, the possibilities in Theorem \ref{T:I8} are:

\medskip

\noindent\fbox{$n=2$}  We have $H^1=0$, $H^0 \not=0$ and $H^2$ has type $(\ast,0,\ast)$.  In terms of the $i_\tprim^{p,q}$ the Hodge numbers are 
\[
  h^{2,0} \ = \ i^{2,0}_\tprim + i^{2,2}_\tprim \,,\quad
  h^{1,1} \ = \ i^{2,2}_\tprim \,.
\]


\noindent\fbox{$n=4$}  We have $H^0 = H^1 = H^3=0$, $H^2 \not=0$ is of Hodge--Tate type and $H^4$ has type $(0,\ast,0,\ast,0)$.  

\medskip

\noindent\fbox{$n=6$}  We have $H^1= H^2= H^3= H^5=0$, at least one of $H^0$ and $H^4$ is nonzero and of Hodge--Tate type, and $H^6$ has type $(0,0,\ast,0,\ast,0,0)$.  

\begin{corollary}
If the limiting mixed Hodge structure corresponds to a point in the closed orbit, then $\tGr_\ast^{W_\sb(N)}$ is rigid in the sense that it does not admit a non--trivial variation of Hodge structure.
\end{corollary}

\noindent The reason is that the infinitesimal period relation is trivial for the $\tGr_\ast^{W_\sb(N)}$.

The above results were under the assumption that the limit period map sends \eqref{E:I4} to a point in the closed orbit.  As we have discussed, in some, but not all, cases the limiting mixed Hodge structure is of Hodge--Tate type.  In the other direction we have 

\begin{proposition} \label{P:I9} 
If the limiting mixed Hodge structure is of Hodge--Tate type, then the limit period is a point in the closed orbit.
\end{proposition}

\noindent The proposition is proved in Section \ref{S:HT}.

Given a degeneration of polarized Hodge structures \eqref{E:I4} whose reduced limit mapping \eqref{E:I5} goes to a point in a $G_\bR$--orbit $\cO \subset \del D$, for the limiting mixed Hodge structure $(V,W_\sb(N),F^\sb)$ the above may be summarized as follows:
\begin{a_list}
\item
If $\cO$ is the closed orbit, and is totally real, then $(V,W_\sb(N),F^\sb)$ is Hodge--Tate.
\item
If $(V,W_\sb(N),F^\sb)$ is Hodge--Tate, then $\cO$ is the closed orbit.
\item
If $\cO$ is the closed orbit and $(V,W_\sb(N),F^\sb)$ is \emph{not} Hodge--Tate, then the nonzero $\tGr^{W_\sb(N)}_{n+k,\tprim}$ have $k \equiv 2$ (mod) $4$, and are Hodge--Tate, and $\tGr^{W_\sb(N)}_{n,\tprim}$ is as close to being Hodge--Tate as the Hodge numbers of $(V,Q,F^\sb)$ will allow.
\end{a_list}

\subsection{Notation}

\begin{bcirclist}
\item $\Delta = \{ t \in \bC \ : \ |t| < 1\}$ is the unit disc, $\Delta^* = \{ t \in \Delta \ : \ t \not=0\}$ is the punctured unit disc, and $\bH = \{ z \in \bC \ : \ \tIm(z) > 0 \}$ is the upper--half plane.
\item $\sX \to \Delta$ is a semi--stable reduction.
\item $\check D = G_\bC / P$ is a generalized flag variety containing $D = G_\bR/R$ as an open $G_\bR$--orbit with a compact isotropy group $R = G_\bR \cap P$.
\item $\Phi : \Delta^* \to \Gamma_T \backslash D$ is a period mapping with lift $\tilde \Phi : \bH \to D$ and $\tilde \Phi_* : T\bH \to I \subset T \check D$, where $I$ denotes the infinitesimal period relation.
\item $\tilde B(N)\subset \check D$ is the set of $N$--polarized limiting mixed Hodge structures, and $B(N)$ the boundary component consisting of equivalence classes of limiting mixed Hodge structures.
\item $\Phi_\infty : B(N) \to \del D$ is the reduced limit period map.
\item $F^\sb_\mathrm{lim} = \lim_{t\to0}\left( e^{-\ell(t)N} F^\sb_t \right) = \lim_{\tIm(z)\to\infty} e^{-zN}\tilde\Phi(z)$ where $\ell(t) = \frac{1}{2\pi\bi} \log t$ and $t = e^{2\pi\bi z}\in\Delta^*$ for $z \in \bH$.
\item $F^\sb_\infty = \lim_{\tIm(z) \to \infty} \tilde\Phi(z)$.
\item $\sT \subset R$ is a compact maximal torus of $G_\bR$.
\item $\fg_\bC$ and $\fg_\bR$ are the Lie algebras of $G_\bC$ and $G_\bR$.
\item $(V,W_\sb(N),F^\sb)$ a limiting mixed Hodge structure with $F^\sb = F^\sb_\mathrm{lim}$, $(V,W_\sb, \tilde F^\sb)$ is the associated $\bR$--split limiting mixed Hodge structure associated, and $(\fg , W_\sb(N)_\fg , F^\sb_\fg)$ the induced adjoint limiting mixed Hodge structures on $\fg$.
\item $V_\bC = \op I^{p,q}$ and $\fg_\bC = \op I^{p,q}_\fg$ are the Deligne splittings of a limiting mixed Hodge structure $(V,W_\sb(N),F^\sb)$ and the induced adjoint limiting mixed Hodge structure $(\fg,W_\sb(N)_\fg , F^\sb_\fg)$.
\item $-Q_\fg$ is the Killing form on $\fg$.
\item $\fh \subset \fg_\bC$ is a Cartan subalgebra with roots $\Delta \subset \fh^*$.\footnote{We shall use the notation $\Delta$ for both the set of roots of $(\fg_\bC,\fh)$ and for the unit disc in $\bC$; the context should make it clear which use of $\Delta$ is being made.}
\item Given $x \in \check D$, $\cO_x = G_\bR \cdot x$ is the $G_\bR$--orbit.
\end{bcirclist}

\section{Reduced limit period mappings}

\subsection{Generalized flag varieties} \label{S:GFV}

A \emph{generalized flag variety} is a homogeneous complex manifold 
\[
  \check D \ = \ G_\bC / P 
\]
where $G_\bC$ is a complex, semi--simple Lie group and $P$ is a parabolic subgroup.\footnote{We use the notation $\check D$ because we shall mainly think of it as the compact dual of a generalized flag domain.}  When $P = B$ is a Borel subgroup we shall use the term \emph{flag variety}.  At the reference point $x_o = P \one$ there is a natural identification of the holomorphic tangent space
\begin{equation}\label{E:TD}
  T_{x_o} \check D \ = \ \fg_\bC/\fp \,.
\end{equation}

We shall describe $\fp$ in terms of the roots associated to a Cartan subalgebra $\fh \subset \fp$, which always exists.  The root space decomposition of $(\fg_\bC,\fh)$ is 
\[
  \fg_\bC \ = \ \fh \ \op \ \bigoplus_{\a\in\Delta} \fg^\a  
\]
where $\Delta \subset \fh^*$ is the set of \emph{roots} and $\fg^\a$ is the one--dimensional \emph{root space}.  We let $\Delta^+ \subset \Delta$ denote a choice of \emph{positive roots}.  This determines a Borel subalgebra $\fb \supset \fh$ by 
\[
  \fb \ = \ \fh \ \op \ \bigoplus_{\a\in\Delta^+} \fg^{-\a} \,.
\]
Conversely, a choice of Borel $\fb \supset \fh$ determines the positive roots by \[
  \Delta^+ \ = \ \{ \a\in\Delta \ : \ \fg^{-\a} \subset \fb \}\,.
\]
Let $\Delta^+_\mathit{s} \subset \Delta^+$ denote the set of \emph{simple roots} relative to that choice.

The first description of parabolic subalgebras $\fp \subset \fg$ is in terms of subsets $\Sigma \subset \Delta^+_\mathit{s}$.   Denoting by $\langle \Sigma \rangle \subset \Delta$ the roots spanned by $\Sigma$, we set 
\[
  \fp_\Sigma \ = \ 
  \underbrace{\fh \ \op \ \Big( \bigoplus_{\a\in\langle\Sigma\rangle} \fg^\a \Big)}_{\fp_r} \ \op \ 
  \underbrace{\Big( \bigoplus_{\b \in \Delta^- \backslash \langle\Sigma\rangle^-} \fg^\b \Big)}_{\fp_n} 
\]
where $\Delta^- = -\Delta^+$ and $\langle\Sigma\rangle^- = \langle\Sigma\rangle \cap \Delta^-$.  Then $\fp_\Sigma$ is a parabolic subalgebra with reductive Levi factor $\fp_r$ and nilpotent radical $\fp_n$.  Note that
\[
  \fh \ \subset \ \fb \ \subset \ \fp_\Sigma \,.
\]  
When a choice of Cartan and Borel has been made, any parabolic of the form $\fp_\Sigma$ is a \emph{standard parabolic}.  Every parabolic subalgebra $\fp$ of $\fg$ is $\tAd(G_\bC)$--conjugate to a standard parabolic.  Using the identification \eqref{E:TD} of the holomorphic tangent space we have 
\begin{equation} \label{E:TD2}
  T_{x_o} \check D \ = \ 
  \bigoplus_{\b\in\Delta^+\backslash\langle\Sigma\rangle^+} \fg^\b 
  \ =: \ \fp_n^+ \,.
\end{equation}

The second description of parabolic subalgebras $\fp \subset \fg$ is in terms of the set $\tHom(\rtL,\bZ)$ of \emph{grading elements} where $\rtL = \langle \Delta \rangle \subset \fh^*$ is the root lattice.  If $\Delta^+_s = \{ \a_1,\ldots,\a_r\}$, then there is a dual basis $\{ L_1 , \ldots , L_r\}$ for the set of grading elements given by
\begin{equation} \label{E:Li}
  \a_j(L_i) \ = \ \d_{ij} \,.
\end{equation}
Given a grading element $L$, under the action of $\fh$ on $\fg_\bC$ we have an eigenspace decomposition
\begin{equation} \label{E:g_l}
\renewcommand{\arraystretch}{1.3}
\begin{array}{rcl}
  \fg_\bC & = &  \fg_{-k} \,\op\cdots\op\,\fg_0\,\op\cdots\op\,\fg_k \,,\\
  \hbox{with}\quad
  \fg_\ell & = & \{ \xi \in \fg \ : \ [L,\x] = \ell \x \} \,.
\end{array}
\end{equation}
Because the roots are integral linear combination of the simple roots, we see from \eqref{E:Li} that the eigenvalues of $L$ are integers $\ell \in \bZ$.  We note that 
\begin{bcirclist}
\item
Each $\fg_\ell$ is a direct sum of root spaces (and $\fh$ when $\ell = 0$).  Explicitly,
\begin{equation} \label{E:glrts}
\renewcommand{\arraystretch}{1.4}
\begin{array}{rcl}
  \fg_\ell & = & \displaystyle
  \bigoplus_{\a(L) = \ell} \fg^\a \,\quad\hbox{if } \ell \not=0 \,,\\
  \fg_0 & = & \displaystyle
  \fh \ \op \  \bigoplus_{\a(L) = 0} \fg^\a \,.
\end{array}
\end{equation}
\item
The eigenspaces $\fg_\ell$ and $\fg_{-\ell}$ pair non--degenerately under the Cartan--Killing form.
\item The Jacobi identity yields
\begin{equation} \label{E:Jid1}
  [\fg_\ell,\fg_m] \ \subset \ \fg_{\ell+m} \,.
\end{equation}
\item 
The Cartan subalgebra $\fh \subset \fg_0$ and $\fg_0$ is a reductive subalgebra.   (Indeed, $\fg_0$ is a Levi factor the parabolic subalgebra \eqref{E:pL}.)  By \eqref{E:Jid1}, each $\fg_\ell$ is a $\fg_0$--module.
\item
More generally, any representation $U$ of $\fg_\bC$ admits an $L$--eigenspace decomposition.  Since the weights of $\fg_\bC$ are rational linear combinations of the roots, the eigenvalues are rational $U = \op_{m \in \bQ}\, U_m$.  
\end{bcirclist}

\begin{remark} \label{R:grelem}
Grading elements may be defined without reference to a choice of positive roots (equivalently, a choice of Cartan and Borel $\fh \subset \fb$).  In general, a grading element is any semisimple endomorphism of $\fg_\bC$ with integer eigenvalues and the property that the eigenspace decomposition \eqref{E:g_l} satisfies \eqref{E:Jid1}.  If $\fg_\bC$ is semisimple, any such endomorphism is necessarily a grading element \cite[Proposition 3.1.2(1)]{MR2532439}.\footnote{The definition of grading element in \cite{MR2532439} is more restrictive than ours: it imposes the condition that $\fg_1$ generate the Lie subalgebra $\fg_+$.  Nonetheless the proof of \cite[Proposition 3.1.2(1)]{MR2532439} applies to our looser notion.}
\end{remark}

Given $L$ we may define the parabolic subalgebra
\begin{equation} \label{E:pL}
  \fp_L \ = \ \fg_0 \,\op\, \fg_-
\end{equation}
where $\fg_- = \fg_{-1}\,\op\cdots\op\,\fg_{-k}$.  Then \eqref{E:pL} is the Levi decomposition of $\fp_L$.  This associates to every grading element a standard parabolic subalgebra.  Conversely, given a standard parabolic $\fp_\Sigma$, there is a canonically associated grading element
\begin{equation} \label{E:L_p}
  L_\Sigma \ = \ \sum_{\a_i \not\in \Sigma} L_i \,.
\end{equation}
In particular, the relationship between the root and grading element descriptions is: Given $L$ define 
\[
  \Sigma_L \ = \ \{ \a \in \Delta^+_s \ : \ \a(L) = 0 \} \ \subset \ 
  \Delta^+_s \,.
\]

With the grading element notation we note that \eqref{E:TD2} is
\[
  T_{x_o} \check D \ = \ \fg_+ \ = \ \fg_1 \,\op\cdots\op\,\fg_k \,.
\]

\subsection{Generalized flag domains}

A real form $\fg_\bR$ of $\fg_\bC$ is the set of fixed points of a conjugation
\[
  \s : \fg_\bC \to \fg_\bC \quad
  \hbox{satisfying}\quad
  \s(\lambda X) = \bar \lambda (X) \hbox{ for all } 
  \lambda \in \bC \hbox{ and } X \in \fg_\bC \,.
\]
We denote by $G_\bR \subset G_\bC$ the corresponding connected real Lie group.  A \emph{generalized flag domain} is defined to be an open $G_\bR$--orbit $D \subset \check D$ whose isotropy group is compact.  A period domain is an example of a generalized flag domain.  If $D = G_\bR \cdot x_o \subset \check D = G_\bC/P$, then the isotropy group is 
\[
  R \ = \ G_\bR \,\cap\, P \,.
\]

Conversely, given a homogeneous complex manifold $D = G_\bR/R$ with $R \subset G_\bR$ the compact centralizer of a torus, its compact dual is a generalized flag variety $\check D = G_\bC/P$ as above in which $D$ is an open $G_\bR$--orbit.  It is known that $R$ contains a compact maximal torus $\sT$ whose complexified Lie algebra $\ft_\bC$ is a Cartan subalgebra $\fh$ of $\fg$.  The roots $\Delta$ of $\fh$ take purely imaginary values on $\ft$, which gives that
\[
  \overline{\fg^\a} \ = \ \fg^{-\a} \,.
\]
We have the identification of the complexified (real) tangent space
\[
  T_{x_o,\bC} D \ = \ \bigoplus_{\a\in\Delta\backslash\langle\Sigma\rangle} \fg^\a
  \ = \ \fg_+ \,\op\, \fg_- \,,
\]
and with a suitable choice of positive roots we have 
\[
  T^{1,0}_{x_o} D \ = \ \bigoplus_{\a\in\Delta^+\backslash\langle\Sigma\rangle^+} \fg^\a
  \ = \ T_{x_o} D \,.
\]
If we specify $\check D$ (equivalently, $P \subset G_\bC$) by a grading element $L$ as above, then we have 
\begin{equation}\label{E:conjL}
  \overline L \ = \ -L \,.
\end{equation}

\subsection{Polarized Hodge structures}

\subsubsection{Definition} \label{S:dfn_phs}

Let $(V,Q,F^\sb)$ denote a polarized Hodge structure of weight $n$ on a real vector space $V$.  That is, the \emph{polarization} $Q : V \times V \to \bR$ is a nondegenerate bilinear form such that
\[
  Q(u,v) \ = \ (-1)^n Q(v,u) \quad\hbox{for all } \ u,v \in V\,,
\]
and the \emph{Hodge filtration} $F^n \subset F^{n-1} \subset \cdots \subset F^1 \subset F^0 = V_\bC$ is a decreasing filtration on $V_\bC$ with the properties that 
\[
\renewcommand{\arraystretch}{1.3}
\begin{array}{lrcl}
  & F^p \,\op\,\overline{F^{n-p+1}} & \stackrel{\simeq}{\to} & V_\bC \,,\\
  \hbox{(HR1)} \hspace{60pt} &
  Q(F^p, F^{n-p+1}) & = &  0 \,, \hspace{70pt}\\
  \hbox{(HR2)} \hspace{60pt} &
  Q(Cv , \bar v) & > & 0 \quad\hbox{for all } 0 \not= v \in V_\bC \hspace{70pt}
\end{array}
\]
where $C$ denotes the Weil operator.  The equation (HR1) is the \emph{first Hodge--Riemann bilinear relation} (HR1); the inequality (HR2) is the \emph{second Hodge--Riemann bilinear relation}.

A \emph{$Q$--isotropic flag} is any filtration $F^\sb$ of $V_\bC$ satisfying (HR1).   Let
\[
  \mathbf{f} = (f^p = \tdim\,F^p)
\]
denote the \emph{Hodge numbers}.  We may regard the Hodge structure as a point in the $\tAut(V_\bC,Q)$--homogeneous generalized flag variety $\tFlag^Q_\mathbf{f}(V_\bC)$ of $Q$--isotropic filtrations $F^\sb$ of $V_\bC$ with dimension $\tdim\,F^p = f^p$.  The $\tAut(V_\bR,Q)$--orbit of any one of these Hodge structures is the \emph{period domain} of $Q$--polarized Hodge structures with Hodge numbers $\mathbf{f}$.  It is an open subset of $\tFlag^Q_\mathbf{f}(V_\bC)$; the latter is the \emph{compact dual} of the period domain.

The \emph{Hodge decomposition} is
\[
 V_\bC \ = \ \bigoplus_{p+q=n} V^{p,q}  \quad\hbox{with}\quad 
  V^{p,q} \ = \ F^p \,\cap\,\overline{F^q} 
\]
and satisfies
\begin{eqnarray*}
  Q(V^{p,q} , V^{r,s}) & \not= &
   0 \quad\hbox{only if} \quad (p,q) = (s,r) \,,\\
   \left. C\right|_{V^{p,q}} & = & \bi^{p-q} \one \,.
\end{eqnarray*}
We will also refer to 
\[
  \bh = (h^{p,q} = \tdim\,V^{p,q}) 
\]
as the \emph{Hodge numbers} of the Hodge structure.  Note that $f^p = \sum_{q\ge p}h^{q,n-q}$. 

A polarized Hodge structure $(V,Q,F^\sb)$ of weight $n$ is equivalent to a homomorphism
\[
  \varphi : S^1 \to \tAut(V_\bR,Q)
\]
of real algebraic groups with the properties that $\varphi(-1) = (-1)^n\one$ and $Q(\varphi(\bi)v,\bar v) > 0$ for all $0\not= v \in V_\bC$.   The Hodge decomposition is the $\varphi$--eigenspace decomposition
\[
  V^{p,q} \ = \ \{ v \in V_\bC \ : \ \varphi(z) v = z^{p-q} v \,,\ 
  \forall \ z \in S^1 \} \,.
\]
The \emph{Mumford--Tate group} $G$ of the Hodge structure is the $\bQ$--algebraic closure of $\varphi(S^1)$ in $\tAut(V_\bR,Q)$.\footnote{The Mumford--Tate group is the subgroup of $\tAut(V,Q)$ stabilizing the Hodge tensors \cite[(I.B.1)]{MR2918237}.}  The stabilizer of $F^\sb$ in $G_\bR$ is the compact $R = Z_\varphi = \{ g \in G_\bR \ : \ g\varphi(z) = \varphi(z) g \ \forall \ z \in S^1\}$.
 
\subsubsection{Induced PHS on $\tEnd(V,Q)$} \label{S:ind1}

There is an induced Hodge structure on the Lie algebra $\tEnd(V,Q)$ of $\tAut(V,Q)$ defined by 
\[
  F^p\,\tEnd(V_\bC,Q) \ = \ 
  \{ \x \in \tEnd(V_\bC,Q) \ : \ \x(F^q) \subset F^{p+q} \ \forall \ q\} \,.
\]
Equivalently, 
\begin{equation} \nonumber 
  \tEnd(V_\bC,Q)^{p,q} \ = \ 
  \{ \x \in \tEnd(V_\bC,Q) \ : \ \x(V^{r,s}) \subset V^{p+r,q+s} \ 
  \forall \ r,s \} \,.
\end{equation}
Note that $\tEnd(V_\bC,Q)^{p,q} = 0$ if $p+q \not=0$, so that
\begin{equation} \nonumber 
  \tEnd(V_\bC,Q) \ = \ \bigoplus_{p \in \bZ} \tEnd(V_\bC,Q)^{p,-p}
\end{equation}
is a weight zero Hodge structure.

\subsubsection{PHS in terms of grading elements} \label{S:PHSgrelem1}

We may view grading elements as ``infinitesimal Hodge structures'' as follows.\footnote{``Infinitesimal'' because, appropriately rescaled, $\varphi'(1)$ is a grading element, and conversely every grading element may be realized as $\varphi'(1)$, \cf\cite{schubVHS}.}  Note that the induced Hodge structure on $\tEnd(V,Q)$ satisfies
\[
  \left[ \tEnd(V_\bC,Q)^{p,-p} \,,\, \tEnd(V_\bC,Q)^{q,-q} \right] 
  \ \subset \ \tEnd(V_\bC,Q)^{p+q,-p-q} \,.
\]
So, if we define $L$ to be the semisimple endomorphism of $\tEnd(V_\bC,Q)$ that acts on $\tEnd(V_\bC,Q)^{p,-p}$ by the eigenvalue $-p$, then the discussion of Remark \ref{R:grelem} implies that $L \in \tEnd(V_\bC,Q)$ is a grading element.   

The grading element $L$ also induces the original Hodge structure on $V$.  In particular, the standard representation $V_\bC$ of $\tAut(V_\bC,Q)$ decomposes into a direct sum $\op_m\, V_m$ of $L$--eigenvalues with rational eigenvalues (Remark \ref{R:grelem}).  This eigenspace decomposition is the Hodge decomposition
\begin{equation}\label{E:HSvES}
  V^{p,q} \ = \ V_{(q-p)/2} \ = \ \{ v \in V_\bC \ : \ L(v) = \half(q-p)v \}\,.
\end{equation}
We say that \emph{the grading element defines a polarized Hodge structure on $V$}. 

\subsection{Mumford--Tate domains} \label{S:mtd}

\subsubsection{Definition} \label{S:dfn_mtd}

A Mumford--Tate domain is a generalized flag domain $D = G_\bR/R$ with additional data arising from a \emph{Hodge representation}
\[
  \rho : G_\bR \to \tAut(V,Q) \,.
\]
Specifically, in addition to $\rho$ we are given a grading element $L \in \tHom(\rtL,\bZ)$ such that 
\begin{bcirclist}
\item 
$\rho_*(L)$ defines a polarized Hodge structure $(V,Q,F^\sb)$;
\item
The Mumford--Tate group of $(V,Q,F^\sb)$ is equal to $\rho(G_\bR)$, and
\item
the isotropy group in $G_\bR$ of $F^\sb$ is equal to $R$.
\end{bcirclist}
By definition the \emph{Mumford--Tate domain $D$} is the set of polarized Hodge structures $\{ (V,Q,\rho(g) F^\sb) \ : \ g \in G_\bR\}$.\footnote{A more precise term would be a \emph{Mumford--Tate domain structure on the generalized flag domain $D$}.}  Note that $D$ is the $G_\bR$--orbit of $F^\sb$ in the period domain containing the polarized Hodge structure $(V,Q,F^\sb)$.  We may think of $D$ as the set of polarized Hodge structures $(V,Q,F^\sb_x)$, with $x \in D$, such that the Mumford--Tate group of each $(V,Q,F^\sb_x)$ is contained in $\rho(G)$, and equality holds for general $x$.

A Mumford--Tate domain structure gives an embedding of the compact dual
\[
  \check D \ \inj \ \tFlag_\mathbf{f}^Q(V_\bC) 
\]
as the $G_\bC$--orbit of any $F^\sb \in D$.  Specifically, $x \in \check D$ gives a flag $F^\sb_x$ that satisfies the first Hodge--Riemann bilinear relation $Q(F_x^p , F_x^{n-p+1}) = 0$.  The second Hodge--Riemann bilinear relation defines the open $G_\bR$--orbit $D \subset \check D$.

\begin{example} \label{eg:pd}
A special case of a Mumford--Tate domain is the period domain of Section \ref{S:dfn_phs}.  In this case $G_\bR = \tAut(V_\bR,Q)$ is $\tSO(2a,b)$ for even weight, and $\tSp(2g,\bR)$ for odd weight. 
\end{example}

In the case when $V$ is an irreducible $G$--module another way of think of a Mumford--Tate domain is that it is given by a pair $(\varphi,\chi)$ consisting of a co-character $\varphi$ and a character $\chi$ of $\sT \subset G_\bR$.  Specifically, for $S^1 = \bR/2\pi\bi\bZ$ and 
\[ 
  \varphi : S^1 \to \sT
\]
the isotropy subgroup $R = Z_{G_\bR}(\varphi(S^1))$ is the centralizer in $G_\bR$ of the circle $\varphi(S^1)$.  The $G_\bR$--invariant complex structure on $D = G_\bR/R$ is given by 
$$
  \tAd \circ \varphi : S^1 \to \fg_\bR/\fr \ \simeq \ T_{x_o,\bR}G_\bR/R \,.
$$
The character $\chi$ is the highest weight of the representation $\rho : G \to \tAut(V,Q)$.  There are conditions, not spelled out here, on the pair $(\varphi,\chi)$.  See \cite{MR2918237} for details.

\subsubsection{Induced PHS on $\fg$} \label{S:ind2}

From the description of Section \ref{S:dfn_mtd} it is clear that a given homogeneous complex manifold $D$, corresponding to $\varphi$ above, may be realized as a  Mumford--Tate domain in multiple ways, corresponding to the $\chi$'s above (see Section \ref{S:relD}).  For this work a particularly important pair of such realizations is the following:  Given a generalized flag domain $D$ realized as a Mumford--Tate domain for polarized Hodge structure $(V,Q,F^\sb)$, another realization is as induced polarized Hodge structures on $\fg \subset \tEnd(V,Q)$. Generalizing Section \ref{S:ind1} from $\tAut(V,Q)$ to the more general Mumford--Tate groups $G$, these are defined by 
\[
  F^p_\fg \ = \ 
  \{ \x \in \fg_\bC \ : \ \x(F^q) \subset F^{p+q} \ \forall \ q\} \,.
\]
Equivalently, 
\begin{equation} \nonumber 
  \fg^{p,q} \ = \ 
  \{ \x \in \fg_\bC \ : \ \x(V^{r,s}) \subset V^{p+r,q+s} \ \forall \ r,s \} \,.
\end{equation}
Note that $\fg^{p,q} = 0$ if $p+q \not=0$, so that
\begin{equation} \label{E:gwt0}
  \fg_\bC \ = \ 
  \fg^{-k,k} \,\op\cdots\op\,\fg^{0,0}\,\op\cdots\op\,\fg^{k,-k}
\end{equation}
is a weight zero Hodge structure.

The polarization $Q$ on $V$ induces a polarization $Q_\fg$ on $\fg$.  The latter is invariant under $G$.  Therefore, if $\fg$ is simple, then $-Q_\fg$ is necessarily a positive multiple of the Killing form.  Unless otherwise stated, 
\begin{center}
\emph{$-Q_\fg$ will denote the Killing form throughout.}
\end{center}

Notice that $F^0_\fg$ is the Lie algebra of the stabilizer $P \subset G_\bC$ of \emph{both} the Hodge structure $F^\sb$ on $V$ and the induced Hodge structure $F^\sb_\fg$ on $\fg$.  In particular, 
\begin{quote}
\emph{the $G_\bC$--orbits of $F^\sb \in \tFlag^Q_\mathbf{f}(V_\bC)$ and $F^\sb_\fg \in \tFlag^{Q_\fg}_{\mathbf{f}_\fg}(\fg_\bC)$ both realize the generalized flag variety $\check D = G_\bC/P$ as a projective variety. 
}
\end{quote}
Moreover, the two infinitesimal period relations agree under this identification, \cf\cite{MR2918237}.  Likewise, if 
$$
  R \ = \ G_\bR \,\cap\, P \,,
$$
then 
\begin{quote}
\emph{the $G_\bR$--orbits of $F^\sb \in \tFlag^Q_\mathbf{f}(V_\bC)$ and $F^\sb_\fg \in \tFlag^{Q_\fg}_{\mathbf{f}_\fg}(\fg_\bC)$ both realize the homogeneous manifold $D = G_\bR/R$ as a Mumford--Tate domain.
}
\end{quote}
By slight abuse of terminology we refer to the $G_\bR$--orbit of $F^\sb_\fg$ as the \emph{adjoint Mumford--Tate domain} $D_\fg$ associated to $D$, where the latter is viewed as the Mumford--Tate domain for the Hodge structure on $V$.  The reason for doing this is that coming from algebraic geometry one thinks of $(V,Q,F^\sb)$ as arising from $H^n(X,Q)_\mathrm{prim}$ where $X$ is a smooth projective variety.  However, in order to study (i) the geometry of the $G_\bR$--orbits in $\check D$ and (ii) Lie--theoretic aspects of the Hodge structure, it is necessary to work with the polarized Hodge structures $(\fg,Q_\fg,F^\sb_\fg)$.

\subsubsection{PHS in terms of grading elements} \label{S:PHSgrelem2}

In analogy with Section \ref{S:PHSgrelem1}, given a Hodge structure on $V$ with Mumford--Tate group $G$, there is a canonical choice of grading element.  To be precise, given the induced Hodge structure \eqref{E:gwt0} on $\fg$, define $L$ by 
\begin{equation} \label{E:Lphs}
  \left.L\right|_{\fg^{-p,p}} \ = \ p \one \,.
\end{equation}
Since $L$ is a derivation and $\fg_\bC$ is semisimple, $L$ is necessarily an element of $\fg_\bC$.  Moreover, since $L$ is semisimple, it is necessarily contained in a Cartan subalgebra.  And since the eigenvalues of $L$ on $\fg_\bC$ are integers, $L$ is necessarily a grading element (Remark \ref{R:grelem}).

Conversely, given a complex semisimple Lie algebra $\fg_\bC$ and a grading element $L \in \tHom(\rtL,\bZ)$, there is a canonical choice of real form $\fg_\bR$ (which we may take to be defined over $\bZ$) with the property that the $L$--eigenspace decomposition \eqref{E:g_l} of $\fg_\bC$ defines a weight zero Hodge structure \eqref{E:gwt0} by the assignment 
\begin{equation} \label{E:gp-pVgp}
  \fg^{-p,p} \ = \ \fg_p \,,
\end{equation}
\cf\cite[Proposition 2.36]{schubVHS}.  This Hodge structure on $\fg$ is related to the initial Hodge structure on $V$ by the grading element: the subspaces $V^{p,q}$ are also $L$--eigenspaces; that is, \eqref{E:HSvES} holds.  Observe, that while the $L$--eigenvalues on $\fg$ are integers, on $V$ they lie in $\half\bZ$.

The Hodge structure \eqref{E:gp-pVgp} on $\fg$ is polarized by $Q_\fg$; equivalently, if
\[
  \fk_\bC \ = \ \bigoplus_{\ell} \,\fg_{2\ell} \tand
  \fk^\perp_\bC \ = \ \bigoplus_\ell\, \fg_{2\ell+1} \,,
\]
then $\fk_\bR = \fg_\bR \cap \fk_\bC$ and $\fk^\perp_\bR = \fg_\bR \cap \fk^\perp_\bC$ define a \emph{Cartan decomposition} 
\[
  \fg_\bR \ = \ \fk_\bR \ \op \ \fk^\perp_\bR \,.
\]
We see from \eqref{E:glrts} and \eqref{E:gp-pVgp} that each $\fg^{p,-p}$ is a direct sum of root spaces (and $\fh$ if $p=0$).  Define \emph{compact} and \emph{noncompact roots} by 
\begin{eqnarray*}
  \Delta_\tcpt & = & \{ \a \in \Delta \ : \ \fg^\a \subset \fk_\bC \}
  \ = \ \{ \a\in \Delta \ : \ \a(L) \hbox{ is even} \} \,,\\
  \Delta_\ncpt & = & \{ \a \in \Delta \ : \ \fg^\a \subset \fk_\bC^\perp \}
  \ = \ \{ \a\in \Delta \ : \ \a(L) \hbox{ is odd} \} \,.
\end{eqnarray*}
Note that
\[
  \Delta \ = \ \Delta_\tcpt \,\cup\,\Delta_\ncpt \,.
\]

\subsubsection{Realizations of $D$} \label{S:relD}

In Section \ref{S:PHSgrelem2} we observed that both the Hodge structure on $V$ and the induced Hodge structure on $\fg$ are given by a common grading element $L$.  This fact may be used to deduce that the two Mumford--Tate domains $D$ and $D_\fg$ parameterizing Hodge structures of these types are isomorphic as $G_\bR$--homogeneous complex submanifolds $G_\bR/R \subset G_\bC/P$, \cf\cite{MR2918237}.  This is one example of a general method to realize $G_\bR/R$ as a Mumford--Tate domain, which we now outline.

Given a generalized flag variety $G_\bC/P$ there are canonically defined Hodge structures that realize the variety as a compact dual.  Let $L = L_\fp$ be the grading element \eqref{E:L_p} associated with $\fp$.  Let $\wtL \subset \fh^*$ be the weight lattice of $\fg_\bC$, and let $\{\w_1,\ldots,\w_r\}$ be the \emph{fundamental weights} with respect to the simple roots $\{\a_1,\ldots,\a_r\}$.  Given any dominant integral weight $\lambda \in \wtL$, let $U^\lambda$ denote the corresponding irreducible representation of $\fg_\bC$ with highest weight $\lambda$.  If $\lambda$ is a weight of $G_\bC$,\footnote{This will always be the case if $G_\bC$ is simply connected.} then the parabolic $P$ is the stabilizer of the highest weight line in $U^\lambda$ if and only if 
\begin{equation} \label{E:lambda}
  \lambda = \sum_{\a_i \not\in \Sigma} \lambda^i \w_i
  \quad\hbox{with}\quad 0 < \lambda^i \in \bZ \,.
\end{equation}
In this case, the $G_\bC$--orbit of the highest weight line is a homogeneous embedding $G_\bC/P \inj \bP U^\lambda$ that realizes the homogeneous complex manifold $G_\bC/P$ as a homogeneous projective variety.

Let $\fg_\bR$ be the real form determined by $L$ (Section \ref{S:PHSgrelem2}).  Given an irreducible representation $V_\bR$ of $\fg_\bR$ there exists an irreducible representation $U$ of $\fg_\bC$ such that one of the following holds:
\begin{bcirclist}
\item
$V_\bC = U$, in which case $U$ is \emph{real};
\item
$V_\bC = U \op U^*$ and $U \simeq U^*$, in which case $U$ is \emph{quaternionic};
\item
$V_\bC = U \op U^*$ and $U \not\simeq U^*$, in which case $U$ is \emph{complex}.
\end{bcirclist}
In the case that $\lambda$ is the highest weight of $U$, the representation $V_\bR$ admits a polarized Hodge structure $(F^\sb,Q)$ with Mumford--Tate group $G$ if and only if 
\begin{equation} \label{E:Llambda}
  L(\lambda) \ \in \ \half \bZ \,,
\end{equation}
\cf\cite{MR2918237}.  (A priori, we have only $L(\lambda) \in \bQ$.)  In this case, the Hodge structure $(V,Q,F^\sb)$ is given by the $L$--eigenspace decomposition \eqref{E:HSvES}.  In particular, 
\begin{quote}
\emph{$G_\bC/P$ is realized as the compact dual for any of these polarized Hodge structures.}
\end{quote}
From this perspective, a very natural realization is given by any $\lambda$ that minimizes the coefficients $\lambda^i$ of \eqref{E:lambda} subject to the constraint \eqref{E:Llambda}.  In many cases it is possible to take $\lambda^i=1$;\footnote{The issue is the following: In the case that $\fg_\bC$ is simple, the weights of $\fg_\bC$ lie in the $\tfrac{1}{d}\bZ$--span of the simple roots, where $1 \le d \in \bZ$ is the determinant of the Cartan matrix.  So, when $d \in \{1,2\}$ \eqref{E:Llambda} will hold with $\lambda^i=1$, $\a_i \not\in\Sigma$.  We have $d \in \{1,2\}$ when $\fg_\bC$ is one of $\fso_{2r+1}\bC$, $\fsp_{2r}\bC$, $\fe_7$, $\fe_8$, $\ff_4$ or $\fg_2$.  In the case that $\fg_\bC = \fsl_n\bC$, the determinant is $n+1$ and we will be able to satisfy \eqref{E:Llambda} with values $\lambda^i \in \{1,\ldots,n+1\}$; in the case that $\fg_\bC = \fso_{2r}\bC$, we have $d = 4$ and will be able to satisfy \eqref{E:Llambda} with $\lambda^i \in \{ 1,2\}$; in the case that $\fg_\bC = \fe_6$, we have $d = 3$ and will be able to satisfy \eqref{E:Llambda} with $\lambda^i \in \{1,2,3\}$.} this corresponds to the minimal homogeneous embedding $G/P \inj \bP U^\lambda$ of $G/P$ as a rational homogeneous variety.

Likewise,
\begin{quote}
\emph{the $G_\bR$--orbit of any of these Hodge filtrations $F^\sb$ realizes $G_\bR/R$ as a Mumford--Tate domain.}
\end{quote}

\begin{remark}
As the highest weight of the adjoint representation, the highest root $\tilde\a$ is a dominant integral weight.  By definition the grading element is integer--valued on $\tilde\a$.  So \eqref{E:Llambda} holds with $\lambda=\tilde\a$.  Whence, the construction above yields a polarized Hodge structure $(\fg , \tilde Q_\fg , \tilde F^\sb_\fg)$.  This agrees with the polarized Hodge structure $(\fg,Q_\fg,F^\sb_\fg)$ induced (as in Section \ref{S:ind2}) from any of the $(V,Q,F^\sb)$ constructed in this section.
\end{remark}

\subsection{Period mappings and nilpotent orbits} \label{S:pmno}

We shall only consider period mappings corresponding to a one--parameter family of degenerating polarized Hodge structures.  Such is given by a Mumford--Tate domain $D$, a unipotent monodromy transformation $T \in G$ and a locally liftable holomorphic mapping
\begin{equation} \label{E:Phi1}
  \Phi : \Delta^* \to \Gamma_T \backslash D \,.
\end{equation}
which satisfies the infinitesimal period relation.  Here $\Gamma_T = \{ T^k \ : \ k \in \bZ \}$.

Denoting by $\bH = \{ z \in \bC \ : \ \tIm(z) > 0 \}$ the upper--half plane with covering map $\bH \to \Delta^*$ given by 
\[
  t \ = \ e^{2\pi\bi z} \,,
\]
we may lift \eqref{E:Phi1} to give 
\begin{eqnarray*}
  & \tilde\Phi : \bH \to D \,, & \\
  & \tilde\Phi(z+1) \ = \ T \cdot \tilde\Phi(z) \,. &
\end{eqnarray*}
Setting $N = \log(T) \in \fg^\tnilp$, we may then ``unwind'' $\Phi$ to
\[
  \tilde\Psi : \bH \to \check D
\]
by defining
\[
  \tilde\Psi(z) \ = \ e^{-zN} \cdot \tilde\Phi(z) \,.
\]
Then $\tilde\Psi(z+1) = \tilde\Psi(z)$ so that there is an induced map 
\[
  \Psi : \Delta^* \to \check D \,.
\]
A basic result is that $\Psi$ extends across the origin $t=0$.  Then setting $\ell(t) = \log(t)/2\pi\bi$, the original period mapping is well approximated by the nilpotent orbit
\begin{equation} \label{E:Psi(0)}
  t \ \mapsto \ e^{\ell(t)N} \cdot \Psi(0) \,,
\end{equation}
see \cite{MR0382272}.  Implicit here is the statement that $e^{\ell(t)N}\cdot \Psi(0) \in D$ for $0 < |t| < \e$.  We shall further explain this below.

We shall sometimes write $\Phi(t) = F^\sb_t$ for the multi--valued filtration on $V_\bC$.  The lift $\tilde\Phi$ to $\bH \to D$ will be denoted by 
$$
  z \ \mapsto \ F^\sb_z
$$
where $F_{z+1}^\sb = T \cdot F_z^\sb$.

Because of the strong approximation of \eqref{E:Phi1} by a nilpotent orbit, we shall replace $\Phi$ by the nilpotent orbit \eqref{E:Psi(0)}.  For this we set
\[
  \Psi(0) \ = \ F^\sb_\tlim \,.
\]
Note that $F^\sb_\tlim$ is defined only up to the action of $\Gamma_T$ and a choice of coordinate $t$.  Rescaling $t$ by $t \mapsto e^{2\pi\bi \lambda}t$ induces the change
\[
  F^\sb_\tlim \ \to \ e^{\lambda N}\cdot F^\sb_\tlim \,.
\]
Thus what is well--defined is a map
\[
  \{ \hbox{period mappings \eqref{E:Phi1}} \} \,\times\, T_0^*\Delta 
  \ \to \ \hbox{nilpotent orbits.}
\]

The conditions that $\Phi$ define a period mapping translate into
\begin{equation} \label{E:Flim}
\renewcommand{\arraystretch}{1.3}
\begin{array}{rcl}
  e^{z N} \cdot F^\sb_\tlim & \in & D 
  \quad\hbox{for } \ \tIm(z) \gg 0 \,,\\
  N \cdot F^p_\tlim & \subset & F^{p-1}_\tlim
  \qquad\hbox{(infinitesimal period relation).}
\end{array}
\end{equation}

\begin{definition}
A \emph{nilpotent orbit} is given by $(F^\sb , N)$ where $F^\sb \in \check D$, $N \in \fg^\tnilp_\bR$ and where the conditions \eqref{E:Flim} are satisfied with $F^\sb$ in place of $F^\sb_\tlim$.
\end{definition}

Two nilpotent orbits $(F^\sb, N)$ and $({}'F^\sb,N)$ are \emph{equivalent} if 
\[
  {}'F^\sb \ = \ e^{\lambda N} F^\sb
\]
for some $\lambda \in \bC$.  We set 
\begin{bcirclist}
\item
$\tilde B(N) \ = \ \{\hbox{nilpotent orbits } (F^\sb,N)\}$,
\item
and let $B(N) = e^{\bC N} \backslash \tilde B(N)$ denote the set of equivalence classes of nilpotent orbits.
\end{bcirclist}

We next set 
\[
  D_N \ = \ D \,\cup\,B(N)
\]
and observe that the action of $\Gamma_T$ extends naturally to $D_N$.  In \cite{MR2465224} the structure of a log--analytic variety with slits is defined on $D_N$, and a basic result is that the period mapping \eqref{E:Phi1} extends to 
\begin{equation} \label{E:Phi3}
  \Phi_e : \Delta \to \Gamma_T \backslash D_N
\end{equation}
where the origin is mapped to the equivalence class of $(F^\sb_\tlim,N)$.  We shall refer to \eqref{E:Phi3} as the \emph{extended period mapping}.

\begin{definition}
The mapping
\[
  \{ \hbox{period mappings \eqref{E:Phi1}} \} \ \to \ \Phi_e(0) \ \in \ 
  \Gamma_T \backslash D_N
\]
will be called the \emph{limit period mapping}.
\end{definition}


\subsection{Nilpotent orbits and limiting mixed Hodge structures} \label{S:NO+MHS}

Let $D$ be a Mumford--Tate domain parameterizing weight $n$, $Q$--polarized Hodge structures on $V$ whose generic Mumford--Tate group is $G$.  Associated to a nilpotent orbit $(F^\sb,N)$ there is a special type of mixed Hodge structure $(V,W_\sb(N),F^\sb)$ called a \emph{limiting mixed Hodge structure} (limiting mixed Hodge structures).  There will then be a bijection of sets
\[
  \left\{ \hbox{nilpotent orbits} \right\} 
  \quad \stackrel{\eqref{E:CKS}}{\longleftrightarrow} \quad
  \left\{ \hbox{limiting mixed Hodge structures} \right\} ,
\]
which will pass to the quotient by taking equivalence classes.

Since $N$ is nilpotent, there exists $0 \le m \le n$ such that $N^m\not=0$ and $N^{m+1} = 0$.  Then the \emph{weight filtration}
\[
  W_{-m}(N) \,\subset\cdots\subset\,W_0(N)\,\subset\cdots\subset\,W_m(N) \ = \ V
\]
is the unique filtration that satisfies
\begin{eqnarray*}
  N : W_k(N) & \longrightarrow &  W_{k-2}(N) \,,\\
  N^k : \tGr_k^{W\sb(N)} & \stackrel{\simeq}{\longrightarrow} & \tGr^{W_\sb(N)}_{-k} \,
  \quad k \ge 0 \,.
\end{eqnarray*}
It is always possible to complete $N$ to a \emph{$\fsl_2$--triple}
\begin{equation} \label{E:sl2trip}
  \{ N , Y , N^+ \} \ \subset \ \fg
\end{equation}
where
\begin{eqnarray*}
  {[Y , N ]} & = & -2 N \,,\\
  {[Y,N^+]} & = & 2 N^+ \,,\\
   {[N^+ , N]} & = & Y \,.
\end{eqnarray*}
The span $\fa$ of \eqref{E:sl2trip} is a three dimensional semisimple subalgebra (TDS) of $\fg$ that is isomorphic to $\fsl_2$.  Denoting by $V_\ell = \{ v \in V \ : \ Y v = \ell v \}$ the $\ell$--weight space for the semisimple action of $Y$, we have 
\begin{equation} \label{E:Wk(N)}
  W_k(N) \ = \ \bigoplus_{\ell \le k} V_\ell 
  \tand
  \tGr^{W_\sb(N)}_k \ \simeq \ V_k \,.
\end{equation}

The \emph{primitive spaces} are defined, for $k \ge 0$, by
\[
  \tGr^{W_\sb(N)}_{k,\tprim} \ = \ \tker \left\{ 
    N^{k+1} : \tGr^{W_\sb(N)}_{k} \to \tGr^{W_\sb(N)}_{-k-2}
  \right\} \,.
\]
Decomposing $V$ under the action of the TDS, the primitive spaces are the highest weight spaces.  There are nondegenerate bilinear forms, of parity $k$, 
\[
  Q_k : \tGr^{W_\sb(N)}_{k} \,\times\,\tGr^{W_\sb(N)}_{k} \ \to \ \bR
\]
defined by 
\[
  Q_k(v,w) \ = \ \epsilon_k Q(v , N^k w)
\]
where $\epsilon_k = \pm1$.

The basic result \cite{MR840721} is 
\begin{equation} \label{E:CKS}
\begin{array}{c}
\hbox{\emph{$(F^\sb,N)$ is a nilpotent orbit if and only if}} \\
\hbox{\emph{$(V,W_\sb(N),F^\sb)$ is a polarized mixed Hodge structure.}}
\end{array}
\end{equation}
The filtration $F^\sb$ induces a Hodge structure of weight $k$ on $\tGr^{W_\sb(N)}_k$ and $N$ is of Hodge type $(-1,-1)$. The polarization condition means that the summand $\tGr^{W_\sb(N)}_{k,\tprim}$ of $\tGr^{W_\sb(N)}_k$ is polarized by the form $Q_k$.  We shall generally suppress mention of the polarization conditions, which we always take to be understood.

The equivalence relation on limiting mixed Hodge structures is induced by rescaling $F^\sb \mapsto e^{\lambda N} F^\sb$.  The Hodge structures on $\tGr^{W_\sb(N)}_k$ are unchanged, but some of the extension data in the mixed Hodge structure will be altered.

\subsection{Reduced limit period mapping} \label{S:rlpm}

This section is summarizes material developed in \cite[Appendix to Lecture 10]{MR3115136} and \cite{KP2013}; the interested reader should consult those references for additional detail and proofs.

Given a period mapping \eqref{E:Phi1} with extension \eqref{E:Phi3} the limiting mixed Hodge structure given by the point 
\[
  \Phi_e(0) \ \in \ \Gamma_T\backslash B(N)
\]
represents, in a precise sense, the maximal amount of information in the limit of a degenerating family of polarized Hodge structures.  For some purposes the other extreme of describing a minimal amount of information in the limit is useful.  In that direction we will consider two notions of a \emph{reduced limit period mapping}.  For the first we lift \eqref{E:Phi1} to 
$$
  \tilde \Phi : \bH \to D \,.
$$

\begin{definition}[First notion] \label{d:rlpm1}
The \emph{reduced limit period mapping} is defined by 
\[
  \Phi \ \mapsto \ \lim_{\tIm(z) \to \infty} \tilde \Phi(z) \ \in \ \del D \,.
\]
\end{definition}

\noindent The map was introduced in \cite{KP2013} under the term \emph{\naive \ limit}.  It has been further discussed in \cite[Appendix to Lecture 10]{MR3115136} and \cite{KR1}.  If we think of $\tilde\Phi(z) = F^\sb_z$ as a filtration on $V_\bC$, then we will set 
\[
  \lim_{\tIm(z)\to\infty} F^\sb_z \ = \ F^\sb_\infty \,.
\]
The reduced limit period mapping has the properties:
\begin{a_list}
\item
It is the same for $\Phi$ and for the approximating nilpotent orbit \eqref{E:Psi(0)}.  So, without loss of generality, we shall assume that $\tilde\Phi(z) = e^{zN}\cdot F^\sb$.
\item 
It is independent of the lifting $\tilde\Phi$.  In fact, $F^\sb_\infty \in \del D$ is a fixed point of $T$ and the differential
\[
  T_* : T_{F^\sb_\infty}\check D \to T_{F^\sb_\infty}\check D
\]
is the identity.  Equivalently, the vector field on $\check D = G_\bC/P$ defined by $N \in \fg$ vanishes to second order at $F^\sb_\infty$.
\item
The mapping
\begin{subequations} \label{SE:Phi_infty}
\begin{equation}
  \Phi_\infty : B(N) \to \del D
\end{equation}
defined on $\tilde B(N)$ by 
\begin{equation}
  \Phi_\infty(F^\sb,N) \ = \ \lim_{\tIm(z)\to\infty} \tilde\Phi(z) \ = \ F^\sb_\infty
\end{equation}
\end{subequations}
is well--defined on $B(N)$, and it is $Z(N)_\bR$--equivariant.  The image of $B(N)$ lies in a $G_\bR$--orbit in $\del D$
\[
  \Phi_\infty\left( B(N) \right) \ \subset \ 
  \cO_{F^\sb_\infty} \ = \ G_\bR \cdot F^\sb_\infty \,.
\]
\end{a_list}

\begin{definition}[Second notion] \label{d:rlpm2}
Given a Mumford--Tate domain $D \subset \check D$ with boundary component $B(N)$, the  \emph{reduced limit period mapping} is \eqref{SE:Phi_infty}.
\end{definition}

\noindent Note that the first notion (Definition \ref{d:rlpm1}) is the composition $\Phi_\infty\circ\Phi_\mathit{e}(0)$.

\begin{definition}[{\cite{KP2012}}] \label{d:pol}
A $G_\bR$--orbit $\cO \subset \del D$ for a generalized flag domain $D$ is \emph{polarizable relative to $D$} if there is a Mumford--Tate domain structure on $D$ and a nilpotent orbit $(F^\sb,N)$ such that $F^\sb_\infty \in \cO$.  The orbit $\cO$ is \emph{polarizable} if there is a Mumford--Tate domain $D$ with $\cO \subset \del D$ relative to which $\cO$ is polarizable.
\end{definition}

\noindent 
There are examples of generalized flag domains $D$ and $D'$ in a generalized flag variety $\check D$ and a $G_\bR$--orbit $\cO \subset \del D \,\cap\, \del D'$ such that $\cO$ is polarizable relative to $D$, but not relative to $D'$, \cf\cite[Appendix to Lecture 10]{MR3115136}.

Let $\cN  \subset \fg_\bR$ denote the $\tAd(G_\bR)$--orbit of $N$.  Then
$$
  \cO_{F^\sb_\infty} \ = \ \bigcup_{N' \in \cN} \Phi_\infty(B(N')) \,.
$$
In this case we say that \emph{$\cO_{F^\sb_\infty}$ is polarized by $\cN$ relative to $D$}.
In this sense, the $G_\bR$--orbits in $\check D$ separate into those that have Hodge--theoretic significance, meaning that over $\bR$ every point is realized as a reduced limit for some Mumford--Tate domain structure, and those that don't have Hodge--theoretic significance in this sense.

To explain this a bit more, given a Mumford--Tate domain $D$, for every $F^\sb_x \in \check D$ we may consider the intersection
\[
  V_x^{p,q} \ = \ F^p_x \,\cap\, \overline F{}_x^q \,.
\]
The $x$ for which
\[
  F^p_x \,\op\, \overline{F{}_x^{n-p+1}} \ \stackrel{\simeq}{\longrightarrow} \ V_\bC
\]
for $0 \le p \le n$ give Hodge structures, perhaps with indefinite polarizations meaning that the Hermitian forms in the second Hodge--Riemann bilinear relation are nonsingular but may not be positive definite.  These $\cO_x$ are exactly the open $G_\bR$--orbits in $\check D$.  For the lower--dimensional orbits, the $V_x^{p,q}$ lead to mixed Hodge structures, but without the presence of an $N$ whose weight filtration together with the $F^p_x$ give a polarized limiting mixed Hodge structure these seem relatively uninteresting.

In light of the equivalence \eqref{E:CKS} between nilpotent orbits and limiting mixed Hodge structures, one may ask: \emph{What point of $\del D$ does a limiting mixed Hodge structure map to?}  To answer this we first need two general facts about mixed Hodge structures.  Given a mixed Hodge structure $(V,W_\sb,F^\sb)$ there is the canonical \emph{Deligne splitting}
\begin{eqnarray*}
  V_\bC & = & \bigoplus I^{p,q} \,,\\
  \overline{I^{p,q}} & \equiv & I^{q,p} \quad\hbox{mod}\quad W_{p+q-2}
\end{eqnarray*}
with
\begin{equation} \label{E:delWF}
  W_k \ = \ \bigoplus_{p+q\le k} I^{p,q} \tand
  F^p \ = \ \bigoplus_{q\ge p} I^{q,\sb} \,.
\end{equation}
The mixed Hodge structure is split over $\bR$, or \emph{$\bR$--split}, if 
\[
  \overline{I^{p,q}} \ = \ I^{q,p} \,.
\]
In this case $(V,W_\sb,F^\sb)$ is a direct sum over $\bR$ of pure Hodge structures $\op_{p+q=k} \, I^{p,q}$ of weight $k$.

For the second property, canonically associated to a mixed Hodge structure $(V,W_\sb,F^\sb)$ there is an $\bR$--split mixed Hodge structure $(V,W_\sb,\tilde F^\sb)$ given by 
\[
  \tilde F^\sb \ = \ e^{-2\bi \d} \cdot F^\sb
\]
where $\displaystyle \d \in \oplus_{p,q<0} I^{p,q}_\fg$.  Here, $I^{p,q}_\fg$ is the Deligne splitting of the induced limiting mixed Hodge structure $(\fg,W_{\sb,\fg},F^\sb_\fg)$.

If $(V,W_\sb(N),F^\sb)$ is a limiting mixed Hodge structure, then so is $(V,W_\sb(N),\tilde F^\sb)$, and conversely.  Moreover, we have 
\begin{equation}\label{E:F1}
  F^\sb_\infty \ = \ \tilde F^\sb_\infty \,;
\end{equation}
that is, denoting by $B(N)_\bR$ the equivalence classes containing an $\bR$--split limiting mixed Hodge structure, the reduced limit period mapping factors
\begin{equation} \label{E:F2}
\hbox{
  \setlength{\unitlength}{5pt}
  \begin{picture}(17,12)(0,0)
  \put(0,0){$B(N)_\bR$}
  \put(3,7.5){\vector(0,-1){4}}
  \put(0,9){$B(N)$}
  \put(8,1){\vector(2,1){6}}
  \put(8,9.5){\vector(2,-1){6}}
  \put(15,4.5){$\del D$}
  \put(10,9){$\Phi_\infty$}
  \end{picture}
}
\end{equation}

Next, in terms of the Deligne splitting
$$
  V_\bC \ = \ \bigoplus \tilde I^{p,q}
$$
associated with $(V,W_\sb(N),\tilde F^\sb)$ we have 
\begin{equation} \label{E:F3}
  F^p_\infty \ = \ \bigoplus_{q \le n-p} \tilde I^{\sb,q} \,.
\end{equation}
The picture is this
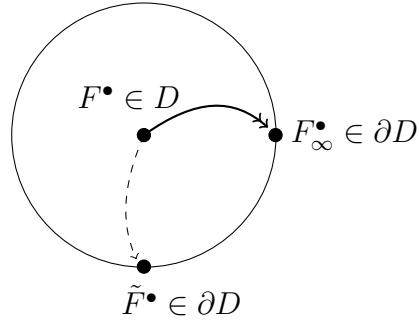
\begin{center}
\begin{tikzpicture}
  \draw (0,0) circle (50pt);
  \filldraw (0,0) circle (2.5pt);
  \filldraw (1.75,0) circle (2.5pt);
  \filldraw (0,-1.75) circle (2.5pt);
  \draw (-0.2,0.5) node {$F^\sb \in D$};
  \draw (2.8,0) node {$F^\sb_\infty \in \del D$};
  \draw (0.5,-2.2) node {$\tilde F^\sb \in \del D$};
  \draw (0,0) [->>, thick] .. controls (0.7,0.5) and (1.2,0.5) .. (1.65,0.1);
  \draw (0,0) [->, dashed] .. controls (-0.3,-0.7) and (-0.3,-1.2) .. (-0.1,-1.65);
\end{tikzpicture}
\end{center}
Starting with $F^\sb \in D$ such that $(V,W_\sb(N),F^\sb)$ is a limiting mixed Hodge structure, we have for the reduced limit period mapping (the solid arrow emanating from $F^\sb \in D$)
\[
  \lim_{y\to\infty} e^{\bi y N } F^\sb \ = \ F^\sb_\infty \,.
\]
We also have the map $F^\sb \to \tilde F^\sb$ (the dashed arrow), and then
\[
  \lim_{y\to\infty} e^{yN} \tilde F^\sb \ = \ F^\sb_\infty \,.
\]
Thus, $F^\sb_\infty$ is reached from $\tilde F^\sb$ by traveling along the real one--parameter subgroup $\texp(\bR N)$ in a $G_\bR$--orbit in $\del D$.

Because of \eqref{E:F1} and the subsequent factorization \eqref{E:F2}, henceforth, unless mentioned otherwise, we shall adopt the 

\begin{convention}
We shall assume that a limiting mixed Hodge structure is $\bR$--split.
\end{convention}

\noindent Because of this we may drop the tilde in \eqref{E:F3} to have 
\begin{equation} \label{E:F4}
  F^p_\infty \ = \ \bigoplus_{q \le n-p} I^{\sb,q} \,.
\end{equation}

\subsection{Reduced limit period mapping for $(\fg,Q_\fg,F^\sb_\fg)$} 

The limiting mixed Hodge structure $(V,W_\sb(N),F^\sb)$ determines a limiting mixed Hodge structure $(\fg , W_\sb(N)_\fg , F^\sb_\fg)$ by
\begin{eqnarray*}
  F^p_\fg & = & 
  \{ \xi \in \fg_\bC \ : \ \xi(F^q) \subset F^{p+q} \ \forall \ q \} \,,\\
  W_\ell(N)_\fg & = & 
  \{ \xi \in \fg_\bR \ : \ 
  \xi(W_m(N)) \subset W_{m+\ell}(N) \ \forall \ m \} \,.
\end{eqnarray*}  
As above $V_\bC = \op I^{p,q}$ will denote the Deligne splitting of $V_\bC$, and 
\begin{equation}\label{E:Del_g}
  \fg_\bC \ = \ \bigoplus I^{p,q}_\fg
\end{equation}
will denote the Deligne splitting on $\fg_\bC$.  If the initial limiting mixed Hodge structure $(V,W_\sb(N),F^\sb)$ is $\bR$--split, so is the induced limiting mixed Hodge structure $(\fg,W_\sb(N)_\fg,F^\sb_\fg)$ on $\fg$; indeed,
\begin{equation} \label{E:indI}
  I^{p,q}_\fg \ = \ \{ \xi \in \fg_\bC \ : \ 
  \xi(I^{r,s}) \subset I^{p+r,q+s} \ \forall \ r,s \} \,.
\end{equation}
 In particular, \eqref{E:F4} implies
$$
  F^p_{\fg,\infty} \ = \ \bigoplus_{q\le-p} I^{\sb,q}_\fg\,.
$$

\begin{example}[A $G_2$ Mumford--Tate domain] \label{eg:G2}
The exceptional simple Lie group $G_2$ of rank two may be realized as the Mumford--Tate group of a weight 6 Hodge structure with Hodge numbers $\bh = (1,1,1,1,1,1,1) = (1^7)$, \cite[Section 6.1.3]{KP2013}.  The associated Mumford--Tate domain $D$ is an open $G_2(\bR)$--orbit in the flag variety $G_2(\bC)/B$.  Including $D$ (which is polarized by $N=0$) there are four $G_2(\bR)$--orbits that are polarized relative to $D$.   The Deligne splittings $V_\bC = \op I^{p,q}$ (left column) and $\fg_\bC = \op I^{p,q}_\fg$ (right column) for each of these orbits are pictured in Figure \ref{f:G2/B}.  A circled node indicates $i^{p,q} = 2$, and an uncircled node indicates $i^{p,q}=1$.

The boundary $\del D$ contains seven $G_2(\bR)$--orbits, four of them are \emph{not} polarized relative to $D$ \cite[Section 6.1.3]{KP2013}. 
\end{example}
\begin{figure}[!th]
\caption{Deligne splittings for polarized orbits in $G_2(\bC)/B$}
\setlength{\unitlength}{10pt}
\begin{picture}(16,8)(-9,0)
\put(0,0){\vector(0,1){7}} \put(0,0){\vector(1,0){7}}
\multiput(0,6)(1,-1){7}{\circle*{0.3}}
\put(-3,4){$D$}
\end{picture}
\hspace{30pt}
\setlength{\unitlength}{7pt}
\begin{picture}(12,10)(-6,-5)
\put(-6,0){\vector(1,0){12}}
\put(0,-6){\vector(0,1){12}}
\multiput(-5,5)(1,-1){11}{\circle*{0.4}}
\multiput(-1,1)(1,-1){3}{\circle{0.8}}
\end{picture}
\\
\setlength{\unitlength}{10pt}
\begin{picture}(16,8)(-9,0)
\put(0,0){\vector(0,1){7}} \put(0,0){\vector(1,0){7}}
\multiput(0,6)(3,-3){3}{\circle*{0.3}}
\multiput(1,4)(1,1){2}{\circle*{0.3}}
\multiput(4,1)(1,1){2}{\circle*{0.3}}
\put(-9,4){$\tcodim\,\cO=1$}
\end{picture}
\hspace{30pt}
\setlength{\unitlength}{7pt}
\begin{picture}(12,14)(-6,-5)
\put(-6,0){\vector(1,0){12}}
\put(0,-6){\vector(0,1){12}}
\multiput(-3,3)(3,-3){3}{\circle*{0.4}}
\multiput(-5,4)(1,1){2}{\circle*{0.4}}
\multiput(-2,1)(1,1){2}{\circle*{0.4}}
\multiput(1,-2)(1,1){2}{\circle*{0.4}}
\multiput(4,-5)(1,1){2}{\circle*{0.4}}
\multiput(-1,-1)(1,1){3}{\circle*{0.4}}
\put(0,0){\circle{0.8}}
\end{picture}
\\
\setlength{\unitlength}{10pt}
\begin{picture}(16,8)(-9,0)
\put(0,0){\vector(0,1){7}} \put(0,0){\vector(1,0){7}}
\multiput(2,2)(1,1){3}{\circle*{0.3}}
\multiput(0,5)(1,1){2}{\circle*{0.3}}
\multiput(5,0)(1,1){2}{\circle*{0.3}}
\put(-9,4){$\tcodim\,\cO=1$}
\end{picture}
\hspace{30pt}
\setlength{\unitlength}{7pt}
\begin{picture}(12,14)(-6,-5)
\put(-6,0){\vector(1,0){12}}
\put(0,-6){\vector(0,1){12}}
\multiput(-5,5)(10,-10){2}{\circle*{0.4}}
\multiput(-4,1)(1,1){4}{\circle*{0.4}}
\multiput(-1,-1)(1,1){3}{\circle*{0.4}}
\multiput(1,-4)(1,1){4}{\circle*{0.4}}
\put(0,0){\circle{0.8}}
\end{picture}
\\
\setlength{\unitlength}{10pt}
\begin{picture}(16,8)(-9,0)
\put(0,0){\vector(0,1){7}} \put(0,0){\vector(1,0){7}}
\multiput(0,0)(1,1){7}{\circle*{0.3}}
\put(-9,4){Closed orbit $\cO_\mathrm{cl}$}
\end{picture}
\hspace{30pt}
\setlength{\unitlength}{7pt}
\begin{picture}(12,13)(-6,-5)
\put(-6,0){\vector(1,0){12}}
\put(0,-6){\vector(0,1){12}}
\multiput(-5,-5)(1,1){11}{\circle*{0.4}}
\multiput(-1,-1)(1,1){3}{\circle{0.8}}
\end{picture}
\label{f:G2/B}
\end{figure}
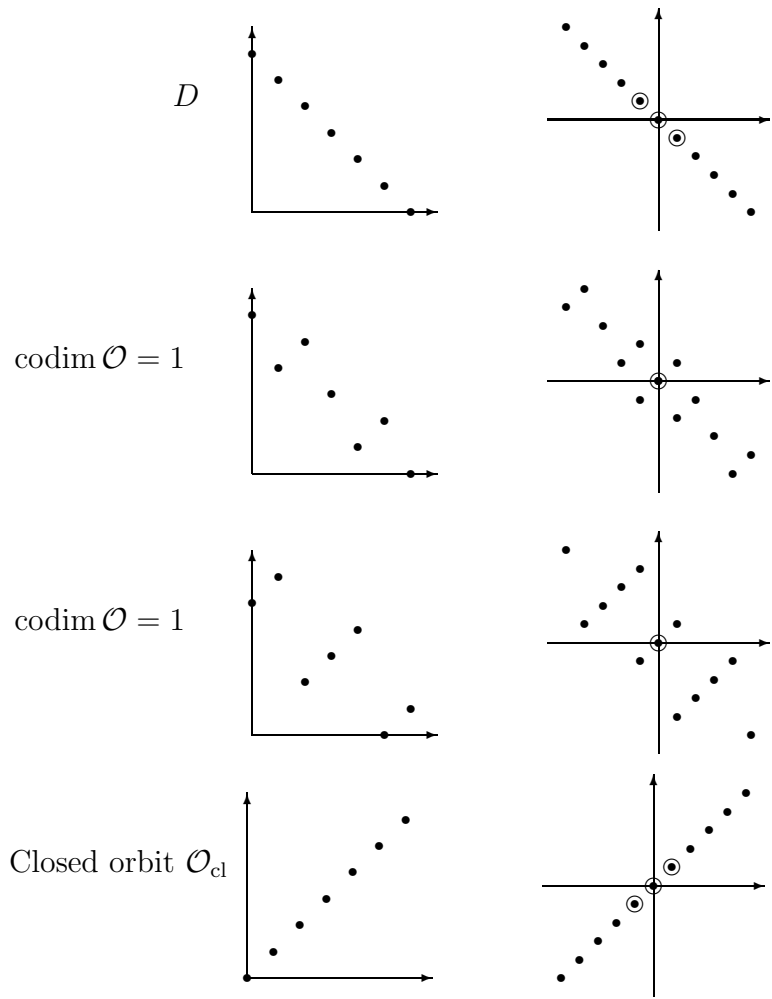

\begin{remark}[Jacobson--Morosov parabolics] \label{R:JMP}
Let $Y$ be the neutral element of an $\fsl_2$--triple \eqref{E:sl2trip} containing the given nilpotent $N$.  In analogy with \eqref{E:Wk(N)} we have 
\[
  W_k(N)_\fg \ = \ \bigoplus_{\ell \le k} \fg_\ell\,,
  \quad\hbox{where}\quad
  \fg_\ell \ = \ \{ \x \in \fg \ : \ [Y,\xi] = \ell \xi \}
\]
is the $\ell$--eigenspace of $Y$.  In particular, 
\[
  W_0(N)_\fg \ = \ \bigoplus_{\ell\ge0} \fg_{-\ell}
\]
is a parabolic subalgebra of $\fg$.  Not every parabolic $\fp \subset \fg$ may be realized as $\fp = W_0(N)_\fg$ for some $N$;\footnote{For example, if $n>2$, then the parabolic $P \subset \tSL_n\bC$ stabilizing a line in $\bC^n$ is not a Jacobson--Morosov parabolic, \cf\cite{MR1251060}.} those that can are \emph{Jacobson--Morosov parabolics}.  When the eigenvalues of $Y$ are even, that is $\fg_{2k+1} = 0$ for all $k \in \bZ$, the we say that $W_0(N)$ is an \emph{even Jacobson--Morosov parabolic}.  A Borel subalgebra $\fb \subset \fg$ may always be realized as an even Jacobson--Morosov parabolic by taking $N$ to be a principal nilpotent \cite{MR1251060}.

Recall that, up to the action of $\tAd(G_\bC)$, the parabolic subalgebras $\fp \subset \fg$ are indexed by (possibly empty) subsets $\Sigma \subset \Delta^+_\mathit{s}$ of the simple roots (Section \ref{S:GFV}).  The subset $\Sigma$ indexing a Jacobson--Morosov parabolic $W_0(N)$ may be determined as follows.  The fact that the eigenvalues of $Y$ are integers implies that $Y$ is a grading element.  (The Jacobson--Morosov parabolic $W_0(N)$ is even if and only if $\a_i(Y)$ is even for all simple roots.)  We may choose $\fh \subset \fb \subset W_0(N)$ so that $\a_i(Y) \ge 0$ for all simple roots $\a_i \in \Delta^+_\mathit{s}$.  (Alternatively, `conjugate' $Y$ so that it lies in a chosen positive Weyl chamber.)  Then 
\[
  W_0(N)_\fg \ = \ \fp_\Sigma \,, \quad\hbox{where}\quad
  \Sigma \ = \ \{ \a_i \ : \ \a_i(Y) = 0 \} \,.
\]
In fact, given the normalization $\a_i(Y) \ge 0$, we have 
\[
  \a_i(Y) \ \in \ \{0,1,2\} \,,
\] 
and the $\tAd(G_\bC)$--orbits of nilpotent $N \in \fg_\bC$ are indexed by the \emph{characteristic vectors} $(\a_1(Y) , \ldots , \a_r(Y))$, \cf \cite{MR1251060, MR0047629_trans, MR0114875}.   
\end{remark}

\section{Extremal degenerations of polarized Hodge structures}

\subsection{Precise definition of extremal degenerations} \label{S:dfn}

\begin{definition} 
Let $(V,Q,F^\sb)$ be a polarized Hodge structure.  The following data constitute a \emph{degeneration of the polarized Hodge structure $(V,Q,F^\sb)$}:
\begin{i_list}
\item
A Mumford--Tate domain $D$ such that $F^\sb = F^\sb_x$ for some $x \in D$.
\item 
A period mapping 
\begin{equation}\label{E:I4'}
  \Phi : \Delta^* \to \Gamma_T\backslash D
\end{equation}
such that for the lift $\tilde\Phi : \bH \to D$ we have 
\[
  \tilde\Phi(z) \ = \ x\quad\hbox{for some } \ z \in \bH \,.
\]
\end{i_list}
\end{definition}

For $t = e^{2\pi\bi z} \in \Delta^*$, we think of $\Phi(t) = F^\sb_t \in \Gamma_T\backslash D$ as defining a one--parameter family of $\Gamma_T$--equivalence classes of polarized Hodge structures whose Mumford--Tate groups are contained in $G \subset \tAut(V,Q)$ and which tend to a ``singular point" as $t \to 0$.  Because of the constraints on the Mumford--Tate groups, this is a more refined notion than just a family of equivalence classes of polarized Hodge structures over the punctured disc.

In this paper two types of limits associated to \eqref{E:I4'} have been discussed.  One is the equivalence class of limiting mixed Hodge structures $(V,W_\sb(N),F^\sb_\tlim)$ viewed as the image of the origin $\Phi_e(0)$ under the extended period map
\[
  \Phi_e : \Delta \to \Gamma_T\backslash \left( D \cup B(N) \right)
\]
(Section \ref{S:pmno}).  The other is the image
\[
  F^\sb_\infty \ \in \ \del D
\]
of the reduced limit period mapping (Section \ref{S:rlpm})
\[
  \Phi_\infty : B(N) \to \del D
\]
applied to $\Phi_e(0) \in B(N)$.  As noted above, $\Phi_\infty$ is invariant under $T$ so that $F^\infty \in \del D$ is well--defined.

The orbits $\cO_x = G_\bR \cdot x$ of the action of $G_\bR$ on the generalized flag variety form a partially ordered set by the relation ``contained in the closure of''
\[
  \cO' \ \prec \ \cO \quad\hbox{if}\quad \cO' \subset \overline \cO \,.
\]
Aside from the open orbits, at one extreme are the real codimension--one orbits; at the other extreme is the unique closed orbit $\cO_\mathrm{cl}$.  In first approximation, an extremal degeneration of a polarized Hodge structure is a degeneration whose reduced limit period $F^\sb_\infty$ lies in one of these two extremes.  However, since there may be no degeneration with $F^\sb_\infty$ in the closed orbit, a refinement of this notion is required.

\begin{definition}
We shall say that a $G_\bR$--orbit $\cO \subset \del D$ is \emph{maximal relative to the Mumford--Tate structure on the generalized flag domain $D$}  if it is polarizable relative to $D$ and if there exists no other orbit $\cO' \precneqq \cO$ that is polarizable relative to $D$.
\end{definition}

\begin{definition} \label{d:extremal}
A degeneration of a polarized Hodge structure is \emph{extremal} if the reduced limit period lies in either a codimension--one orbit (in which case the degeneration is \emph{minimal}), or in an orbit that is maximal relative to the Mumford--Tate domain structure on $D$ (in which case the degeneration is \emph{maximal}).
\end{definition}

Thus extremal degenerations of a polarized Hodge structure represented by a point $x \in D$ are the least and most degenerate the reduced limit period map can realize for a period map \eqref{E:I4'} with $\Phi(z) = x$ for some $z\in \bH$.  There are a few subtleties in the concept.
\begin{a_list}
\item
We cannot talk of the degenerations of a generic $x \in D$.  As a trivial example, the infinitesimal period relation $I \subset TD$ may be zero.  Even if we make the reasonable assumption that $I$ is bracket--generating, we know of no result that ensures the existence of a period mapping \eqref{E:I4'}.
\item
Given \eqref{E:I4'}, we may replace it by the corresponding nilpotent orbit without changing the limit $F^\sb_\infty$.  If \eqref{E:I4'} arises from a family of algebraic varieties, the corresponding nilpotent orbit usually does not; however, this doesn't matter in the sense that the limiting mixed Hodge structure, constructed algebro--geometrically in \cite{MR0429885}, will be the same as that constructed analytically from the nilpotent orbit.  Thus we may speak unambiguously of the degeneration of the Hodge structure on $H^n(X_t , \bQ)_\tprim$ for $X_t = \pi^{-1}(t)$, $t \not=0$, in a family $\sX \stackrel{\pi}{\to} \Delta$.
\end{a_list}

The study of extremal degenerations requires that we understand the $G_\bR$--orbit structure of $\check D = G_\bC/P$.  The necessary material is presented in Sections \ref{S:Lie_str}--\ref{S:c_orb}.

\subsection{$G_\bR$--orbit structure} \label{S:Lie_str}

The representation theory of Lie groups and Lie algebras plays an essential r\^ole in the analysis of the geometry of the compact dual $G_\bC/P$ and its $G_\bR$--orbits.  In this section we outline that structure.

Fix a generalized flag variety $\check D = G_\bC/P$ and a $G_\bR$--orbit $\cO$.  Given a point $x \in \cO$ the Lie algebra $\fp_x$ of the stabilizer $P_x = \tStab_{G_\bC}(x)$ contains a Cartan subalgebra $\fh_x$ that is stable under conjugation \cite[Corollary 2.1.3]{MR2188135}.  In particular, $\fh_x = \fh_x(\bR)\ot\bC$ is defined over $\bR$.

Set $\fp = \fp_x$ and $\fh = \fh_x$.  Choose a Borel subalgebra $\fh \subset \fb \subset \fp$, and let $L\in\fh$ be the grading element \eqref{E:L_p} associated with this triple.  The fact that $\fh$ is closed under conjugation implies that $\fh(\bR) = \fh \cap \fg_\bR$ is Cartan subalgebra of $\fg_\bR$.  It also implies that $\bar\a$ is a root whenever $\a$ is a root.  It follows that 
\begin{center}
  \emph{$\bar L$ is a grading element.}\footnote{We may assume that $\fh \subset \fk$ if and only if $x$ lies in an open $G_\bR$--orbit.  When the orbit is not open, \eqref{E:conjL} will fail.} 
\end{center}
Additionally, the $\fh$--roots of $\fg$ decompose into three types:  we say $\a$ is \emph{real} if $\bar\a =\a$; we say $\a$ is \emph{imaginary} if $\bar\a = -\a$; and we say $\a$ is \emph{complex} otherwise.  

As elements of the Cartan subalgebra $L$ and $\overline L$ define a bigraded eigenspace decomposition
\begin{subequations} \label{SE:matgpq}
\begin{equation} 
  \fg \ = \ \op\,\fg^{p,q}
\end{equation}
given by 
\begin{equation}
  \fg^{p,q} \ = \ \{ \xi \in \fg \ : \ [ L,\xi] = p\xi \,,\ 
  [\overline L,\xi] = q \xi \} \,.
\end{equation}
\end{subequations}
This is the bigrading of \cite[Lemma 3.2]{KP2013}.  If the infinitesimal period relation is bracket--generating, then the filtration corresponding to $x$ is
\begin{equation} \label{E:Fpx}
  F^p_{\fg,x} \ = \ \bigoplus_{q\le -p} \fg^{q,\sb} \,.
\end{equation}
From this point on: 
\begin{equation}\label{E:bg}
\hbox{\emph{We assume that the infinitesimal period relation is bracket--generating.}}\footnote{In the event that the IPR is not bracket--generating one may either modify the definition of $L$ so that \eqref{E:Fpx} holds, as in \cite{KP2013}, or reduce to the case that the IPR is bracket--generating as in \cite[Section 3.3]{schubVHS}.}
\end{equation}

In the event that $F^\sb_{\fg,x} = F^\sb_{\fg,\infty}$ lies in the image of a reduced limit period mapping, the bigrading \eqref{SE:matgpq} is related to the Deligne splitting \eqref{E:Del_g} by 
\begin{equation} \label{E:Ivg}
  I^{p,q}_\fg \ = \ \fg^{q,p} \,,
\end{equation}
\cf\cite{MR3115136, KP2013}.  If $\cO$ is polarized by $N$ relative to $D$ (Definition \ref{d:pol}), then \eqref{E:Ivg} implies
\begin{equation}\label{E:NinI}
  N \ \in \ I^{-1,-1}_\fg \ = \ \fg^{-1,-1} \,.
\end{equation}

Note that $\fg^{p,q}$ is a direct sum of root spaces (and $\fh$ if $p,q=0$).  In particular, 
\begin{equation} \label{E:pq_roots}
  \renewcommand{\arraystretch}{1.4}
  \begin{array}{rcl}
  \fg^{p,q} & = & \displaystyle
  \bigoplus_{\mystack{\a(L)=p}{\bar\a(L)=q}} \fg^\a \,,
  \quad\hbox{if } (p,q)\not=(0,0)\,,\\
  \fg^{0,0} & = & \displaystyle \fh \ \op \ 
  \bigoplus_{\mystack{\a(L)=0}{\bar\a(L)=0}} \fg^\a \,.
  \end{array}
\end{equation}
Observe that 
\begin{equation} \label{E:conjgpq}
  \overline{\fg^{p,q}} \ = \ \fg^{q,p} \,;
\end{equation}
that is,
\begin{center}
\emph{complex conjugation corresponds to reflection about the line $p=q$.}
\end{center}
From \eqref{E:conjgpq} we see that $\fg^{p,q} \op \fg^{q,p}$ is defined over $\bR$.  Let $(\fg^{p,q} \op \fg^{q,p})_\bR = (\fg^{p,q} \op \fg^{q,p}) \,\cap\, \fg_\bR$ denote the real form.  Then the tangent, CR--tangent and normal spaces are given by 
\begin{eqnarray}
  \nonumber
  T_x \cO & = & \bigoplus_{p>0 \ \mathrm{or} \ q>0}
  (\fg^{p,q} \op \fg^{q,p})_\bR  \\
  \label{E:TNsp}
  T^\mathrm{CR}_x\cO & = & \bigoplus_{\mystack{p>0}{q\ge0}} 
  (\fg^{p,-q} \op \fg^{-q,p})_\bR \\
  \nonumber
  N_x\cO & = & \bi\,\bigoplus_{p,q>0} (\fg^{p,q} \op \fg^{q,p})_\bR \,.
\end{eqnarray}
In the case that $F^\sb_{\fg,x} = F^\sb_{\fg,\infty}$, equations \eqref{E:Ivg} and \eqref{E:TNsp} yield \eqref{SE:TN}.  Note also that Lie algebra of the stabilizer $R \subset G_\bR$ is
\[
  \fr \ = \ \left(\fp \op \bar\fp\right)_\bR \ = \  
  \bigoplus_{p,q\ge0} (\fg^{-p,-q}\op\fg^{-q,-p})_\bR \,.
\]
It is convenient to visualize this structure in the $(p,q)$--plane: see Figure \ref{f:pq}.  
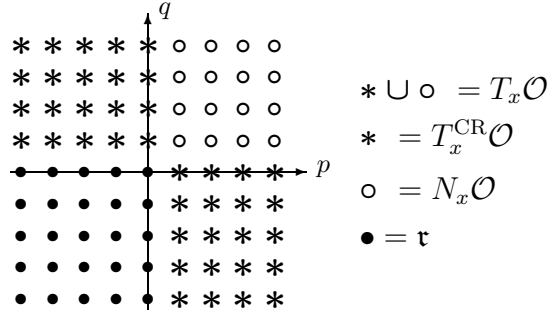
\begin{figure}[!ht]
\caption{Visualization of tangent and normal spaces}
\setlength{\unitlength}{4pt}
\begin{picture}(32,37)(-16,-19)
\begin{footnotesize}
\put(-13,0){\vector(1,0){28}} \put(16,0){$p$}
\put(0,-13){\vector(0,1){28}}\put(1,15){$q$}
\end{footnotesize}
\put(20,7){$\hbox{\bmath{$\ast \cup \circ$}} = T_x\cO$}
\put(20,2.5){$\hbox{\bmath{$\ast$}} = T^\mathrm{CR}_x\cO$}
\put(20,-2.5){$\hbox{\bmath{$\circ$}} = N_x\cO$}
\put(20,-7){$\bullet = \fr$}
\multiput(0,0)(0,-3){5}{\circle*{1.0}}
\multiput(-3,0)(0,-3){5}{\circle*{1.0}}
\multiput(-6,0)(0,-3){5}{\circle*{1.0}}
\multiput(-9,0)(0,-3){5}{\circle*{1.0}}
\multiput(-12,0)(0,-3){5}{\circle*{1.0}}
\multiput(-1,2.12)(-3,0){5}{\large{\bmath{$\ast$}}}
\multiput(-1,5.12)(-3,0){5}{\large{\bmath{$\ast$}}}
\multiput(-1,8.12)(-3,0){5}{\large{\bmath{$\ast$}}}
\multiput(-1,11.12)(-3,0){5}{\large{\bmath{$\ast$}}}
\multiput(2,-0.88)(0,-3){5}{\large{\bmath{$\ast$}}}
\multiput(5,-0.88)(0,-3){5}{\large{\bmath{$\ast$}}}
\multiput(8,-0.88)(0,-3){5}{\large{\bmath{$\ast$}}}
\multiput(11,-0.88)(0,-3){5}{\large{\bmath{$\ast$}}}
\thicklines
\multiput(12,3)(0,3){4}{\circle{1.0}}
\multiput(9,3)(0,3){4}{\circle{1.0}}
\multiput(6,3)(0,3){4}{\circle{1.0}}
\multiput(3,3)(0,3){4}{\circle{1.0}}
\thinlines
\end{picture}
\label{f:pq}
\end{figure}
%

\subsection{Codimension--one orbits} \label{S:codim=1}

Note that, if $\cO$ is of codimension--one, then \eqref{E:pq_roots} and \eqref{E:TNsp} imply that there exists a real root $\a$ such that 
\[
  N_x\cO \ = \ \bi\, I^{p,p}_\bR \ = \ \bi\,\fg^{\a}_\bR \,,\quad\hbox{with}\quad p = \a(L) \,.
\]
The assumption \eqref{E:bg} forces $p = 1$ \cite[Corollary 4.4]{KP2013}.

Suppose that $\cO \subset \del D$.  Then for a suitably scaled root vector $N^+ \in \fg^\a_\bR$ (which is necessarily nilpotent) we have 
\[
  e^{\bi N^+} \cdot F^\sb_x \ \in \ D \,.
\]  
We may complete $N^+$ to an $\fsl_2$--triple \eqref{E:sl2trip} with $N \in \fg^{-\a}_\bR$.  Note that the $\fsl_2 = \tspan\{ N , Y , N^+\}$ gives an $\tSL_2(\bC)$--homogeneous embedding of $\bP^1 = \bC\bP^1$ into $\check D$ with $\bH = \bP^1 \cap D$ and $\bR\bP^1 = \bP^1 \cap \cO$.  In particular, 
\[
  \lim_{z\to\infty} e^{z N } \cdot \left( e^{\bi N^+} \cdot F^\sb_x \right) 
  \ \in \ \cO \,.
\]
It now follows from Definition \ref{d:pol} that we have recovered \cite[Proposition 5.16]{KP2013}:
\begin{quote}
\emph{Every codimension--one orbit $\cO \subset \del D$ is polarizable relative to $D$ by a root vector $N \in \fg^{-\a}_\bR$.}
\end{quote}

Moreover, in this case 
\[
  L \,+\, \overline L \ = \ Y \,,
\]
and the Deligne splitting $\fg_\bC =\op I_\fg^{p,q}$ is \emph{explicitly} given by \eqref{E:Ivg} as an eigenspace decomposition
\begin{equation} \label{E:Icodim1}
  I_\fg^{p,q} \ = \ \{ \xi \in \fg_\bC \ : \ [L,\xi] = q\xi \,,\ 
   [Y,\xi] = (p+q) \xi \} \,.
\end{equation}
Moreover, in the event that the Deligne splitting on $\fg_\bC$ is induced from one on $V_\bC$ as in \eqref{E:indI}, we have
\[
  I^{p,q} \ = \ \{ v \in V_\bC \ : \ L(v) = qv  \,,\ 
   Y(v) = (p+q) v \} \,.
\]
These decompositions are in practice straightforward to compute, \cf\cite{KR1} or Example \ref{eg:F4}.  

\begin{remark} \label{R:matpt}
Each $G_\bR$--orbit in $\check D$ contains a distinguished set of ``Matsuki points'' (Section \ref{S:mat_pt}).  The open orbit $D$ is related to the codimension--one orbits $\cO \subset \del D$ by an application of a Cayley transform to a Matsuki point: Given a Matsuki point $x_0 \in D$ one may apply a Cayley transform $\bc_\a \in G_\bC$ to obtain a Matsuki point $x = \bc_\a(x_0)$ in the codimension one orbit $\cO$, \cite{MR3115136, KP2013, KR1}.
\end{remark}

\subsection{Orbit dimensions} \label{S:orb_dims}

Fix a maximal, connected compact Lie subgroup $K_\bR \subset G_\bR$.  Let $\fk_\bR \subset \fg_\bR$ denote the Lie algebra of $K_\bR$, and let 
\[
  \fg_\bR \ = \ \fk_\bR \ \op \ \fk_\bR^\perp \,.
\]
denote the Cartan decomposition.  The complexification will be written as
\[
  \fg_\bC \ = \ \fk_\bC \ \op \ \fk^\perp_\bC \,.
\]
Let $\theta$ denote the Cartan involution
\begin{equation} \label{E:theta}
  \left.\theta\right|_{\fk_\bC} \ = \ \one \tand
  \left.\theta\right|_{\fk^\perp_\bC} \ = \ -\one \,.
\end{equation}
The dimensions of the $G_\bR$ and $K_\bR$--orbits through $x$ are given by certain $\fh$--root ``counts,'' \cf\eqref{E:dimGRorb} and \eqref{E:dimKRorb} below.  A comparison of these root counts yields a characterization (Lemma \ref{L:c_orb}) of the unique closed $G_\bR$--orbit in $\check D$.

\subsubsection{Matsuki points} \label{S:mat_pt}

We say that a point $x \in \check D = G/P$ is a \emph{Matsuki point} if the Lie algebra $\fp_x$ of the stabilizer $P_x = \tStab_{G_\bC}(x)$ contains a Cartan subalgebra $\fh_x$ that is stable under conjugation and the Cartan involution \eqref{E:theta}.  This implies that the $K_\bR$--orbit is equal to the intersection of the $K_\bC$--orbit with the $G_\bR$--orbit:  $K_\bR \cdot x = (G_\bR \cdot x) \,\cap\, (K_\bC \cdot x)$, \cf\cite[Chapter 8]{MR2188135}.  That is, the $G_\bR$--orbit and the $K_\bC$--orbit are Matsuki dual.  Every $G_\bR$--orbit (resp. $K_\bC$--orbit) contains a Matsuki point.

The fact that $\fh$ is closed under the Cartan involution implies that $\theta\a$ is a root whenever $\a$ is a root.  Moreover, 
\begin{equation} \label{E:ci}
  -\a = \theta\bar\a = \overline{\theta\a} \,,
\end{equation}
and the complex roots appear in quartets 
$$
  \left\{ \a \,,\, \bar\a \,,\, \theta\a \,,\, \theta\bar\a = \overline{\theta\a} \right\} \ = \ 
  \left\{ \pm \a \,,\, \pm \bar\a \right\} \,.
$$
Note that \eqref{E:ci} implies that the bigrading \eqref{SE:matgpq} satisfies
\begin{equation} \label{E:thetagpq}
  \theta(\fg^{p,q}) \ = \ \fg^{-q,-p} \,;
\end{equation}
that is,
\begin{center}
\emph{the Cartan involution corresponds to reflection about the line $p=-q$.}
\end{center}

\subsubsection{The $G_\bR$--orbit} \label{S:GRorb}

Let $\cO = G_\bR \cdot x$ be the $G_\bR$--orbit of $x$.  Then 
$$
  \tdim_\bR\cO \ = \ \tdim_\bR\,\fg_\bR \,-\, \tdim_\bR\,\fg_\bR\cap\fp \,.
$$
Note that 
$$
  \fg_\bR \cap \fp \ = \ (\fp\cap\bar\fp)_\bR \ = \ \fh(\bR) \ \op \ 
  \sum_{\a(L),L(\bar\a)\le0} (\fg^\a + \fg^{\bar\a})_\bR\,.
$$
The sum above is over roots in (closed) lower--left quadrant
$$
  \Delta( \le0,\le0 ) \ = \ 
  \{ \a \in \Delta \ : \ \a(L),L(\bar\a)\le0 \} \,,
$$
\cf Figure \ref{f:Ds}.  Set
$$
  \Delta(\cO) \ = \ \Delta \backslash \Delta( \le0,\le0 ) \,.
$$
Then the (real) tangent space
$$
  T_x \cO \ \simeq \ \fg_\bR / \fp 
  \ \simeq \ \sum_{\a\in\Delta(\cO)} (\fg^\a + \fg^{\bar\a})_\bR 
$$
as a vector space, and we conclude that
\begin{equation} \label{E:dimGRorb}
 \tdim_\bR\cO \ = \ \half\,
 \left|\left\{ \left.\a \in \Delta(\cO) \ \right| \ \a\not=\bar\a \right\}\right|
 \ + \ \left|\left\{ \left.\a \in \Delta(\cO) \ \right| \ \a=\bar\a\right\}\right| \,.
\end{equation}

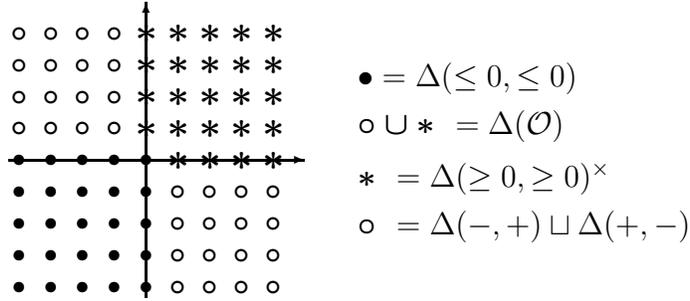
\begin{figure}[!ht]
\caption{Visualization of root sets}
\setlength{\unitlength}{4pt}
\begin{picture}(52,37)(-16,-19)
\put(-13,0){\vector(1,0){28}}
\put(0,-13){\vector(0,1){28}}
\put(20,7){$\bullet = \Delta( \le0,\le0 )$}
\put(20,2.5){$\hbox{\bmath{$\circ \cup \ast$}} = \Delta(\cO)$}
\put(20,-2.5){$\hbox{\bmath{$\ast$}} = \Delta(\ge0,\ge0)^\times$}
\put(20,-7){$\hbox{\bmath{$\circ$}} = \Delta(-,+)\sqcup\Delta(+,-)$}
\multiput(0,0)(0,-3){5}{\circle*{1.0}}
\multiput(-3,0)(0,-3){5}{\circle*{1.0}}
\multiput(-6,0)(0,-3){5}{\circle*{1.0}}
\multiput(-9,0)(0,-3){5}{\circle*{1.0}}
\multiput(-12,0)(0,-3){5}{\circle*{1.0}}
\thicklines
\multiput(-3,3)(0,3){4}{\circle{1.0}}
\multiput(-6,3)(0,3){4}{\circle{1.0}}
\multiput(-9,3)(0,3){4}{\circle{1.0}}
\multiput(-12,3)(0,3){4}{\circle{1.0}}
\multiput(3,-3)(3,0){4}{\circle{1.0}}
\multiput(3,-6)(3,0){4}{\circle{1.0}}
\multiput(3,-9)(3,0){4}{\circle{1.0}}
\multiput(3,-12)(3,0){4}{\circle{1.0}}
\thinlines
\multiput(-1,2.1)(0,3){4}{\large{\bmath{$\ast$}}}
\multiput(2,-0.9)(0,3){5}{\large{\bmath{$\ast$}}}
\multiput(5,-0.9)(0,3){5}{\large{\bmath{$\ast$}}}
\multiput(8,-0.9)(0,3){5}{\large{\bmath{$\ast$}}}
\multiput(11,-0.9)(0,3){5}{\large{\bmath{$\ast$}}}
\thinlines
\end{picture}
\label{f:Ds}
\end{figure}
%

\subsubsection{The $K_\bR$--orbit} \label{S:KRorb}

Likewise, we may compute the dimension of the $K_\bR$--orbit 
$$
  \cO^{K_\bR} \ = \ K_\bR\cdot x
$$
as follows.  Decompose 
$$
  \Delta(\cO) \ = \ \Delta(\ge0,\ge0)^\times \,\sqcup\, \Delta(-,+) 
  \,\sqcup\, \Delta(+,-)
$$ 
into three disjoint sets by
\begin{eqnarray*}
  \Delta(\ge0,\ge0)^\times & = &  
  \{ \a \in \Delta(\cO) \ : \ -\a \in \Delta(\le0,\le0) \}
  \ = \ - \Delta(\ge0,\ge0) \backslash \{(0,0)\} \,,\\
  \Delta(+,-) & = & \{ \a \in \Delta \ : \ \a(L) >0 \,,\ L(\bar\a) < 0 \} \,,\\
  \Delta(-,+) & = & \{ \a \in \Delta \ : \ \a(L) <0 \,,\ L(\bar\a) >0 \} \,.
\end{eqnarray*}
We may visualize $\Delta(+,-)$ and $\Delta(-,+)$ as open quadrants in the $(p,q)$--plane, and $\Delta(\ge0,\ge0)^\times$ as the closed upper--right quadrant minus the origin, \cf Figure \ref{f:Ds}.  

We have an identification $T_o \cO^{K_\bR} \simeq \fk_\bR / \fp$, as vector spaces.  So the dimension of $\cO^{K_\bR}$ is equal to the dimension of $\fk_\bR / \fp$.  Recall that $\fk_\bC \subset \fg_\bC$ is the fix point locus of the Cartan involution $\theta : \fg_\bC \to \fg_\bC$.  In particular, $\fk_\bC \ = \ \{ x + \theta x \ : \ x \in \fg_\bC \}$.   Then, making use of \eqref{E:conjgpq} and \eqref{E:thetagpq}, we have
\begin{equation} \label{E:dimKRorb}
\renewcommand{\arraystretch}{1.3}
\begin{array}{rcl}
 \tdim\,\cO^{K_\bR} & = &  
 \left|\left\{ \a \in \Delta(\cO) \ : \ \a = \bar\a \right\}\right|
  \ + \ 
 \left|\left\{ \left.\a \in \Delta(-,+) \ \right| 
               \ \a=\theta\a\right\}\right| \\
 & & + \ \half\,
 \left|\left\{ \left.\a \in \Delta(-,+) \ \right| 
               \ \a\not=\theta\a \right\}\right| \\
 & & + \ 
 \half\,
 \left|\left\{ \left.\a \in \Delta(\le0,\le0)^\times \ \right| 
               \ \a\not=\bar\a \right\}\right|  \,.
\end{array}
\end{equation}

\subsection{Characterization of closed orbits} \label{S:c_orb}

The flag manifold $\check D = G/P$ contains a unique closed $G_\bR$--orbit \cite[Theorem 3.3]{MR0251246}, which is contained in the closure of every other $G_\bR$--orbit \cite[Corollary 3.4]{MR0251246}.  It follows that the closed $G_\bR$--orbit is the unique orbit for which $\cO = \cO^{K_\bR}$.  Therefore $\cO$ is closed if and only if the two dimensions \eqref{E:dimGRorb} and \eqref{E:dimKRorb} are equal.  Whence we obtain the following characterization of the closed $G_\bR$--orbit.

\begin{lemma} \label{L:c_orb}
Let $x \in \check D$ be a Matsuki point, and let $\fh$ be a Cartan subalgebra of $\fg$ contained in the Lie algebra $\fp$ of the stabilizer $P = \tStab_G(x)$ that is stable under complex conjugation and the Cartan involution.  Then the $G_\bR$--orbit $\cO$ through $x$ is closed (equivalently, $\cO = \cO^{K_\bR}$) if and only if the $\fh$--roots satisfy
\begin{equation} \label{E:c_orb}
  \theta\a \ = \ \a \quad\hbox{for all} \quad \a \,\in\, \Delta(-,+) \,.
\end{equation}
That is, by \eqref{E:ci}, all the roots of $\Delta(-,+)$ are imaginary and compact.
\end{lemma}

\noindent Keeping in mind that $\a \mapsto \bar\a$ is a bijection between $\Delta(+,-)$ and $\Delta(-,+)$, we see that \eqref{E:c_orb} holds if and only if
$$
  \theta\a \ = \ \a \quad\hbox{for all} \quad \a \,\in\, \Delta(+,-) \,.
$$
Whence the lemma may be visualized as in Figure \ref{f:c_orb}.a, where the potentially nonzero $\fg^{p,q}$ are indicated by a node at the point $(p,q)$.
\begin{figure}[!h]
\begin{footnotesize}
\caption{Visualization of (polarized) closed $G_\bR$--orbit}
\setlength{\unitlength}{3.8pt}
\begin{picture}(32,37)(-16,-19)
\put(-13,0){\vector(1,0){28}} \put(16,0){$p$}
\put(0,-13){\vector(0,1){28}}\put(1,15){$q$}
\multiput(0,-12)(0,3){9}{\circle*{1.0}}
\multiput(3,0)(0,3){5}{\circle*{1.0}}
\multiput(6,0)(0,3){5}{\circle*{1.0}}
\multiput(9,0)(0,3){5}{\circle*{1.0}}
\multiput(12,0)(0,3){5}{\circle*{1.0}}
\multiput(-3,0)(0,-3){5}{\circle*{1.0}}
\multiput(-6,0)(0,-3){5}{\circle*{1.0}}
\multiput(-9,0)(0,-3){5}{\circle*{1.0}}
\multiput(-12,0)(0,-3){5}{\circle*{1.0}}
\multiput(-12,12)(3,-3){9}{\circle*{1.0}}
\put(-3,-17){(\ref{f:c_orb}.a)}
\end{picture}
\hspace{15pt}
\begin{picture}(32,37)(-16,-19)
\put(-13,0){\vector(1,0){28}} \put(16,0){$p$}
\put(0,-13){\vector(0,1){28}}\put(1,15){$q$}
\multiput(0,-6)(0,3){5}{\circle*{1.0}}
\multiput(3,0)(0,3){4}{\circle*{1.0}}
\multiput(6,0)(0,3){5}{\circle*{1.0}}
\multiput(9,3)(0,3){4}{\circle*{1.0}}
\multiput(12,6)(0,3){3}{\circle*{1.0}}
\multiput(-3,0)(0,-3){4}{\circle*{1.0}}
\multiput(-6,0)(0,-3){5}{\circle*{1.0}}
\multiput(-9,-3)(0,-3){4}{\circle*{1.0}}
\multiput(-12,-6)(0,-3){3}{\circle*{1.0}}
\multiput(-6,6)(3,-3){5}{\circle*{1.0}}
\put(-12,12){\circle*{1.0}}\put(12,-12){\circle*{1.0}}
\put(-3,-17){(\ref{f:c_orb}.b)}
\put(0,6){\circle{2}}
\put(6,0){\circle{2}}
\put(6,12){\circle{2}}
\put(12,6){\circle{2}}
\end{picture}
\label{f:c_orb}
\end{footnotesize}
\end{figure}
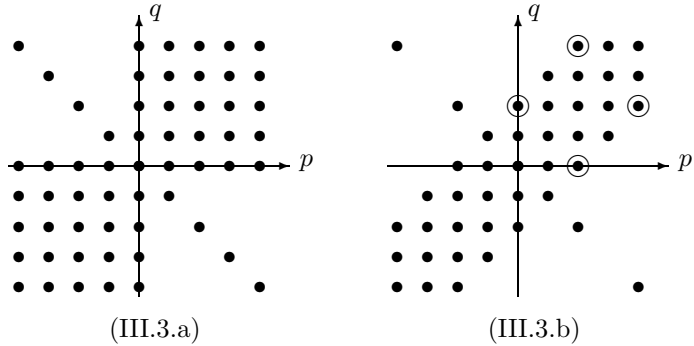
Equivalently, we will say that the bigraded $(L,\overline L)$--eigenspace decomposition \eqref{SE:matgpq} must be of the form depicted in Figure \ref{f:c_orb}.a.

\begin{remark*}
We remark that Figure \ref{f:c_orb}.a is necessary, but not sufficient, for an orbit to be closed.  That is, there are examples in which the bigrading takes this form, but \eqref{E:c_orb} fails.
\end{remark*}

\subsection{Totally real orbits} \label{S:TR}

A $G_\bR$--orbit $\cO \subset \check D$ is \emph{totally real} if the Cauchy--Riemannian tangent bundle $T^\mathrm{CR}\cO$ is trivial.  Visually, the orbit is totally real if there are no {\boldmath$\ast$\unboldmath}'s in Figure \ref{f:pq}.  By \eqref{E:TNsp} and Lemma \ref{L:c_orb} any totally real orbit is necessarily the unique closed orbit.  Moreover, recalling that 
\[
  \fp \ = \ \bigoplus_{p\ge0}\fg^{-p,\sb} \ = \ \fg^{\le0,\sb} \,,
\]
the first half of Theorem \ref{T:I6} is now evident:

\begin{theorem} \label{T:TR}
The following are equivalent:
\begin{i_list_emph}
\item 
The orbit $\cO_\mathrm{cl}$ is totally real.
\item 
The real dimension of the closed orbit is the complex dimension of the compact dual: $\tdim_\bR \cO_\mathrm{cl} = \tdim_\bC \check D$.
\item 
The stabilizer $P$ is $\bR$--split and $\cO_\mathrm{cl} = G_\bR/P_\bR$.
\end{i_list_emph}
\end{theorem}

\section{Minimal degenerations of polarized Hodge structures}

Let $D$ be a Mumford--Tate domain.  In Section \ref{S:codim=1} we saw that every codimension--one $G_\bR$--orbit $\cO \subset \del D$ is polarizable relative to $D$, and gave explicit descriptions of Deligne splittings as eigenspace decompositions.  In Section \ref{S:min_pd} we specialize to the case that $D$ is a period domain and give a refined description.
In Example \ref{eg:F4} we give an example illustrating the (in general) more complicated structure in the Mumford--Tate domain case.

\subsection{Minimal degenerations in period domains} \label{S:min_pd}

\begin{definition*}
A \emph{type I basic boundary component} for weight $n$ consists of limiting mixed Hodge structures that are of the form
\begin{eqnarray*}
  H^{n-1-2k}(k+1) & \stackrel{N}{\longrightarrow} & H^{n-1-2k}(k) \\
  & H^n & 
\end{eqnarray*}
for some $k$ with $0 \le 2k \le n-1$.  A \emph{type II basic boundary component} for weight $n=2m$ consists of limiting mixed Hodge structures that are of the form
\begin{eqnarray*}
  H^0(m+1) & \stackrel{N}{\longrightarrow} \ H^0(m) \ 
  \stackrel{N}{\longrightarrow} & H^{0}(m-1) \\
  & H^n\,. & 
\end{eqnarray*}
Any $G_\bR$--orbit $\cO \subset \del D$ containing the image of a basic boundary component $B(N)$ under the reduced limit period mapping is a \emph{basic boundary orbit}.
\end{definition*}

\noindent The Deligne splittings of the limiting mixed Hodge structures in a basic boundary component are illustrated in Figure \ref{f:basicBC}.

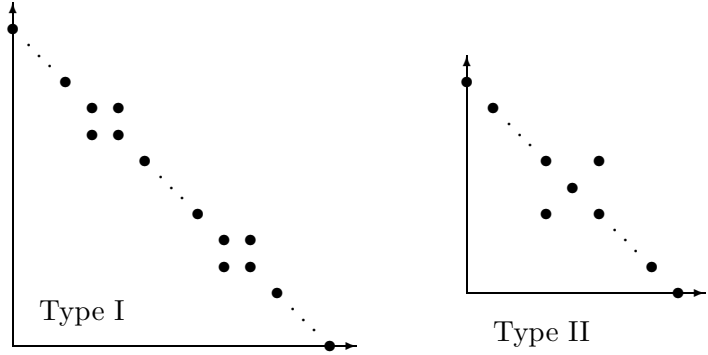
\begin{figure}
\caption{Limiting mixed Hodge structures in basic boundary component}
\setlength{\unitlength}{10pt}
\begin{picture}(13,14)
\put(0,0){\vector(0,1){13}}
\put(0,0){\vector(1,0){13}}
\put(0,12){\circle*{0.4}}
\multiput(0.6,11.4)(0.4,-0.4){3}{\circle*{0.15}}
\put(2,10){\circle*{0.4}}
\put(3,8){\circle*{0.4}} 
\put(4,9){\circle*{0.4}}
\put(3,9){\circle*{0.4}} 
\put(4,8){\circle*{0.4}}
\multiput(5.6,6.4)(0.4,-0.4){3}{\circle*{0.15}}
\put(5,7){\circle*{0.4}}
\put(7,5){\circle*{0.4}}
\put(8,3){\circle*{0.4}}
\put(9,4){\circle*{0.4}}
\put(8,4){\circle*{0.4}}
\put(9,3){\circle*{0.4}}
\put(10,2){\circle*{0.4}}
\multiput(10.6,1.4)(0.4,-0.4){3}{\circle*{0.15}}
\put(12,0){\circle*{0.4}}
\put(1,1){\small{Type I}}
\end{picture}
\hspace{30pt}
\setlength{\unitlength}{10pt}
\begin{picture}(9,10)(0,-2)
\put(0,0){\vector(0,1){9}}
\put(0,0){\vector(1,0){9}}
\put(0,8){\circle*{0.4}}\put(1,7){\circle*{0.4}}
\multiput(1.6,6.4)(0.4,-0.4){3}{\circle*{0.15}}
\put(3,5){\circle*{0.4}}
\put(3,3){\circle*{0.4}} 
\put(4,4){\circle*{0.4}}
\put(5,3){\circle*{0.4}} 
\put(5,5){\circle*{0.4}}
\multiput(5.6,2.4)(0.4,-0.4){3}{\circle*{0.15}}
\put(7,1){\circle*{0.4}}
\put(8,0){\circle*{0.4}}
\put(1,-1.8){\small{Type II}}
\end{picture}
\label{f:basicBC}
\end{figure}

\begin{theorem} \label{T:basicBC}
Given a period domain $D \subset \check D$, the codimension--one boundary orbits are basic.  Additionally, the nonzero off--diagonal $I^{p,q}$ all have dimension one.  Moreover, the type II basic boundary components occur only for even weight Hodge structures with $h^{m,m}$ odd.
\end{theorem}

\begin{proof}
Any codimension--one orbit $\cO \subset \tbd(D)$ contains the image under the reduced limit period mapping of a boundary component $B(N)$; moreover we may take $N$ to be a root vector (Section \ref{S:codim=1}).  Therefore, for an appropriate choice of basis $\{ e_i \}_{i=1}^d$ of $V$ we may take $N$ to be of one of the following normal forms.  Let $\{ e^i \}_{i=1}^d$ be the dual basis of $V^*$, and let $\{ e_j^i = e_j \ot e^i \}_{i,j=1}^d$ be the corresponding basis of $\tEnd(V)$.

Suppose the weight $n$ is odd, so that $d = 2c$.  The nilpotent $N$ is of one of the following forms
\begin{equation}\label{E:rtN1}
  e^i_j - e^{c+j}_{c+i} \,,\quad
  e^{c+i}_j + e^{c+j}_i \,,\quad
  e^i_{c+j} + e^j_{c+i} \,,\quad
  e^i_{c+i} \,,\quad e^{c+i}_i \,,\quad\hbox{with}\quad i,j \le c \,.
\end{equation}
If the weight $n$ is even, and the dimension $d = 2c$ of $V$ is even, then the nilpotent $N$ is of one of the following forms
\begin{equation}\label{E:rtN2}
  e^i_j - e^{c+j}_{c+i} \,,\quad
  e^{c+i}_j - e^{c+j}_i \,,\quad
  e^i_{c+j} - e^j_{c+i} \,,\quad\hbox{with}\quad i,j \le c \,.
\end{equation}
In the event that the dimension $d = 2c+1$ is odd, then the nilpotent $N$ is either of one of the three forms in \eqref{E:rtN2}, or of one of the following forms
\begin{equation}\label{E:rtN3}
  e^d_i - e^{c+i}_d \,,\quad
  e^d_{c+i} - e^i_d \,,\quad\hbox{with}\quad i \le c \,.
\end{equation}
It is clear from these explicit expressions that:
\begin{bcirclist}
\item If $N$ is of one of the forms in \eqref{E:rtN1} or \eqref{E:rtN2}, then $N^2 = 0$ and $N$ has rank at most two, so that the boundary component is type I basic.
\item
If $N$ if of one of the forms in \eqref{E:rtN3}, then $N^3=0$ while $N^2 \not= 0$, and $N$ has rank two, so that the boundary component is type II basic.  Moreover, in this case the weight $n = 2m$ is even, and $d = 2c+1 = \sum h^{p,q}$ forces $h^{m,m}$ to be odd.
\end{bcirclist}
\end{proof}



\subsection{Minimal degenerations in Mumford--Tate domains} \label{S:min_mtd}

In the case of Mumford--Tate domains the $I^{p,q}$ corresponding to a minimal degeneration may be more complicated that those arising in the case of period domains (Figure \ref{f:basicBC}).  Illustrations of this are given in Examples \ref{eg:G2} (see, in particular, the rows 2 and 3 of Figure \ref{f:G2/B}) and \ref{eg:F4}.  What we can say in general is

\begin{proposition} \label{P:min_mtd}
Fix a flag domain $G_\bR/R$ admitting the structure of a Mumford--Tate domain $D$.  Then there exists a Hodge representation $(V,Q)$ of $G$ that realizes $G_\bR/R$ as Mumford--Tate domain $D$ and with the property that each codimension--one $G_\bR$--orbit $\cO \subset \del D \subset \check D$ is polarized by a limiting mixed Hodge structure $(V,W_\sb(N),F^\sb)$ such that 
\[
  N^3 \ = \ 0
  \quad\hbox{as an element of} \quad
  \fg_\bR \ \subset \ \tEnd(V_\bR,Q) \,.
\]
In fact, if $G_\bC \not= G_2$ is simple, then we may take $V$ to be the adjoint representation $\fg$; if $G_\bC = G_2(\bC)$, then we may take $V$ to be the standard representation $V_\bR = \bR^7$.
\end{proposition}

\noindent The proof follows Example \ref{eg:F4}.

\begin{example}[An $F_4$ Mumford--Tate domain] \label{eg:F4}
Consider the case that the Mumford--Tate group is the exceptional simple Lie group of rank four.  The smallest representation $V_\bC$ of $G_\bC = F_4(\bC)$ is of dimension 26.  There exist real forms $V_\bR$ and $G_\bR = \mathrm{F\,I}$, the latter of real rank four and with maximal compact subalgebra $\fk_\bR = \fsp(3)\op\fsu(2) \subset \fg_\bR$, with the property that $V_\bR$ admits the structure of a $G_\bR$--Hodge representation with Hodge numbers 
\[
  \bh \ = \ ( 1^4 , 2^9, 1^4 ) \,.
\]  
In this case the compact dual $G_\bC/B$ is the flag variety and has dimension 24.  The Deligne splittings associated with the four codimension one boundary orbits are depicted in Figure \ref{f:F4}; these are computed by \eqref{E:Icodim1}.  Each circled node has dimension two, the remainder have dimension one.    Note that in the first row of Figure \ref{f:F4} we have $N^2 = 0$; in the second row $N^2 \not=0$ and $N^3 = 0$.
\begin{figure}[!ht]
\caption{Deligne splittings associated with codimension one boundary orbits for $D \subset F_4(\bC)/B$.}
\setlength{\unitlength}{6pt}
\begin{picture}(18,19)
\put(0,0){\vector(1,0){18}} \put(0,0){\vector(0,1){18}}  \put(0,16){\circle*{0.5}} \put(1,15){\circle*{0.5}} \put(2,14){\circle*{0.5}} \put(3,12){\circle*{0.5}} \put(4,11){\circle*{0.5}} \put(4,13){\circle*{0.5}} \put(5,10){\circle*{0.5}} \put(5,12){\circle*{0.5}} \put(6,10){\circle*{0.5}} \put(6,11){\circle*{0.5}} \put(7,9){\circle*{0.5}} \put(8,8){\circle*{0.5}} \put(9,7){\circle*{0.5}} \put(10,5){\circle*{0.5}} \put(10,6){\circle*{0.5}} \put(11,4){\circle*{0.5}} \put(11,6){\circle*{0.5}} \put(12,3){\circle*{0.5}} \put(12,5){\circle*{0.5}} \put(13,4){\circle*{0.5}} \put(14,2){\circle*{0.5}} \put(15,1){\circle*{0.5}} \put(16,0){\circle*{0.5}}
\multiput(7,9)(1,-1){3}{\circle{1.0}}
\end{picture}
\hspace{15pt}
\begin{picture}(18,19)
\put(0,0){\vector(1,0){18}} \put(0,0){\vector(0,1){18}}  \put(0,16){\circle*{0.5}} \put(1,15){\circle*{0.5}} \put(2,13){\circle*{0.5}} \put(3,14){\circle*{0.5}} \put(4,12){\circle*{0.5}} \put(5,10){\circle*{0.5}} \put(5,11){\circle*{0.5}} \put(6,9){\circle*{0.5}} \put(6,11){\circle*{0.5}} \put(7,9){\circle*{0.5}} \put(7,10){\circle*{0.5}} \put(8,8){\circle*{0.5}} \put(9,6){\circle*{0.5}} \put(9,7){\circle*{0.5}} \put(10,5){\circle*{0.5}} \put(10,7){\circle*{0.5}} \put(11,5){\circle*{0.5}} \put(11,6){\circle*{0.5}} \put(12,4){\circle*{0.5}} \put(13,2){\circle*{0.5}} \put(14,3){\circle*{0.5}} \put(15,1){\circle*{0.5}} \put(16,0){\circle*{0.5}}
\multiput(4,12)(4,-4){3}{\circle{1.0}}
\end{picture}
\vspace{5pt}\\
\begin{picture}(18,19)
\put(0,0){\vector(1,0){18}} \put(0,0){\vector(0,1){18}}  \put(0,16){\circle*{0.5}} \put(1,14){\circle*{0.5}} \put(2,15){\circle*{0.5}} \put(3,12){\circle*{0.5}} \put(4,11){\circle*{0.5}} \put(4,13){\circle*{0.5}} \put(5,11){\circle*{0.5}} \put(5,12){\circle*{0.5}} \put(6,9){\circle*{0.5}} \put(6,10){\circle*{0.5}} \put(7,7){\circle*{0.5}} \put(7,10){\circle*{0.5}} \put(8,8){\circle*{0.5}} \put(9,6){\circle*{0.5}} \put(9,9){\circle*{0.5}} \put(10,6){\circle*{0.5}} \put(10,7){\circle*{0.5}} \put(11,4){\circle*{0.5}} \put(11,5){\circle*{0.5}} \put(12,3){\circle*{0.5}} \put(12,5){\circle*{0.5}} \put(13,4){\circle*{0.5}} \put(14,1){\circle*{0.5}} \put(15,2){\circle*{0.5}} \put(16,0){\circle*{0.5}}
\put(8,8){\circle{1.0}}
\end{picture}
\hspace{15pt}
\begin{picture}(18,19)
\put(0,0){\vector(1,0){18}} \put(0,0){\vector(0,1){18}}  \put(0,15){\circle*{0.5}} \put(1,16){\circle*{0.5}} \put(2,14){\circle*{0.5}} \put(3,13){\circle*{0.5}} \put(4,11){\circle*{0.5}} \put(4,12){\circle*{0.5}} \put(5,10){\circle*{0.5}} \put(5,12){\circle*{0.5}} \put(6,9){\circle*{0.5}} \put(6,11){\circle*{0.5}} \put(7,7){\circle*{0.5}} \put(7,10){\circle*{0.5}} \put(8,8){\circle*{0.5}} \put(9,6){\circle*{0.5}} \put(9,9){\circle*{0.5}} \put(10,5){\circle*{0.5}} \put(10,7){\circle*{0.5}} \put(11,4){\circle*{0.5}} \put(11,6){\circle*{0.5}} \put(12,4){\circle*{0.5}} \put(12,5){\circle*{0.5}} \put(13,3){\circle*{0.5}} \put(14,2){\circle*{0.5}} \put(15,0){\circle*{0.5}} \put(16,1){\circle*{0.5}}
\put(8,8){\circle{1.0}}
\end{picture}
\label{f:F4}
\end{figure}
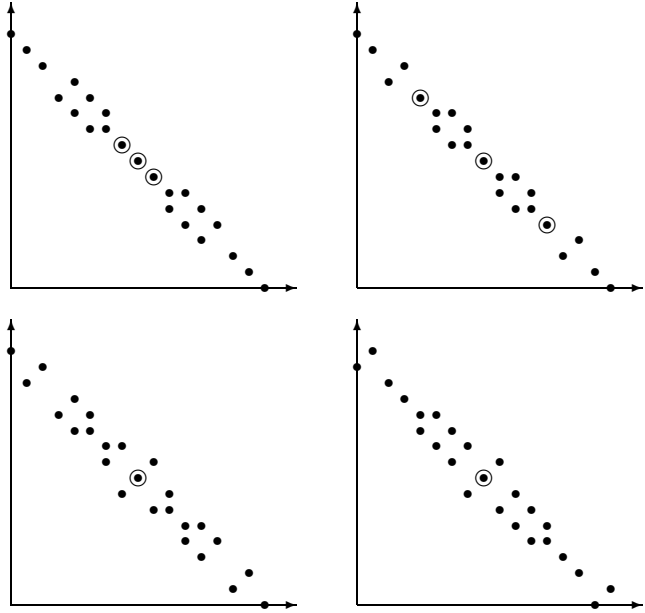
\end{example}

\begin{proof}[Proof of Proposition \ref{P:min_mtd}]
Given a polarized Hodge structure $(V,Q,F^\sb) \in D$, let $(\fg,Q_\fg,F^\sb_\fg)$ denote the induced polarized Hodge structure on $\fg_\bR$ (Section \ref{S:ind2}).  Any codimension--one boundary $G_\bR$--orbit $\cO \subset \del D$ is polarized by a root vector $N \in \fg^{-\a}_\bR$.  It is well--known that root vectors satisfy $N^4 = 0$ as elements of $\tEnd(\fg)$.  (Equivalently, every $\a$--string has length at most four.)  
\begin{bcirclist}
\item
If $\a$ is \emph{not} a root of the exceptional Lie algebra $\fg_2$, then $N^3 = 0$ on $\fg_\bC$ (every $\a$--string has length at most three).
\item
If $\a$ is a root of $\fg_2$, then it is possible that $N^3\not=0$ on $\fg_\bC$; this is illustrated in Figure \ref{f:G2/B} (row 3, column 2).  However, as is also illustrated in Figure \ref{f:G2/B} (rows 2 and 3, column 1), any root vector will satisfy $N^3 = 0$ as an element of $\fg_\bC \subset \tEnd(\bC^7)$ acting on the standard representation.
\end{bcirclist}
\end{proof}

\section{Maximal degenerations of polarized Hodge structures}

\subsection{Polarizable closed orbits} \label{S:PCO}

In the case that the closed orbit is polarizable, Lemma \ref{L:c_orb} and Figure \ref{f:c_orb}.a may be refined.  We now recall Theorem \ref{T:I7} and prove the slightly stronger 

\begin{theorem} \label{T:cp_orb} 
Suppose that the limiting mixed Hodge structure $(\fg,W_\sb(N),F^\sb)$ is sent to the closed orbit under the reduced limit period mapping.  Then Deligne splitting $\fg_\bC = \op I^{p,q}_\fg$ satisfies:  
\begin{subequations} \label{SE:cp}
\begin{eqnarray}
  \label{E:cp1}
  I_\fg^{p,-q} = 0 & \hbox{for all} & p\not=q > 0\,, \\
  \label{E:cp3}
  I_\fg^{p,-p} = 0 & \hbox{for all} & \hbox{odd } \ p \ge 3 \,,\\
  \label{E:cp2}
  I_\fg^{p,q} = 0 & \hbox{for all} & p+q\not=0 \quad \hbox{with} \quad |p-q| > 2 \,.
\end{eqnarray}
\end{subequations}
Moreover, any $N$--string in $\op_{q-p=2}\,\fg^{p,q}$ has length $\equiv 3$ mod $4$.  In particular, if $\tGr_{k,\tprim}^{W_\sb(N)_\fg} \not=0$, then $k$ is even.
\end{theorem}

\noindent The Deligne splitting is pictured in Figure \ref{f:c_orb}.b.  The circled nodes indicate those $I^{p,q}_\fg$ with $p-q = \pm2$ that \emph{may} (but need not) have a non--trivial primitive component.  (There are no constraints on the primitive components with $p-p\not=\pm2$, aside from the obvious $I^{p,-p}_\fg = I^{p,-p}_{\fg,\tprim}$ for all $|p|\ge 2$.)

\begin{proof}
Lemma \ref{L:c_orb} implies \eqref{E:cp1}.  Assertion \eqref{E:cp2} follows from the the property $N : \fg^{p,q} \to \fg^{p-1,q-1}$ and the fact that the $N$--strings are uninterrupted.  Given this, it follows that  
$$
  \Gamma \ = \ \bigoplus_{p\ge2} (\fg^{-p,p}\op\fg^{p,-p} ) \ \subset \ 
  \tGr_{0,\tprim}^{W_\sb(N)_\fg} \,.
$$
Recall (Section \ref{S:NO+MHS}) that $ \tGr_{0,\tprim}^{W_\sb(N)_\fg}$ carries a weight zero Hodge structure polarized by $Q_\fg$.  So given a nonzero $v \in \fg^{-p,p} \subset \Gamma$, we have $(-1)^{p}Q_\fg(v,\bar v) > 0$.  By \eqref{E:pq_roots}, $\Gamma$ is a direct sum of root spaces.  Moreover, Lemma \ref{L:c_orb} asserts that the roots are imaginary and compact.  Recall that $-Q_\fg$ is the Killing form.  Therefore, given a nonzero root vector $v \in \fg^\a \subset \fg^{-p,p}$, the inequality $(-1)^{p}Q_\fg(v,\bar v ) > 0$ implies that $\a$ is compact if $p$ is even and noncompact if $p$ is odd.  Therefore, $p$ is necessarily even; this establishes \eqref{E:cp3}.

Any $N$--string in $\op_{q-p=2}\,\fg^{p,q}$ is necessarily of the form 
$$
  u \,,\ Nu \,,\ N^2 u\,,\ldots,\ N^{2k} u \,,
$$
with $0\not=u \in \fg^{k-1,k+1}$, $N^k u \in \fg^{-1,1}$ and $N^{2k+1}u = 0$.  By the polarization hypothesis, 
$$
  0 \ < \ -Q_\fg( u , N^{2k} \bar u) \ = \ (-1)^{k+1} Q_\fg (N^k u , N^k \bar u) \,.
$$
Lemma \ref{L:c_orb} implies $0\not=v = N^k u \in \fg^{-1,1}$ satisfies $Q_\fg(v,\bar v) >0$.  Therefore $k$ is odd, and the length of the $N$--string is $2k+1 = 4\ell+3 \equiv 3$ mod 4.
\end{proof}

\subsection{The Hodge--Tate case} \label{S:HT}

The polarized limiting mixed Hodge structure $(V,W_\sb(N),F^\sb)$ is \emph{Hodge--Tate} if 
\begin{equation}\label{E:HT}
  I^{p,q} \ = \ 0 \quad\hbox{for all}\quad p\not=q\,.
\end{equation}
Given that the $N$--strings are uninterrupted and centered on the line $p=-q$, we obtain the following corollary to Theorem \ref{T:I7}:

\begin{corollary}\label{C:TR}
If a limiting mixed Hodge structure $(V,W_\sb(N),F^\sb)$ polarizes a totally real (and necessarily closed) $G_\bR$--orbit, then the limiting mixed Hodge structure is of Hodge--Tate type.
\end{corollary}

We will now recall and prove 

\begin{propI5} 
The polarized limiting mixed Hodge structure $(V,W_\sb(N),F^\sb)$ is Hodge--Tate if and only if the induced limiting mixed Hodge structure $(\fg,W_\sb(N)_\fg,F^\sb_\fg)$ is Hodge--Tate.
\end{propI5}

\noindent Together Corollary \ref{C:TR} and Proposition \ref{P:I5} yield the second half of Theorem \ref{T:I6}.  Proposition \ref{P:I5} is a consequence of the following lemma.

\begin{lemma} \label{L:HT}
Fix a real form $\fg_\bR$ and let $\fh \subset \fg_\bC$ be a Cartan subalgebra that is stable under complex conjugation.  Let $L \in \fh$ be any grading element.  Let $V$ be a representation of $\fg_\bC$, and let $\Lambda(V) \subset \wtL$ denote the weights of $V$.  Suppose that $\lambda(L) = \lambda(\bar L)$ for all $\lambda \in \Lambda(V)$.  (Equivalently, $\lambda(L) = \bar\lambda(L)$.)  Then $\mu(L) = \mu(\bar L)$ for the all weights $\mu \in \wtL$ of $\fg_\bC$.
\end{lemma}

\begin{proof}[Proof of Lemma \ref{L:HT}]
The weight lattice $\wtL \subset \fh^*$ is spanned over $\bQ$ by the weights of $V$; that is, $\wtL = \tspan_\bQ\,\Lambda(V)$.
\end{proof}

\begin{proof}[Proof of Proposition \ref{P:I5}]
The nontrivial $I^{p,q}$ are precisely those with $(p,q) = (\lambda(L) , \lambda(\bar L))$ for some $\lambda \in \Lambda(V)$.
\end{proof}

\begin{proof}[Proof of Proposition \ref{P:I9}]
Follows directly from Proposition \ref{P:I5} and Lemma \ref{L:c_orb}.
\end{proof}


There are a number of conditions that a Mumford--Tate domain $D$ must satisfy in order to admit a Hodge--Tate degeneration.  Some of the conditions are Hodge--theoretic in nature (Lemma \ref{L:HT-HN}), while others are representation theoretic (Lemma \ref{L:JMP} and Remark \ref{R:HTinG/B}).\footnote{With respect to the latter we present only an illustrative sketch.  The general story will be discussed in a later work.}

Given weight $n$, set
\begin{subequations} \label{SE:m}
\begin{equation}
  m \ = \ \lfloor n/2 \rfloor \,;
\end{equation}
that is, $m$ is defined by 
\begin{equation}
  n \ \in \ \{ 2m , 2m+1 \} \,.
\end{equation}
\end{subequations}

\begin{lemma} \label{L:HT-HN}
Let $D$ be a Mumford--Tate domain for polarized Hodge structures of weight $n$ and with $\bh = (h^{n,0} , h^{n-1,1} , \ldots , h^{1,n-1} , h^{0,n})$.  If a point of $D$ admits a Hodge--Tate degeneration, then the Hodge numbers satisfy
\begin{equation}\label{E:HT-HN}
  h^{n,0} \ \le \ h^{n-1,1} \ \le \cdots \le h^{n-m,m} \,.
\end{equation}
\end{lemma}

\noindent The converse to Lemma \ref{L:HT-HN} holds when $D$ is a period domain (Theorem \ref{T:HT-PD}), but fails for Mumford--Tate domains in general (Remark \ref{R:HT-nPD}).

\begin{proof}
If there is a Hodge--Tate degeneration of a point in $D$, then the limiting mixed Hodge structure looks like
\begin{equation} \label{E:6}
\renewcommand{\arraystretch}{1.3}
\begin{array}{rcl}
  H^0(n) \ \to \ H^0(n-1)  \ \to \ H^0(n-2) & \to \cdots \to & 
  H^0(2) \ \to \ H^0(1) \ \to \ H^0_0 \\
  H^0_2(n-1)  \ \to \ H^0_2(n-2) & \to \cdots \to & 
  H^0_2(2) \ \to \ H^0_2(1) \\
  H^0_4(n-2) & \to \cdots \to & 
  H^0_4(2) \\
  & \vdots & 
\end{array}\end{equation}
where $H^0_{2k}$ is a Hodge--Tate structure of weight zero.  The inequalities \eqref{E:HT-HN} are then simple consequences of this picture: the Deligne splitting must satisfy
\begin{equation}\nonumber
  i^{n,n} \ \le \ i^{n-1,n-1} \ \le \cdots \le 
  i^{n-m,n-m} \,.
\end{equation}
Since $h^{p,q} = i^{p,n-q}$, this yields \eqref{E:HT-HN}.

For example, when $n=4$, the $I^{p,q}$ picture of the Hodge--Tate limiting mixed Hodge structure is 
\begin{center}
\setlength{\unitlength}{10pt}
\begin{picture}(30,5)
\put(0,0){\vector(0,1){5}}\put(0,0){\vector(1,0){5}}
\multiput(0,0)(1,1){5}{\circle*{0.3}}
\multiput(1,1)(1,1){3}{\circle{0.55}}
\put(2,2){\circle{0.85}}
\put(7,2){and $\ h^{4,0} = i^{4,4} \ \le \ h^{3,1} = i^{3,3} \ \le \ h^{2,2} = i^{2,2}$.}
\end{picture}
\end{center}
\end{proof}

A second condition that $D$ must satisfy to admit a Hodge--Tate degeneration is 

\begin{lemma} \label{L:JMP}
If a Mumford--Tate domain $D \subset \check D = G_\bC/P$ admits a Hodge--Tate degeneration, then it is necessarily the case that $\fp$ is an even Jacobson--Morosov parabolic \emph{(Remark \ref{R:JMP})}.
\end{lemma}

\begin{example}[Even Jacobson--Morosov parabolics in $G_2$] \label{eg:JM_G2}
The exceptional simple Lie algebra $\fg_\bC = \fg_2$ has three conjugacy classes of (proper) parabolic subalgebras.  Each may be realized as Jacobson--Morosov parabolics, but only two are even; they are indexed by $\Sigma = \{ \a_2 \}$ (the simple root $\a_2$ is long) and $\Sigma = \{ \a_1,\a_2\}$ (indexing the Borel), \cf\cite[Section 8.4]{MR1251060}.  In these two cases there exist Mumford--Tate domains $D \subset G_\bC/P$, homogeneous with respect to the split real form of $G_2$, admitting Hodge--Tate degenerations, \cf\cite[Section 6.1.3]{KP2013}.
\end{example}

\begin{proof}
Suppose that $(V,W_\sb(N),F^\sb)$ is Hodge--Tate limiting mixed Hodge structure for the period domain $D$.  Then the induced limiting mixed Hodge structure $(\fg,W_\sb(N)_\fg,F^\sb_\fg)$ is also Hodge--Tate (Proposition \ref{P:I5}).  Therefore the Lie algebra of the parabolic stabilizing $F^\sb$ is
\[
  \fp \ = \ F^0_\fg \ \stackrel{\eqref{E:delWF}}{=} 
  \ \bigoplus_{p\ge 0} I^{p,\sb}_\fg
  \ \stackrel{\eqref{E:HT}}{=} \ \bigoplus_{p \ge 0} I^{p,p}_\fg
  \ \stackrel{\eqref{E:delWF}}{=} \ W_0(N^+)_\fg
\]
where $N^+$ is the nilpositive element of an $\fsl_2$--triple \eqref{E:sl2trip} containing $N$.  It follows that $\fp$ is a Jacobson--Morosov parabolic.  Moreover, since the neutral element $Y$ acts on $I^{p,p}_\fg$ by the even scalar $2p$, it follows that $\fp$ is an even Jacobson--Morosov parabolic.
\end{proof}

\begin{remark}[Hodge--Tate degenerations in flag varieties $\check D = G_\bC/B$] \label{R:HTinG/B}
Lemma \ref{L:JMP} imposes constraints on the compact duals without reference to the real form $G_\bR$.  Once we pick a Jacobson--Morosov parabolic $P \subset G_\bC$, there are constraints on the real forms $G_\bR$ that may arise as Mumford--Tate groups with Hodge representation admitting a Hodge--Tate degeneration.  To illustrate this we consider the case that $G_\bC$ is one of the classical groups and $P = B$ is the Borel.  First a few general remarks.

The set $\cN \subset \fg_\bC$ of nilpotent endomorphisms decomposes into a finite union of $\tAd(G_\bC)$--orbits.  Moreover, there is a unique open orbit $\cN_\mathrm{prin}$, elements of which are the principal nilpotents.  A Jacobson--Morosov parabolic $W_0(N)$ is even and a Borel subalgebra if and only if $N$ is principal \cite[Section 4.1]{MR1251060}.  In the case that $G_\bC$ is one of the classical simple Lie groups, the principal nilpotents are nicely characterized by the decomposition of the standard representation under an $\fsl_2 \subset \fg_\bC$.  To be precise, given any nilpotent $N \in \cN$, let $\fsl_2 \subset \fg_\bC$ be the Lie subalgebra spanned by an $\fsl_2$--triple \eqref{E:sl2trip} containing $N$.  Then:
\begin{i_list}
\item
Let $V_\bC = \bC^n$ be the standard representation of any one of $\fsl_n\bC$, $\fsp_{2m}\bC$ with $n=2m$, or $\fso_{2m+1}\bC$ with $n = 2m+1$.  Then the nilpotent $N$ is principal if and only $V_\bC$ is irreducible as an $\fsl_2$--module.
\item
If $G_\bC = \tSO_{2m}\bC$, then $N$ is principal if and only if the standard representation $V_\bC = \bC^{2m}$ decomposes as $\bC \op \bC^{2m-1}$ (a trivial subrepresentation plus an irreducible subrepresentation).
\end{i_list}

The challenge to the real form $G_\bR$ is that:
\begin{quote}
\emph{An open $G_\bR$--orbit $D \subset G_\bC/B$ can admit the structure of a Mumford--Tate domain with a Hodge--Tate degeneration only if $\fg_\bR \cap \cN_\mathrm{prin}$ is nonempty.}
\end{quote}
From the classification \cite[Section 9.3]{MR1251060} of $\tAd(G_\bR)$--orbits of nilpotent elements in $\fg_\bR$ we deduce that $\fg_\bR \cap \cN_\mathrm{prin}$ is nonempty only for the following real forms (in the classical case):
\[
  \fsu(m,m) \,,\ \fsu(m,m+1) \,;\quad
  \fsp(n,\bR) \,;\quad
  \fso(m,m)\,,\ \fso(m+1,m) \,, \fso(m+2,m) \,.
\]
The classification also determines the open $G_\bR$--orbits $D \subset G_\bC/B$:
\begin{a_list}
\item
Suppose that $\fg_\bR = \fsp(n,\bR) = \tEnd(\bR^{2n},Q)$.  From the classification \cite[Theorem 9.3.5]{MR1251060} we may draw the following conclusions.  The intersection $\fg_\bR \cap \cN_\mathrm{prin}$ consists of two $\tAd(G_\bR)$--orbits $\cN_{\mathrm{prin},\bR}^+$ and $\cN_{\mathrm{prin},\bR}^-$.  Given $N \in \cN_{\mathrm{prin},\bR}^\pm$, there exists $v \in \bR^{2n}$ such that $\{ v , Nv , \ldots , N^{2n-1}v \}$ is a basis of $\bR^{2n}$ and 
\[
  \pm Q(N^av,N^bv) = (-1)^a \d^{a+b}_{2n-1} \,.
\]
So we see that, if $N \in \cN_{\mathrm{prin},\bR}^+$ and 
\begin{equation} \label{E:egFp}
  F^p \ = \ \tspan_\bC\{ N^a v \ : \ a \ge p \}
\end{equation}
then $(\bC^{2n},W_\sb(N),F^\sb)$ is a limiting mixed Hodge structure for the period domain parameterizing weight $2n-1$, $Q$--polarized Hodge structures with Hodge numbers $\bh = (1,1,\ldots,1)$.  Note that 
\[
  I^{p,p} \ = \ \tspan\{ N^{2n-1-p} v\} \,.
\]
\item 
In the case that $\fg_\bR = \fso(m+1,m) = \tEnd(\bR^{2m+1},Q)$, the classification \cite[Theorem 9.3.4]{MR1251060} asserts that the intersection $\fg_\bR \cap \cN_\mathrm{prin}$ consists of one $\tAd(G_\bR)$--orbit $\cN_{\mathrm{prin},\bR}$, and given an element $N$ in that orbit, $\bR^{2m+1}$ admits a basis $\{ v , Nv , \ldots , N^{2m}v\}$ such that 
\[
  (-1)^m Q( N^a v , N^b v ) \ = \ (-1)^a \d^{a+b}_{2m} \,.
\]
Taking again the definition \eqref{E:egFp}, we see that $(\bC^{2n} , W_\sb(N),F^\sb)$ is a limiting mixed Hodge structure for the period domain parametrizing weight $2m$, $(-1)^mQ$--polarized Hodge structures with Hodge numbers $\bh = (1,\ldots,1)$.
\item 
In the case that $\fg_\bR = \fso(m,m) = \tEnd(\bR^{2m},Q)$, the classification \cite[Theorem 9.3.4]{MR1251060} asserts that the intersection $\fg_\bR \cap \cN_\mathrm{prin}$ consists of two $\tAd(G_\bR)$--orbits $\cN_{\mathrm{prin},\bR}^+$ and $\cN_{\mathrm{prin},\bR}^-$.  Given $N \in \cN_{\mathrm{prin},\bR}^\pm$ there exists a $Q$--orthogonal decomposition $\bR^{2m} = \bR \op \bR^{2m-1}$ into $\fsl_2$--submodules admitting bases $\{ w \}$ and $\{ v , Nv , \ldots , N^{2m-2}\}$ such that 
\[
  \pm Q(w,w) \ = \ 1 \tand
  \pm Q(N^av,N^bv) \ = \ (-1)^{m+a}\d^{a+b}_{2m-2} \,.
\]
Let 
\begin{equation} \label{E:egFp2}
\renewcommand{\arraystretch}{1.2}
\begin{array}{rcl}
  F^p & = & \tspan\{ N^a v \ : \ a \ge p \} \,, \quad p \ge m \,,\\
  F^p & = & \tspan\{ w \,,\, N^a v \ : \ a \ge p \} \,,\quad p \le m-1 \,.
\end{array}
\end{equation}
If $m$ is even and $N \in \cN_{\mathrm{prin},\bR}^+$, then $(\bC^{2m},W_\sb(N) , F^\sb)$ is a limiting mixed Hodge structure for the period domain parametrizing weight $2m-2$, $Q$--polarized Hodge structures with Hodge numbers $\bh = (1,\ldots,1,2,1\ldots,1)$.  Note that 
\[
  I^{p,p} \ = \ \left\{ \begin{array}{ll}
    \tspan\{ N^{2m-2-p} v\} \,,\ & p \not= m-1\,,\\
    \tspan\{ w , N^{m-1}v \} \,,\ & p = m-1 \,.
  \end{array}\right.
\]
\item 
In the case that $\fg_\bR = \fso(m+2,m) = \tEnd(\bR^{2m},Q)$, the classification \cite[Theorem 9.3.4]{MR1251060} asserts that the intersection $\fg_\bR \cap \cN_\mathrm{prin}$ consists of a single $\tAd(G_\bR)$--orbit $\cN_{\mathrm{prin},\bR}$.  Given $N \in \cN_{\mathrm{prin},\bR}$ there exists a $Q$--orthogonal decomposition $\bR^{2m} = \bR \op \bR^{2m-1}$ into $\fsl_2$--submodules admitting bases $\{ w \}$ and $\{ v , Nv , \ldots , N^{2m-2}\}$ such that 
\[
  Q(w,w) \ = \ 1 \tand
  Q(N^av,N^bv) \ = \ (-1)^{m+a}\d^{a+b}_{2m} \,.
\]
Taking \eqref{E:egFp2} as above, we see that if $m$ is even, then $(\bC^{2m},W_\sb(N) , F^\sb)$ is a limiting mixed Hodge structure for the period domain parametrizing weight $2m-2$, $Q$--polarized Hodge structures with Hodge numbers $\bh = (1,\ldots,1,2,1\ldots,1)$.
\end{a_list}
\end{remark}

\subsection{Hodge--Tate degenerations in period domains}

In the event that $D$ is a period domain, the converse to Lemma \ref{L:HT-HN} holds.

\begin{theorem} \label{T:HT-PD}
Let $D$ be the period domain for polarized Hodge structures of weight $n$ and with Hodge numbers $\bh = (h^{n,0} , h^{n-1,1} , \ldots , h^{1,n-1} , h^{0,n})$.  If \eqref{E:HT-HN} holds, then a point of $D$ admits a Hodge--Tate degeneration.
\end{theorem}

\begin{proof}
This is an existence question, and we will construct a Hodge--Tate limiting mixed Hodge structure $(V,W_\sb(N),F^\sb)$.  The construction will be given as direct sum of ``atomic'' Hodge--Tate limiting mixed Hodge structures with dimensions
\[
  \mathbf{i}_{k,d} \ = \ 
  \left(i_{k,d}^{n,n} \,,\, \ldots \,,\ i^{0,0}_{k,d} \right) \ = \ 
  \big( \underbrace{0,\ldots,0}_{k \ \mathrm{terms}} \,,\, d , \ldots , d \,,\,
         \underbrace{0,\ldots,0}_{k \ \mathrm{terms}} \big) \,;
\]
the latter are illustrated in Figure \ref{f:atomic}.  
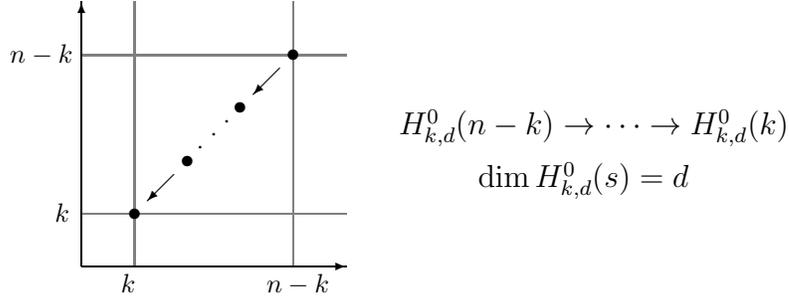
\begin{figure}[!ht]
\setlength{\unitlength}{10pt}
\caption{The atomic limiting mixed Hodge structures $(V_{k,d} , W_\sb(N_{k,d}) , F^\sb_{k,d})$}
\begin{picture}(26,12)(-3,-1)
\color{gray}
\put(2,0){\line(0,1){10}}
\put(8,0){\line(0,1){10}}
\put(0,2){\line(1,0){10}}
\put(0,8){\line(1,0){10}}
\color{black}
\put(0,0){\vector(1,0){10}}
\put(0,0){\vector(0,1){10}}
\multiput(2,2)(2,2){4}{\circle*{0.4}}
\multiput(4.5,4.5)(0.5,0.5){3}{\circle*{0.15}}
\put(7.5,7.5){\vector(-1,-1){1}}
\put(3.5,3.5){\vector(-1,-1){1}}
\begin{footnotesize}
\put(1.5,-1){$k$}
\put(7,-1){$n-k$}
\put(-1,1.7){$k$}
\put(-2.7,7.7){$n-k$}
\end{footnotesize}
\put(12,5){$H^0_{k,d}(n-k) \to \cdots \to H^0_{k,d}(k)$}
\put(15,3){$\tdim\,H^0_{k,d}(s) = d$}
\end{picture}
\label{f:atomic}
\end{figure}
To make this precise, given $0 \le d \in \bZ$ and $0 \le k \le m$, with $m$ as in \eqref{SE:m}, define
\[
  H^0_{k,d}(s) \ = \ \tspan_\bR\{ e_1^s,\ldots,e_d^s \} \ \simeq \ \bR^d \,,
  \quad k \le s \le n-k \,.
\]
Set 
\[
  V_{k,d}(\bR) \ = \ \bigoplus_{k \le s \le n-k} H^0_{k,d}(s) \,,
\]
and define a nilpotent $N_{k,d} \in \tEnd(V_{k,d},\bR)$ by 
\[
  N_{k,d}( e^s_a ) \ = \ \left\{ \begin{array}{ll}
    e^{s-1}_a \,,\ & s > k \,,\\
    0 \,,\ & s=k \,.
  \end{array} \right.
\]
Note that 
\[
  N_{k,d} \left( H^0_{k,d}(s) \right) \ = \ 
  \left\{ \begin{array}{ll}
    H^0_{k,d}(s-1) \,,\ & s > k \,,\\
    0 \,,\ & s=k \,.
  \end{array} \right.
\]

Define a $(-1)^n$--symmetric $Q_{k,d}$ on $V_{k,d}$ by 
\[
  (-1)^{n-k-s}\, Q( e_a^s , e_b^t) \ = \ \d_{ab}\,\d^{s+t}_{n} \,.
\]
Then $N_{k,d} \in \tEnd(V_{k,d},Q_{k,d})$ is nilpotent with weight filtration
\[
  W_\ell(N_{k,d}) \ = \ \tspan\{ e^s_a \ : \ 1 \le a \le d \,,\ s \le \ell \} \,.
\]  
Setting
\[
  F^p_{k,d} \ = \ \tspan\{ e^s_a \ : \ 1 \le a \le d \,, s \ge p \}
\]
defines an $\bR$--split, polarized, Hodge--Tate limiting mixed Hodge structure $(V_{k,d} , W_\sb(N_{k,d}) , F^\sb_{k,d})$, with dimensions $\mathbf{i}_{k,d}$.  

We now define $(V,W_\sb(N),F^\sb)$ as follows.  Set
\[
  d_k \ = \ \left\{
  \begin{array}{ll}
    h^{n,0} \,, \ & k=0 \,,\\
    h^{n-k,k}-h^{n-k+1,k-1} \,,\, & 1 \le k \le m \,,
  \end{array} \right.
\]
and
\begin{eqnarray*}
  V & = & \bigoplus_{0 \le k \le m} V_{k,d_k} \,,\\
  N & = & \bigoplus_{0 \le k \le m} N_{k,d_k} \,,\\
  F^p & = & \bigoplus_{0 \le k \le m} F^p_{k,d_k} \,.
\end{eqnarray*}
Then $(V,W_\sb(N),F^\sb)$ is an $\bR$--split, polarized, Hodge--Tate limiting mixed Hodge structure with dimensions
\[
  \mathbf{i} \ = \ \sum \mathbf{i}_{k,d_k} \ = \ \bh \,.
\]  
Therefore $e^{zN}\cdot F^\sb \in D$ for all $\tIm(z) > 0$.
\end{proof}

\begin{remark} \label{R:HT-nPD}
Theorem \ref{T:HT-PD} is false for for Mumford--Tate domains in general.  That is, the inequalities \eqref{E:HT-HN} do not imply the existence of a Hodge--Tate degeneration.  To see this suppose that $(V,Q)$ and $(\tilde V , \tilde Q)$ are two Hodge representations of $G$ realizing $G_\bR/R$ as a Mumford--Tate domain (Section \ref{S:relD}); denote the two (isomorphic) realizations by $D$ and $\tilde D$.
\begin{a_list}
\item
Any limiting mixed Hodge structure $(V,W_\sb(N),F^\sb)$ on $(V,Q)$ determines a limiting mixed Hodge structure $(\tilde V , \tilde W_\sb(N),\tilde F^\sb)$ on $(\tilde V,\tilde Q)$.  Moreover, these two limiting mixed Hodge structures induce the same limiting mixed Hodge structure $(\fg,W_\sb(N)_\fg,F^\sb_\fg)$.
\item 
The reduced limit period mapping $\Phi_\infty$ sends $(F^\sb,N)$ to the closed orbit if and only if it sends $(\tilde F^\sb,N)$ to the closed orbit.
\item
By Proposition \ref{P:I5} the limiting mixed Hodge structure $(V,W_\sb(N),F^\sb)$ is Hodge--Tate if and only if $(\tilde V,\tilde W_\sb(N),\tilde F^\sb)$ is Hodge--Tate.
\item 
Suppose that \eqref{E:HT-HN} fails for $(\tilde V,\tilde Q)$.  Then $\tilde D$, as the Mumford--Tate domain for $(\tilde V,\tilde Q)$, contains \emph{no} point admitting a Hodge--Tate degeneration.  Therefore, even if \eqref{E:HT-HN} holds for $(V,Q)$, the Mumford--Tate domain $D$ can not have a point admitting a Hodge--Tate degeneration.  So to disprove the theorem if suffices to exhibit a pair $(V,Q)$ and $(\tilde V,\tilde Q)$ such that \eqref{E:HT-HN} holds for one but not the other.
\end{a_list}

\noindent Here is a counter--example.

\begin{example}
Fix a symmetric nondegenerate bilinear form $Q$ on $\bR^5$ of signature $(1,4)$.  Set $G_\bR = \tSO(1,4) = \tAut(\bR^5,Q)$.  Then $G_\bC = \tSO(5,\bC)$.  Let the compact dual $\check D = G_\bC/P = \tGr^Q(2,\bC^5)$ be the variety of $Q$--isotropic 2--planes in $\bC^5$, and let $D$ be the period domain parameterizing $Q$--polarized Hodge structures on $\bC^5$ with Hodge numbers $\bh = (2,1,2)$.  Note that \eqref{E:HT-HN} fails.  On the other hand the induced $Q_\fg$--polarized Hodge structure on $\fg_\bR$ has Hodge numbers $\bh_\fg = ( 1,2,4,2,1 )$; in this case \eqref{E:HT-HN} holds.
\end{example}
\end{remark}

\subsection{Non--Hodge--Tate degenerations in period domains} \label{S:nHT-PD}

We recall and prove

\begin{thmI8}
Let $D$ be a period domain parameterizing weight $n$ Hodge structures.  If there exists a limiting mixed Hodge structure $(V,W_\sb(N),F^\sb)$ that maps to the closed $G_\bR$--orbit in $\check D$, but is \emph{not} of Hodge--Tate type, then $n = 2m$ is even and:
\begin{a_list_emph}
\item 
For $k\not=0$, $\tGr^{W_\sb(N)}_{n+k,\tprim}$ is of Hodge--Tate type.  
\emph{(Thus $k$ is even.)}
\item
For $k\not=0$, $\tGr^{W_\sb(N)}_{n+k,\tprim} \not=0$ implies $k \equiv 2$ (mod) $4$.
\item
$\tGr^{W_\sb(N)}_{n,\tprim} \not=0$, and the only nonzero $I^{p,q}_\tprim$, with $p+q=n$, are 
\[
  I^{m+1,m-1}_\tprim \tand
  I^{m-1,m+1}_\tprim \,.
\] 
\end{a_list_emph}
\end{thmI8}

\noindent The proof of Theorem \ref{T:I8} proceeds in four steps.

\subsubsection*{Step 1: odd weight} 

For odd weight $n=2m+1$, we have $G_\bR = \tSp(2g,\bR)$.  This real form is $\bR$--split; whence the closed orbit $\cO_\mathrm{cl}$ is totally real.  It follows from Theorem \ref{T:I6} that, if a degeneration to $\cO_\mathrm{cl}$ exists, then the limiting mixed Hodge structure must be Hodge--Tate.  So from this point on we assume that 
\[
  n \ = \ 2m \,.
\]

\subsubsection*{Step 2: preliminaries for even weight}  

By Theorem \ref{T:I7} we are given the form of the induced limiting mixed Hodge structure $(\fg , W_\sb(N)_\fg,F^\sb_\fg)$.  The issue to is extract from the inclusion
\[
  \fg \ \subset \ \tEnd(V,Q)
\]
the form of original limiting mixed Hodge structure $(V,W_\sb(N),F^\sb)$.  So, the question is: \emph{given the Deligne splitting for $(\fg , W_\sb(N)_\fg,F^\sb_\fg)$, what can we infer about the Deligne splitting for $(V,W_\sb(N),F^\sb)$?}  An especially convenient way to do this is to use the relationship between the weights of $V_\bC$ and the roots of $\fg_\bC = \tEnd(V_\bC,Q)$; this will amount to comparing eigenvalues between the two spaces.

Define $L \in \fg_\bC$ by 
\[
  \left.L\right|_{I^{p,q}_\fg} = q \,\one \,.
\]
Then \eqref{E:Ivg} implies the Deligne splitting is the $(L,\overline L)$--eigenspace decomposition of $\fg_\bC$,
\[
  \left. (L,\overline L) \right|_{I^{p,q}_\fg} \ = \ (q,p)\,\one \,.
\]
Likewise, in analogy with the discussion of Sections \ref{S:ind1} and \ref{S:ind2}, the Deligne splitting of $V_\bC$ is an $(L,\overline L)$--eigenspace decomposition.  To be precise, $(L,\overline L)$ acts on $I^{p+m,q+m}$ by the scalars $(q,p)$.  It will be convenient to shift the bigrading by 
\[
  \tilde I^{p,q} \ = \ I^{p+m,q+m} \,,
  \quad\hbox{so that}\quad
  \left. (L,\overline L) \right|_{\tilde I^{p,q}} \ = \ (q,p) \,\one\,.
\]
Note that, like $\op I^{p,q}_\fg$, the splitting $\op \tilde I^{p,q}$ is symmetric about the line $p+q=0$.

Let $\fh$ be a Cartan subalgebra containing the grading elements $L$ and $\overline L$.  We may assume without loss of generality that $\fh$ is closed under conjugation.  Let $\Lambda(V) \subset \fh^*$ denote the weights of the standard representation $V_\bC$, and given $\lambda \in \Lambda(V)$, let $V^\lambda \subset V_\bC$ denote the weight space.  Then the $(L,\overline L)$--eigenvalues of $V_\bC$ are $\{ (L(\lambda),\overline L(\lambda) \ : \ \lambda \in \Lambda(V) \}$; equivalently,
\begin{subequations} \label{SE:wts+rts}
\begin{equation} \label{E:wts}
  \tilde I^{p,q} \ = \ 
  \bigoplus_{\mystack{L(\lambda)=q}{\overline L(\lambda)=p}} V^\lambda \,.
\end{equation}
Likewise, the $(L,\overline L)$--eigenvalues of $\fg_\bC$ are $\{ (\a(L),\a(\overline L)) \ : \ \a \in \Delta\,\cup\,\{0\} \}$; that is, 
\begin{equation} \label{E:rts}
  I^{p,q}_\fg \ = \ 
  \bigoplus_{\mystack{\a(L)=q}{\a(\overline L)=p}} \fg^\a
  \ \Big( \ \op \ \fh \quad \hbox{if } \ (p,q) = (0,0) \Big) \,. 
\end{equation}
The relationship between the roots of $\fg_\bC$ and the weights $\Lambda(V)$ of $V_\bC$ is 
\begin{equation} \label{E:wtsVrts}
  \Delta \,\cup\, \{0\} \ = \ \{ \a = \lambda+\mu \ : \ 
  \lambda\not=\m \in \Lambda(V) \} \,.
\end{equation}
\end{subequations}
This yields relationships between the $(L,\overline L)$--eigenvalues on $V_\bC$ and $\fg_\bC$.  In particular, \eqref{SE:wts+rts} implies the following:
\begin{equation} \label{E:Ipq}
\begin{array}{l}
  \hbox{Suppose $\tilde I^{p,q} , \tilde I^{r,s}\not=0$ and $(p,q)\not=(r,s)$.}\\
  \hbox{Then $I^{p+q,r+s}_\fg \not=0$.}
\end{array}
\end{equation}

\begin{remark}[Properties of $\tilde I^{p,q}$] \label{R:tI}
In the arguments that follow it will be helpful to keep in mind that both $\tilde I^{p,q}$ and $I^{p,q}_\fg$ are symmetric about the $p+q=0$ and $p-q=0$ lines: if $\tilde I^{p,q}$ is nonzero, then so are $\tilde I^{q,p}$, $\tilde I^{-q,-p}$ and $\tilde I^{-p,-q}$.  (Similarly for $I^{p,q}_\fg$.)  In fact they all have the same dimension.  The symmetry about the line $p-q=0$ is due to the fact that both the roots $\Delta$ and the weights $\Lambda(V)$ are closed under conjugation (because $\fh$, $\fg$ and $V$ are defined over $\bR$).  The symmetry about the line $p+q=0$ is due to the fact that the $N$--strings are uninterrupted and centered on the line $p+q=0$.  This also implies the following: suppose that $\tilde I^{p,q}$ is nonzero.  Then $\tilde I^{k,-k}$ is nonzero if $p-q = 2k$, and $\tilde I^{k+1,-k}$ and $\tilde I^{k,-k-1}$ are nonzero if $p-q = 2k+1$.
\end{remark}

\subsubsection*{Step 3: Suppose there exists $\tilde I^{p,q}\not=0$ with $p-q = 2k+1$}

Then $\tilde I^{k+1,-k} , \tilde I^{k,-k-1} \not=0$ (Remark \ref{R:tI}).  So \eqref{E:Ipq} implies $I^{2b+1,-2b-1}_\fg = 0$, and \eqref{E:cp3} forces $c=\pm1$.  So, suppose that $\tilde I^{\pm 1,0}$ and $\tilde I^{0,\pm1}$ are all nonzero.  (As discussed in Remark \ref{R:tI}, if any one of the four is nonzero, then all four are nonzero.)  If $\tilde I^{r,s}\not=0$, then \eqref{E:Ipq} implies 
\[
  I^{r\pm1,s}_\fg \tand I^{r,s\pm1}_\fg \quad\hbox{are nonzero.}
\]
Given \eqref{E:cp2}, this forces $|r-s| \le 1$.  Whence the Deligne splitting must be of the form depicted in Figure \ref{f:Del1}.  But this implies that \eqref{E:HT-HN} holds.  Whence Theorem \ref{T:HT-PD} implies the limiting mixed Hodge structure is Hodge--Tate, contradicting our hypothesis.  To conclude:
\begin{center}
\emph{If $\tilde I^{p,q} \not=0$, then $p-q$ is even.}
\end{center}
\begin{figure}[!ht]
\caption{The Deligne splitting $V_\bC = \op I^{p,q}$.}
\setlength{\unitlength}{10pt}
\begin{picture}(10,10)(0,0)
\put(0,0){\vector(1,0){9}}
\put(0,0){\vector(0,1){9}}
\multiput(0,0)(1,1){9}{\circle*{0.4}}
\multiput(0,1)(1,1){8}{\circle*{0.4}}
\multiput(1,0)(1,1){8}{\circle*{0.4}}
\end{picture}
\label{f:Del1}
\end{figure}

\subsubsection*{Step 4: Suppose there exists $\tilde I^{p,q}\not=0$ with $p-q = 2k>0$}

Then $\tilde I^{k,-k}$ and $\tilde I^{-k,k}$ are nonzero (Remark \ref{R:tI}).  Since $N\not=0$, there exists some nonzero $\tilde I^{r,s}$ with $r+s \not=0$.  Then \eqref{E:Ipq} implies $I^{r-k,s+k}_\fg$ and $I^{r+k,s-k}_\fg$ are nonzero, and \eqref{E:cp2} forces $|r-s\pm2k| \le 2$.  By hypothesis $2k\ge2$, so it must be the case that $k=1$ and $r-s=0$.  Thus, the Deligne splitting $V_\bC = \op I^{p,q}$ is as depicted in Figure \ref{f:Del2}.a.  (The induced Deligne splitting $\fg_\bC = \op I^{p,q}_\fg$ is as in Figure \ref{f:Del2}.b.)  This establishes Theorem \ref{T:I8}(a).
\begin{figure}[!ht]
\caption{The Deligne splittings $V_\bC = \op I^{p,q}$ and $\fg_\bC = \op I^{p,q}_\fg$.}
\setlength{\unitlength}{10pt}
\begin{picture}(10,12)(0,-2)
\put(0,0){\vector(1,0){9}}
\put(0,0){\vector(0,1){9}}
\multiput(0,0)(1,1){9}{\circle*{0.4}}
\put(3,5){\circle*{0.4}}
\put(5,3){\circle*{0.4}}
\put(1,-1.5){Figure \ref{f:Del2}.a}
\end{picture}
\hspace{20pt}
\setlength{\unitlength}{9pt}
\begin{picture}(10,13)(-5,-7.5)
\put(-5,0){\vector(1,0){10}}
\put(0,-5){\vector(0,1){10}}
\multiput(-5,-5)(1,1){11}{\circle*{0.4}}
\multiput(-3,-5)(1,1){9}{\circle*{0.4}}
\multiput(-5,-3)(1,1){9}{\circle*{0.4}}
\put(-2,2){\circle*{0.4}}
\put(2,-2){\circle*{0.4}}
\put(-3.5,-7){Figure \ref{f:Del2}.b}
\end{picture}
\label{f:Del2}
\end{figure}

The hypothesis that $(V,W_\sb(N),F^\sb)$ is not Hodge--Tate implies
\begin{equation} \label{E:I11}
  \tilde I^{-1,1} = \tilde I^{-1,1}_\tprim \ \not= \ 0 \,.
\end{equation}
Moreover, if $\tilde I^{p,p}_\tprim \not=0$, then \eqref{E:Ipq} and \eqref{E:I11} imply that $I^{p-1,p+1}_{\fg,\tprim} \not=0$.  By Theorem \ref{T:cp_orb}, this forces $p \ge 1$ to be odd.  This establishes Theorem \ref{T:I8}(b).  Finally, since $\tilde I^{-1,1} = I^{m-1,m+1}$, we see that \eqref{E:I11} establishes Theorem \ref{T:I8}(c) and completes the proof.  

\subsection{All degenerations are induced from maximal Hodge--Tate degenerations}

In a suitably interpreted sense all degenerations are induced from a (maximal) degeneration of Hodge--Tate type.  Some care must be taken with this statement, as it is not necessarily the case that the underlying degeneration arises algebro--geometrically: this is a statement about the orbit structure and representation theory associated with the $\tSL_2$--orbit approximating an arbitrary degeneration, which may or may not arise algebro--geometrically.

Fix a Mumford--Tate domain $D \subset \check D$.  Let $N \in \fg_\bR$ be a nilpotent element and consider the corresponding boundary component $B(N) = \tilde B(N) / \texp(\bC N)$.  Given $F^\sb \in \tilde B(N)$, let $(F^\sb , W(N)_\sb)$ denote the corresponding limiting mixed Hodge structure on $\fg$.  Let $\fg_\bC = \op I^{p,q}_\fg$ be the Deligne splitting.  Without loss of generality, the limiting mixed Hodge structure is $\bR$--split.  The diagonal subalgebra
\[
  \fs_\bC \ = \ \bigoplus_p I^{p,p}_\fg
\]
is a conjugation stable subalgebra of $\fg_\bC$ containing $N$.  Let $\fs_\bR = \fs_\bC \cap \fg_\bR$ denote the real form.  Moreover, as the zero eigenspace for the grading element $L - \overline L$, the subalgebra $\fs_\bR$ is necessarily a Levi subalgebra, and therefore reductive. 

\begin{lemma}
The limiting mixed Hodge structure $(F^\sb_\fg,W_\sb(N)_\fg)$ on $\fg_\bR$ induces a sub--limiting mixed Hodge structure $(F^\sb_\fs , W_\sb(N)_\fs)$ on $\fs_\bR$ by 
\[
  F^\sb_\fs \ = \ F^\sb \,\cap\, \fs_\bC 
  \ = \ \bigoplus_{q\ge p} I^{q,q}_\fg \tand
  W_\sb(N)_\fs \ = \ W_\sb(N)_\fg \,\cap\, \fs_\bR 
  \ = \ \bigoplus_{q \le p} I^{q,q}_{\fg,\bR} \,.
\] 
\end{lemma}

\begin{proof}
This follows directly from the definition of limiting mixed Hodge structures (Section \ref{S:NO+MHS}).
\end{proof}

Let $S_\bC \subset G_\bC$ be the connected Lie subgroup with Lie algebra $\fs_\bC$, and set
\[
  \check \sD \ = \ S_\bC \cdot F^\sb \tand \sD \ = \ \check \sD \,\cap\, D \,.
\]
By [CKS], $F^\sb_\fg(z) = e^{zN} F^\sb_\fg \in D$ for all $\tIm(z)>0$; equivalently, the Hodge filtration $F^\sb_\fg(z)$ defines a Hodge structure $\varphi_z$ on $\fg_\bR$ (Section \ref{S:dfn_phs}).  Likewise, $F^\sb_\fs(z) = e^{zN}F^\sb_\fs \in \sD$ defines a Hodge structure $\left.\varphi_z\right|_\fs$ on $\fs_\bR$ for all $\tIm(z) > 0$.  This implies
\begin{quote}
  \emph{$\sD$ carries the structure of a Mumford--Tate subdomain of $D$ with 
  compact dual $\check \sD$.}
\end{quote}
In particular, $\sD$ is an open $S_\bR$--orbit in $\check\sD$.  From Proposition \ref{P:I9}, and the fact that $(F^\sb_\fs,W_\sb(N)_\fs)$ is Hodge--Tate, we see that 
\begin{center}
  \emph{$S_\bR \cdot F^\sb$ is the closed $S_\bR$--orbit in $\check\sD$.}
\end{center}
In this sense, 
\begin{quote}
\emph{the nilpotent orbit $(F^\sb_\fs,N)$ is a maximal degeneration of the Hodge structure on $\fs_\bR$.}
\end{quote}

So far we have viewed $\fs_\bR$ as having sub--Hodge structures $\left.\varphi_z\right|_\fs$ that are restrictions of Hodge structures $\varphi_z$ on $\fg_\bR$.  In fact a stronger statement holds: the circle $\varphi_z$ is contained in $S_\bR$.  

\begin{lemma}
The Hodge structure $(\fg,\varphi_z)$ is given by a Hodge representation of $S_\bR$.
\end{lemma}  

\noindent In this sense, 
\begin{quote}
\emph{the degeneration of Hodge structure on $\fg$ given by $(F^\sb_\fg,N)$ is induced from a maximal degeneration of Hodge structure on $\fs_\bR$.}
\end{quote}

\begin{proof}
We need to show that the circle $\varphi_z$ is contained in $S_\bC$, \cf\cite{MR2918237}.  The corresponding grading element $L_z$ (Sections \ref{S:PHSgrelem1} and \ref{S:PHSgrelem2}) is equal to $\varphi'(1)/4\pi\bi$, \cf\cite[Section 2.3]{schubVHS}.  So $\varphi_z \subset S_\bR$ if and only if $L_z$, a priori an element of $\fg_\bC$, is an element of $\fs_\bC$.  

Decompose $\fs_\bC = \fz_\bC \op \fs_\bC^\tss$ into its center $\fz_\bC$ and semisimple factor $\fs_\bC^\tss = [\fs_\bC,\fs_\bC]$.  Since $\fs_\bR$ is a sub--Hodge structure of $\fg_\bR$, with respect to $\varphi_z$, it follows that $\fs_\bR^\tss$ is also a sub--Hodge structure.  Therefore, $L_z$ determines a graded decomposition of $\fs_\bC^\tss$.  As discussed in Remark \ref{R:grelem}, this graded decomposition is also induced by a grading element $L'_z \in \fs_\bC$.  It follows that $L_z - L_z' \in \fg_\bC$ is contained in the centralizer of $\fs_\bC$.  It is here that the fact that $\fs_\bC$ is a Levi subalgebra is key, for it is a well--known property of Levi subalgebras that the centralizer in $\fg_\bC$ is equal to the center $\fz_\bC$.  Therefore, $L_z - L_z' \in \fz_\bC \subset \fs_\bC$.  Whence $L_z \in \fs_\bC$.
\end{proof}

\begin{remark}
Since $\fs$ is reductive, $\fg=\fs \op \fs^\perp$ as a $\fs$--module.  The idea here is that the essential structure/relationship is between $N$ and the Levi subalgebra $\fs$; the remaining structure on $\fg = \fs \op \fs^\perp$, that is the structure on $\fs^\perp$, is induced from the $\fs$--module structure on $\fs^\perp$.  This sort of idea does back to Bala and Carter's classification \cite{MR0417306, MR0417307} of nilpotent orbits $\cN \subset \fg_\bC$, where the idea is to look at minimal Levi subalgebras $\fl$ containing a fixed $N \in \cN$, and to classify the pairs $(N,\fl)$.  (In fact, the idea goes back farther to Dynkin \cite{MR0047629_trans}, who looked at minimal reductive subalgebras containing $N$, but this approach does not seem to work as well.)
\end{remark}





\hspace{5pt}

\bigskip

\def\cprime{$'$} \def\Dbar{\leavevmode\lower.6ex\hbox to 0pt{\hskip-.23ex
  \accent"16\hss}D}

\end{document}